\documentclass{imsart}
\usepackage{latexsym,epsfig,amssymb,amsmath,amsthm,color,url,bbm,pdfsync}
\usepackage{graphicx,float,subfigure,bm}
\usepackage{amsfonts}
\usepackage{graphics}
\usepackage{mathrsfs}

\numberwithin{equation}{section}
\allowdisplaybreaks
\usepackage[round]{natbib}
\usepackage{hyperref}

\newtheorem{lemma}{Lemma}[section]

\newtheorem{theorem}[lemma]{Theorem}
\newtheorem{proposition}[lemma]{Proposition}
\newtheorem{definition}[lemma]{Definition}
\newtheorem{corollary}[lemma]{Corollary}
\newtheorem{example}[lemma]{Example}
\newtheorem{exercise}[lemma]{Exercise}
\newtheorem{remark}[lemma]{Remark}

\newcommand{\cov}{{\rm cov}}
\newcommand{\var}{{\rm Var}}
\newcommand{\cid}{\stackrel{d}{\rightarrow}}
\newcommand{\cip}{\stackrel{p}{\rightarrow}}
\newcommand{\ciw}{\stackrel{w}{\rightarrow}}
\newcommand{\civ}{\stackrel{v}{\rightarrow}}
\newcommand{\stas}{\stackrel{\rm a.s.}{\rightarrow}}
\newcommand{\ex}{{\rm e}\,}
\newcommand{\bbs}{{\mathbb S}}
\newcommand{\bbr}{{\mathbb R}}
\newcommand{\bfx}{{\bf x}}
\newcommand{\bfX}{{\bf X}}
\newcommand{\bfY}{{\bf Y}}
\newcommand{\E}{{\mathbb E}}
\renewcommand{\P}{{\mathbb P}}
\newcommand{\bT}{{\mathbf\Theta}}
\newcommand{\beao}{\begin{eqnarray*}}
\newcommand{\eeao}{\end{eqnarray*}\noindent}
\newcommand{\beam}{\begin{eqnarray}}
\newcommand{\eeam}{\end{eqnarray}\noindent}
\newcommand{\beqq}{\begin{equation}}
\newcommand{\eeqq}{\end{equation}\noindent}
\newcommand{\beqo}{\begin{equation*}}
\newcommand{\eeqo}{\end{equation*}\noindent}
\newcommand{\bth}{\begin{theorem}}
\newcommand{\ethe}{\end{theorem}}

\newcommand{\bre}{\begin{remark}\em }
\newcommand{\ere}{\end{remark}}

\newcommand{\ble}{\begin{lemma}}
\newcommand{\ele}{\end{lemma}}

\newcommand{\bde}{\begin{definition}}
\newcommand{\ede}{\end{definition}}
\newcommand{\bco}{\begin{corollary}}
\newcommand{\eco}{\end{corollary}}
\newcommand{\bpr}{\begin{proposition}}
\newcommand{\epr}{\end{proposition}}
\newcommand{\bpf}{\begin{proof}}
\newcommand{\epf}{\end{proof}}
\newcommand{\bexer}{\begin{exercise}}
\newcommand{\eexer}{\end{exercise}}
\newcommand{\bexam}{\begin{example}\rm }
\newcommand{\eexam}{\end{example}}

\newcommand{\revise}[1]{{\color{black} #1}}

\begin{document}

%--------------------------------------------  TITLE AND MISC  --------------------------------------------%

%\bibliographystyle{alpha}
%\today

\begin{frontmatter}
\title{\Large Threshold Selection for Multivariate Heavy-Tailed Data}
\runtitle{Threshold Selection for Multivariate Heavy-Tailed Data}
\begin{aug}
\author{Phyllis Wan}\ead[label=e1]{phyllis@stat.columbia.edu}\and
\author{Richard A. Davis}
\ead[label=e2]{rdavis@stat.columbia.edu}
\address{Department of Statistics\\
Columbia University\\
1255 Amsterdam Avenue, MC 4690\\
New York, NY 10027\\
\printead{e1,e2}}

\runauthor{Wan and Davis}

%\thankstext{T1}{The authors of research is supported in part by ARO MURI grant W911NF-12-1-0385. }

\end{aug}

\begin{keyword}
\kwd{distance covariance}
\kwd{heavy-tailed data}
\kwd{multivariate regular variation}
\kwd{threshold selection}
\end{keyword}

\begin{abstract}
Regular variation is often used as the starting point for modeling multivariate heavy-tailed data. A random vector is regularly varying if and only if its radial part $R$ is regularly varying and is asymptotically independent of the angular part $\Theta$ as $R$ goes to infinity. The conditional limiting distribution of $\Theta$ given $R$ is large characterizes the tail dependence of the random vector and hence its estimation is the primary goal of applications. A typical strategy is to look at the angular components of the data for which the radial parts exceed some threshold. While a large class of methods has been proposed to model the angular distribution from these exceedances, the choice of threshold has been scarcely discussed in the literature. In this paper, we describe a procedure for choosing the threshold by formally testing the independence of $R$ and $\Theta$ using a measure of dependence called distance covariance. We generalize the limit theorem for distance covariance to our unique setting and propose an algorithm which selects the threshold for $R$. This algorithm incorporates a subsampling scheme that is also applicable to weakly dependent data. Moreover, it avoids the heavy computation in the calculation of the distance covariance, a typical limitation for this measure. The performance of our method is illustrated on both simulated and real data.

\end{abstract}

\end{frontmatter}
%\subjclass{Primary 62M10; Secondary 60E10 60F05 60G10 62H15 62G20}
\maketitle
\bigskip

%---------------------------------------------------------------------------------------------------------------------------------------------%
%--------------------------------------------  INTRODUCTION  --------------------------------------------%
%---------------------------------------------------------------------------------------------------------------------------------------------%

\section{Introduction}\label{sec:intro}

For multivariate heavy-tailed data, the principal objective is often to study dependence in the `tail' of the distribution. To achieve this goal, the assumption of multivariate regular variation is typically used as a starting point. A random vector $\bfX\in\bbr^d$  is said to be {\em multivariate regularly varying} 
%(see, e.g., Chapter 6 of Resnick (2007) \cite{resnick:2007}), 
if the polar coordinates $(R,\mathbf\Theta) = (\|\bfX\|,\bfX/\|\bfX\|)$, where $\|\cdot\|$ is some norm, {satisfy} the conditions
\begin{enumerate}
	\item[(a)]
		$R$ is univariate regularly varying, i.e., $\P(R>r) = L(r) r^{-\alpha}$, where $L(\cdot)$ is a slowly varying function at infinity;
	\item[(b)]
		$\P(\mathbf\Theta\in\cdot|R>r)$ converges weakly to a measure $S(\cdot)$ as $r\to\infty$.
\end{enumerate}
The $\alpha$ is referred to as the index of the regular variation, while the $S$ is called the angular distribution and characterizes the limiting tail dependence. There are other equivalent definitions of regular variation \citep{resnick:2002}, but this one is the most convenient for our purposes.

Given observations $\{\bfX_i\}_{i=1}^n$ and their corresponding polar coordinates $\{(R_i,\mathbf\Theta_i)\}_{i=1}^n$, a straightforward procedure for estimating $S$ is to look at angular components of the data for which the radii are greater than a large threshold $r_0$, that is, $\mathbf\Theta_i$ for which $R_i>r_0$. In most studies, one takes $r_0$ to be a large empirical quantile of $R$. While there has been extensive research on choosing a threshold for which the distribution of $R$ is regularly varying (i.e., limit condition (a)), little research has been devoted to ensuring the threshold is large enough for the independence of $\Theta$ and $R$ to be reasonable (i.e., limit condition (b)). To this end, \cite{dehaan:deronde:1998} fit a parametric extreme value distribution model to each marginal and examined the parameter stability plot of each coordinate. The St\v{a}ric\v{a} plot \citep{starica:1999} looked at the joint tail empirical measure, but was, in some way, equivalent to only examining the extremal behavior of $R$. {\cite{resnick:2007} suggested an automatic threshold selection from the St\v{a}ric\v{a} plot but observed that the thresholds were sometimes systematically underestimated.} In their study of the threshold based inference for parametric max-stable processes, \cite{jeon:smith:2014} suggested choosing the threshold by minimizing the MSE of the estimated parameters. 

In this paper, we propose an algorithm which selects the threshold for modeling $S$. Our motivation is the implied property that $R$ and $\mathbf\Theta$ given $R>r$ become independent as $r \to \infty$. Given a sequence of candidate threshold levels, we test the degree of dependence between $R$ and $\mathbf\Theta$ for the truncated data above each level. The dependence measure we use is the distance covariance introduced by \cite{szekely:rizzo:bakirov:2007}. This measure has the ability to account for various types of dependence and to be applicable to data in higher dimensions. The resulting test statistics are given in the form of $p$-values and are compared across all levels through a subsampling scheme. This enables us to extract more accurate information from the test statistics while not overloading the computational burden.

The remainder of this paper is organized as follows. We first provide some theoretical backgrounds on multivariate regular variation in Section \ref{sec:mrv}. The distance covariance and its theoretical properties are introduced in Section \ref{sec:dcor}. Applying this dependence measure in our conditioning setting, we propose a test statistic and prove relevant theoretical results in Section \ref{sec:theory}. Our proposed algorithm for threshold selection is presented in Section \ref{sec:method}, and illustrated on simulated and real examples in {Section} \ref{sec:examples}. The paper concludes with a discussion.

% The standard estimate is based on Hill's estimator, which averages over the percentage of the data which exceed a large threshold. This threshold should be chosen large enough, such that the assumption of regular variation should hold for the data in the tail, but not too large, in order to have a sufficiently large amount of observations to make reasonable estimates of the parameter. There has been abundant literatures on choosing such a threshold. Typically one looks at the Hill plot \cite{hill:1975} and decides on a threshold for which the Hill estimator looks constant. Other automatic procedures include \cite{clauset:shalizi:newman:2009}, which minimizes the distance between the empirical distribution function and the approximating Pareto distribution for the upper tail of the distribution.

%**********************************************************************************

%---------------------------------------------------------------------------------------------------------------------------------------------%
%--------------------------------------------  MRV AND SET-UP  --------------------------------------------%
%---------------------------------------------------------------------------------------------------------------------------------------------%

\section{Multivariate regular variation and problem set-up} \label{sec:mrv}

One way to approach multivariate heavy-tailed data is through the notion of multivariate regular variation. For a detailed review, see, for example, Chapter 6 of \cite{resnick:2007}. Let $\mathbf{X}=(X_1,\ldots,X_d)$ be a $d$-dimensional random variable defined on the cone ${\bbr}^d_+ = {[\mathbf{0},\mathbf{\infty})}\backslash\{\mathbf{0}\} $. 
%It is convenient to standardize the marginal distributions of $\bfX$ to be regularly varying with the same index, for example, through applying a rank transform, see, e.g.~Joe et al.\cite{joe:smith:weissman:1992}. 
Define the polar coordinate transformation \beqq \label{eq:polar}
	T(\mathbf X) = (\|\mathbf{X}\|,\mathbf{X}/\|\mathbf{X}\|) =: (R,\mathbf\Theta),
\eeqq
{where $\|\cdot\|$ denotes some norm.}
Then $\bfX$ is regularly varying if and only if there exists a probability measure $S(\cdot)$ on $\mathbb{S}^{d-1}$, the unit sphere in $\mathbb{R}^d$, and a function $b(t)\to\infty$, such that
\beqq \label{eq:2}
	t\P\left[\left(R/b(t),\mathbf\Theta\right) \in \cdot \right] \overset{v}{\to} \nu_\alpha \times S, \quad  t \to \infty,
\eeqq
where $\overset{v}{\to}$ denotes vague convergence, and $\nu_\alpha$ is a measure defined on $(0,\infty]$ such that
$$
	\nu_\alpha(x,\infty] = x^{-\alpha},\quad  x>0.
$$
Here $b(t)$ can be chosen as the $1-t^{-1}$-quantile, i.e.,
$$
%	\P[R>b(t)] \sim \frac{1}{t}.
	b(t) = \inf\{s|\P (R\le s) \ge 1-t^{-1}\}.
$$
The convergence \eqref{eq:2} implies that
\begin{equation} \label{eq:3}
	\P\left[\left(\frac{R}{r},\mathbf\Theta\right) \in \cdot \middle| R>r \right] \ciw \nu_\alpha \times S, \quad r \to \infty,\quad \mbox{ on } [1,\infty) \times \bbs^{d-1},
\end{equation}
where $\ciw$ denotes weak convergence.
In other words, given that $R>r$ for $r$ large, the conditional distribution of $R/r$ and $\mathbf\Theta$ are {independent in the limit}. {In view of \eqref{eq:3}, we restrict the measure $\nu_\alpha$ to $[1,\infty)$ throughout the remainder of the paper.} The angular measure $S$ characterizes the tail dependence structure of $\bfX$. If $S$ is concentrated on $\{e_i,i=1,\ldots d\}$, where $e_i=(0,\ldots,0,1,0,\ldots,0)$, then the components of $\bfX$ are asymptotically independent in the tail, a case known as asymptotic independence. If $S$ has mass lying in the subspace {$\{(t_1,\ldots,t_d)\in\bbs^{d-1}|t_i>0,t_j>0,i\ne j\}$}, then an extreme observation in the $X_i$ direction implicates a positive probability of an extreme observation in the $X_j$ direction, a case known as asymptotic dependence. Hence the estimation of $S$ from observations is an important problem, and often the primary goal, in multivariate heavy-tailed modeling.

The following convergence is implied from \eqref{eq:3}:
\beqq \label{eq:ang:dist}
	\P(\mathbf\Theta \in \cdot|R>r) \ciw S(\cdot),\quad r \to \infty.
\eeqq
This suggests estimating $S$ using the angular data $(\mathbf\Theta_i)$ whose radial parts satisfy $R_i>r_0$ for $r_0$ large. The motivation behind our method is to seek $r_0$ such that when $R>r_0$, $R$ and $\mathbf\Theta$ are virtually independent. Given a candidate threshold sequence $\{r_k\}$, we formally test the independence between $(R_i,\mathbf\Theta_i)$ on the index set $\{i|R_i>r_k\}$. The use of Pearson's correlation as the dependence measure is unsuitable in this case, for two reasons. First, correlation is only applicable to univariate random variables, whereas $\mathbf\Theta$ lies on the sphere of dimension $d-1$. Second, correlation only describes the linear relationship between two random variables, thus having zero correlation is not a sufficient condition for independence. Instead, we use a more powerful dependence measure, the {\it distance covariance}, which is introduced in the next section.

%In practice, it is convenient to choose $r_n$ to be an order statistic or quantile from the observations, \cite{caperaa:fougere:2000} showed that the convergences are asymptotically equivalent.

%---------------------------------------------------------------------------------------------------------------------------------------------%
%--------------------------------------------  DISTANCE COVARIANCE  --------------------------------------------%
%---------------------------------------------------------------------------------------------------------------------------------------------%

\section{Distance covariance}\label{sec:dcor}

In this section, we briefly review the definition and some properties of the distance covariance. More detailed descriptions and proofs can be found in \cite{szekely:rizzo:bakirov:2007} and \cite{davis:matsui:mikosch:wan:2016}.

Let $X\in\bbr^p$ and $Y\in\bbr^q$ be two random vectors, then the distance covariance between $X$ and $Y$ is defined as
\beqq \label{eq:dcov}
	T(X,Y;\mu) = \int_{\bbr^{p+q}} \big|\varphi_{X,Y}(s,t)-\varphi_X(s)\,\varphi_Y(t)\big|^2\,\mu(ds,dt)\,,\,\quad (s,t)\in\bbr^{p+q}
\eeqq
where $\varphi_{X,Y}(s,t),\varphi_X(s),\varphi_Y(t)$ denote the joint and marginal characteristic functions of $(X,Y)$ and $\mu$ is a {suitable} measure on $\bbr^{p+q}$. {In order to ensure that $T(X,Y;\mu)$ is well-defined, one of the following conditions is assumed to be satisfied throughout the paper \citep{davis:matsui:mikosch:wan:2016}:
\vspace{-.05in}
\begin{enumerate}
\item
	$\mu$ is a finite measure on $\bbr^{p+q}$;
\item
	$\mu$ is an infinite measure on $\bbr^{p+q}$ such that
	$$
		\int_{\bbr^{p+q}} (1\wedge|s|^\alpha) (1\wedge|t|^\alpha) \mu(ds,dt) < \infty
	$$
	and
	$$
		\E[|XY|^\alpha+|X|^\alpha+|Y|^\alpha] < \infty
	$$
	for some $\alpha \in (0,2]$.
	\vspace{-.02in}
\end{enumerate}}
One advantage of distance covariance over, say, Pearson's covariance, is that, if  $\mu$ has a positive Lebesgue density on $\bbr^{p+q}$, then $X$ and $Y$ are independent if and only if $T(X,Y;\mu)=0$. Another attractive property of this dependence measure is that it readily applies to random vectors of different dimensions.

To estimate $T(X,Y;\mu)$ from observations $(X_1,Y_1),\ldots,(X_n,Y_n)$, define the empirical distance covariance
\beqo
	T_n(X,Y;\mu) = \int_{\bbr^{p+q}} \big|\hat\varphi_{X,Y}(s,t)-\hat\varphi_X(s)\,\hat\varphi_Y(t)\big|^2\,\mu(ds,dt)\,,
\eeqo
where $\hat\varphi_{X,Y}(s,t)= \dfrac 1 n \sum_{j=1}^n \ex^{i\,\langle s,X_j\rangle+i\,\langle t,Y_j\rangle}$ and $\hat\varphi_X(s)=\hat\varphi_{X,Y}(s,0),\hat\varphi_Y(t)=\hat\varphi_{X,Y}(0,t)$ are the respective empirical characteristic functions. 
If we assume that $\mu = \mu_1\times\mu_2$ and is symmetric about the origin, {then under the conditions where $T(X,Y;\mu)$ exists,} $T_n(X,Y;\mu)$ also has the computable form
\beao
	T_n(X,Y;\mu) 
		&=& \frac{1}{n^2}\,\sum_{i,j=1}^n\,\tilde \mu_1 (X_i-X_j)\,\tilde\mu_2(Y_i-Y_j) \nonumber\\
	 	&&+\, \frac{1}{n^4}\, \sum_{i,j,k,l=1}^n\,\tilde\mu_1(X_i-X_j) \tilde \mu_2(Y_k-Y_l)\, -\,\frac{2}{n^3}\,\sum_{i,j,k=1}^n\,\tilde \mu_1(X_i-X_j) \tilde \mu_2(Y_i-Y_k), \label{eq:edovalt}
\eeao
where $\tilde \mu(x)= \int(1-\cos \langle s,x\rangle) \,\mu(ds)$ \citep{davis:matsui:mikosch:wan:2016}.

The most popular choice of $\mu$, first mentioned by \cite{feuerverger:1993} and then more extensively studied by \cite{szekely:rizzo:bakirov:2007}, is
\beqq \label{eq:dcormeas}
	\mu(ds,dt)=c_{p,q}|s|^{-\kappa-p}|t|^{-\kappa-q}ds\,dt\,.
\eeqq
where $c_{p,q}$ is as defined in Lemma 1 of \cite{szekely:rizzo:bakirov:2007}. This choice of $\mu$ gives $\tilde\mu(x) \tilde\mu(y) = |x|^\kappa|y|^\kappa$. Moreover, this is the only choice of $\mu$ for which the distance covariance is invariant relative to scale and orthogonal transformations. Note that in order for the integral \eqref{eq:dcov} to exist, it is required that 
\beqo \label{eq:dcor:moment}
	E[|X|^\kappa|Y|^\kappa+|X|^\kappa+|Y|^\kappa] < \infty.
\eeqo
We will utilize the described weight measure \eqref{eq:dcormeas} with $\kappa=1$ in our simulations and data analyses in Section \ref{sec:examples}, but applied to the log transformation on $R$ to ensure that the moment condition is satisfied.

As detailed in \cite{davis:matsui:mikosch:wan:2016}, if the sequence $\{(X_i,Y_i)\}$ is stationary and ergodic, then
\beqq\label{eq:dcor:consist}
	T_n(X,Y;\mu) \stas T(X,Y;\mu).
\eeqq
Further, if $X$ and $Y$ are independent, then under an $\alpha$-mixing condition,
\beqq\label{eq:dcor:asymp}
	n\,T_n(X,Y;\mu) \cid \int_{\bbr^{p+q}}  |G(s,t)|^2\mu(s,t)
\eeqq
for some centered Gaussian field $G$. On the other hand, if $X$ and $Y$ are dependent, then
$$
	\sqrt{n}(T_n(X,Y;\mu) -T(X,Y;\mu)) \cid G'_\mu
$$
for some non-trivial limit $G'_\mu$, implying that $n\,T_n(X,Y;\mu)$ diverges as $n\to\infty$. Naturally one can devise a test of independence between $X$ and $Y$ using the statistic $n\,T_n(X,Y;\mu)$: the null hypothesis of independence is rejected at level $\chi$ if $n\,T_n(X,Y;\mu) > c$, where $c$ is the upper $\chi$-quantile of $\int_{\bbr^{p+q}} |G(s,t)|^2\mu(s,t)$. 

In practice, the distribution $\int_{\bbr^{p+q}} |G(s,t)|^2\mu(s,t)$ is intractable and is typically approximated through bootstrap. Hence the main drawback of using distance covariance is the computation burden it brings for large sample size: the computation of a single distance covariance statistic requires $O(n^2)$ operations, while finding the cut-off values via resampling requires much more additional computation. Our method, however, overcomes this problem through subsampling the data, as will be described in Section \ref{sec:method}.

% Analogous to linear covariance and correlation, one can define the normalized {\em distance correlation}
%$$
%R(X,Y;\mu)  = \frac{T(X,Y;\mu)}{(T(X,X;\mu)T(Y,Y;\mu))^{1/2}} \in [0,1],
%$$
%where two end points indicate independence and complete dependence (up to a certain transformation) respectively.

%---------------------------------------------------------------------------------------------------------------------------------------------%
%--------------------------------------------  THEORETICAL RESULTS  --------------------------------------------%
%---------------------------------------------------------------------------------------------------------------------------------------------%

\section{Theoretical results}\label{sec:theory}

Let $\{\bfX_i\}_{i=1}^n$ be iid observations in $\bbr^d$ from a multivariate regularly varying distribution $\mathbf{X}$ satisfying \eqref{eq:polar} and \eqref{eq:3}, and $\{(R_i,\mathbf\Theta_i)\}_{i=1}^n$ be their polar coordinate transformations. Given a threshold $r_n$, we measure the dependence between ${R/r_n}$ and $\bT$ conditional on $R>r_n$ by the empirical distance covariance of $({R/r_n},\bT)$ on the set $\{R>r_n\}$:
\beqq \label{eq:teststat}
	T_n := \int_{\bbr^{d+1}}  |C_n(s,t)|^2 \mu(ds,dt),
\eeqq
with 
$$C_n(s,t) := \hat\varphi_{{\frac{R}{r_n}},\Theta|r_n}(s,t) - \hat\varphi_{{\frac{R}{r_n}}|r_n}(s)\hat\varphi_{\Theta|r_n}(t),$$
where $\hat\varphi_{{\frac{R}{r_n}},\Theta|r_n}$ is the conditional empirical characteristic function of $({R/r_n},\bT)$,
\beqo
	\hat\varphi_{{\frac{R}{r_n}},\Theta|r_n}(s,t) = \frac{1}{\sum_{j=1}^n \mathbf{1}_{\{R_j>r_n\}}}\sum_{j=1}^n e^{isR_j/r_n + it^T\mathbf\Theta_j}\mathbf{1}_{\{R_j>r_n\}},\ s \in \bbr,\, \ t=(t_1,\ldots,t_d)^T \in \bbr^{d},
\eeqo
and $\hat\varphi_{{\frac{R}{r_n}}|r_n}, \hat\varphi_{\Theta|r_n}$ are the corresponding empirical conditional marginal characteristic functions,
$$
	\hat\varphi_{{\frac{R}{r_n}}|r_n}(s) = \hat\varphi_{{\frac{R}{r_n}},\Theta|r_n}(s,0),\,\quad\,
	\hat\varphi_{\Theta|r_n}(t)=\hat\varphi_{{\frac{R}{r_n}},\Theta|r_n}(0,t).
$$

In this section, we establish the limiting results \eqref{eq:dcor:consist} and \eqref{eq:dcor:asymp} adapted to the conditional distance covariance.
For ease of notation, let
$$
p_n := \P(R>r_n)\ , \ \quad \ \hat p_n := \frac{1}{n}  \sum_{j=1}^n \mathbf{1}_{\{R_j>r_n\}}
$$
be the theoretical and empirical probability of exceedance, and let
$$
\varphi_{{\frac{R}{r_n}},\Theta|r_n}(s,t) := \E\left[ e^{isR/r_n + it^T\mathbf\Theta} | R>r_n\right] = \frac{\E\left[ e^{isR/r_n + it^T\mathbf\Theta} \mathbf{1}_{R>r_n}\right] }{p_n},
$$
and 
$$
\varphi_{{\frac{R}{r_n}}|r_n}(s) := \varphi_{{\frac{R}{r_n}},\Theta|r_n}(s,0),\,\quad\,\varphi_{\Theta|r_n}(t):=\varphi_{{\frac{R}{r_n}},\Theta|r_n}(0,t),
$$
be the theoretical conditional joint and marginal characteristic functions.

{Recall from \eqref{eq:3} that as $n\to\infty$, $R/r_n$ and $\mathbf\Theta$ become asymptotically independent and converge to $\nu_\alpha$ and $S$ respectively. Denote the characteristic functions of the corresponding limit distributions by }
\beam
	\varphi_R(s) &:=& {\int_1^\infty} \exp(is r) \alpha r^{-\alpha-1}dr = \lim_{n\to\infty} \varphi_{{\frac{R}{r_n}}|r_n}(s), \label{eq:phi:r} \\
	\varphi_\Theta(t) &:=& {\int_{\bbs^{d-1}}}  \exp(it\theta) S(d\theta) = \lim_{n\to\infty} \varphi_{\Theta|r_n}(t).\label{eq:phi:theta}
\eeam
{We have the following results.}

%--------------------------------------------  THM  --------------------------------------------%

\begin{theorem} \label{thm:1}

\begin{enumerate}
\item
	Let $\mathbf{X}_1, \ldots,\mathbf{X}_n$ be iid observations generated from $\mathbf{X}$, {where $\mathbf{X}$ is multivariate regularly varying with index $\alpha>1$}.  Let $T_n$ be the conditional {empirical} distance covariance between the angular and radial component defined in \eqref{eq:teststat}. Further assume that $np_n \to \infty$ and the weight measure $\mu$ satisfies 
	\beqq \label{eq:weight}
		\int_{\bbr^{d+1}} (1\wedge |s|^{\beta})(1\wedge |t|^{2}) \mu(ds,dt) <\infty,
	\eeqq
	 for some {$1<\beta<2\wedge\alpha$}. Then
	\beqo
		T_n \cip 0.
	\eeqo
	
\item
	In addition, if $\{r_n\}$ satisfies
	\beqq \label{eq:cond}
		np_n \int_{\bbr^{d+1}} |\varphi_{{\frac{R}{r_n}},\Theta|r_n}(s,t)- \varphi_{{\frac{R}{r_n}}|r_n}(s)\varphi_{\Theta|r_n}(t)|^2\mu(ds,dt) \to 0,
	\eeqq
	then
	\beqq \label{eq:thm2}
		n\hat p_nT_n \cid \int_{\bbr^{d+1}} |Q(s,t)|^2\mu(ds,dt),
	\eeqq
	{where $Q$ is a centered Gaussian process with covariance function
	\beqq \label{eq:Qcov}
		\cov(Q(s,t),Q(s',t')) = (\varphi_R(s-s') - \varphi_R(s)\varphi_R(-s'))(\varphi_\Theta(t-t') - \varphi_\Theta(t)\varphi_\Theta(-t'))
	\eeqq
	with $\varphi_R,\varphi_\Theta$ as defined in \eqref{eq:phi:r} and \eqref{eq:phi:theta}.}

\end{enumerate}
\end{theorem}

%We first prove the theorem, and then provide a sufficient condition of for assumption \eqref{eq:cond} (see Remark~\ref{rm:sec:order}). 

%--------------------------------------------  REMARK - MOMENT  --------------------------------------------%

\begin{remark}

In the case where $\mathbf{X}$ is regularly varying with index $\alpha\le1$, similar results hold if we replace $R/r_n$ with $\log(R/r_n)$ {for which all moments exist.}

\end{remark}

%--------------------------------------------  REMARK - 2ND ORDER  --------------------------------------------%

The proof of the theorem is delayed to Appendix~\ref{app:1}. {In the following remark, we discuss certain sufficient conditions for assumption \eqref{eq:cond}.}

\begin{remark}\label{rm:sec:order}

{Assume that} $\mu = \mu_1 \times \mu_2$, where $\mu_1,\mu_2$ are measures on $\bbr$ and $\bbr^d$, respectively, and symmetric about the origin. From Section~2.2 of \cite{davis:matsui:mikosch:wan:2016}, condition \eqref{eq:cond} is equivalent to
{
\beam
&&np_n\,\left(\E[\tilde \mu_1 (\frac{R}{r_n}-\frac{R'}{r_n})\,\tilde\mu_2(\mathbf\Theta - \mathbf\Theta')|R,R'>r_n] \right.\qquad\nonumber\\
&&\qquad\quad\left.+\E[\tilde \mu_1(\frac{R}{r_n}-\frac{R'}{r_n})|R,R'>r_n]\,\E[ \tilde \mu_2(\mathbf\Theta - \mathbf\Theta')|R,R'>r_n] \right.\nonumber\\
&&\qquad \qquad \left.- 2\,\E[\tilde\mu_1(\frac{R}{r_n}-\frac{R'}{r_n})
 \tilde \mu_2(\mathbf\Theta - \mathbf\Theta'')|R,R'>r_n]\right) \,\to\, 0, \label{eq:conv:1}
\eeam
where
\beao
\tilde \mu_i(x)= \int(1-\cos (x^Ts)) \,\mu_i(ds),\,\quad\, i=1,2.
\eeao
Let $P_{\frac{R}{r_n},\Theta|r_n}$ denote the conditional joint distribution of $(R/r_n,\Theta)$ given $R>r_n$ and $P_{\frac{R}{r_n}|r_n}, P_{\Theta|r_n}$ be the respective conditional marginals. Then \eqref{eq:conv:1} can be expressed as
\beam
&&np_n \int_{(1,\infty)\times\bbs^{d-1}}\int_{(1,\infty)\times\bbs^{d-1}} \tilde \mu_1 (T-T')\,\tilde\mu_2(\mathbf\Theta - \mathbf\Theta')\nonumber \\
&& \quad \left(P_{\frac{R}{r_n},\Theta|r_n}(dT,d\mathbf\Theta) P_{\frac{R}{r_n},\Theta|r_n}(dT',d\mathbf\Theta')  \right. \nonumber\\
&&  \qquad \left. +P_{\frac{R}{r_n}|r_n}(dT)P_{\Theta|r_n}(d\mathbf\Theta)P_{\frac{R}{r_n}|r_n}(dT')P_{\Theta|r_n}(d\mathbf\Theta') \right. \nonumber\\
&& \qquad \quad \left.- 2 P_{\frac{R}{r_n},\Theta|r_n}(dT,d\mathbf\Theta) P_{\frac{R}{r_n}|r_n}(dT')P_{\Theta|r_n}(d\mathbf\Theta')  \right)\nonumber\\
&=& \int_{(1,\infty)\times\bbs^{d-1}}\int_{(1,\infty)\times\bbs^{d-1}} \tilde \mu_1 (T-T') \,\tilde\mu_2(\mathbf\Theta - \mathbf\Theta') \,\nonumber\\
&& \quad \sqrt{np_n} \left(P_{\frac{R}{r_n},\Theta|r_n}(dT,d\mathbf\Theta) - P_{\frac{R}{r_n}|r_n}(dT)P_{\Theta|r_n}(d\mathbf\Theta)\right) \nonumber\\
&&  \qquad  \sqrt{np_n} \left( P_{\frac{R}{r_n},\Theta|r_n}(dT',d\mathbf\Theta')-P_{\frac{R}{r_n}|r_n}(dT')P_{\Theta|r_n}(d\mathbf\Theta')  \right) \,\to\, 0,\label{eq:conv:cond} 
\eeam
where
$(R', \mathbf\Theta')$, $(R'', \mathbf\Theta'')$ are iid copies of $(R,\mathbf\Theta)$.
}
One way to verify {\eqref{eq:conv:cond}} is to assume a second-order like condition on the distribution of $(R,\mathbf\Theta)$. For example, assume that
\beqo \label{eq:sec:order}
\frac{P_{{\frac{R}{r_n}},\Theta|r_n} - \nu_\alpha \times S}{A(r_n)} \ciw \chi,
\eeqo
where $\chi$ is a signed measure such that $\chi([r,\infty]\times B)$ is finite for all $r\ge1$ and $B$ Borel set in $\bbs^{d-1}$, the unit sphere in $\bbr^d$, and the scalar function $A(t)\to 0$ as $t\to\infty$. {When the components of $\mathbf{X}$ are asymptotically independent, this is equivalent to the second order condition for multivariate regular variation \citep{resnick:2002}.} If we choose the sequence $r_n$ such that $\sqrt{np_n} \to\infty$ and $\sqrt{np_n} A(r_n) \to 0$, then
\beao
&&\sqrt{np_n}A(r_n)\, \frac{P_{{\frac{R}{r_n}},\Theta|r_n}((\cdot,\cdot)) - P_{{\frac{R}{r_n}}|r_n}(\cdot)\times P_{\Theta|r_n}(\cdot)}{A(r_n)} \\
&=& \sqrt{np_n}A(r_n)\, \left( 
\frac{P_{{\frac{R}{r_n}},\Theta|r_n}((\cdot,\cdot)) - \nu_\alpha \times S((\cdot,\cdot))}{A(r_n)}
\right. - 
\frac{{\left(P_{{\frac{R}{r_n}}|r_n}(\cdot) - \nu_\alpha(\cdot)\right)}\times {P_{\Theta|r_n}(\cdot)}}{A(r_n)}  \\
&& \quad
\left. -\frac{\nu_\alpha(\cdot) \times (P_{\Theta|r_n}(\cdot) - S(\cdot))}{A(r_n)} \right) \,\ciw\, 0
\eeao
on $[1,\infty] \times \bbs^{d-1}$.
{In the case where $\mu_1,\mu_2$ are finite measures, $\tilde\mu_1,\tilde\mu_2$ are bounded and \eqref{eq:conv:cond} is satisfied since the integrand can be written as}
\beao
&&np_nA^2(r_n)\,\int_{(1,\infty)\times\bbs^{d-1}}\left[\int_{(1,\infty)\times\bbs^{d-1}}\tilde \mu_1 (T-T')\,\tilde\mu_2(\mathbf\Theta - \mathbf\Theta')  \  \right.\\
&& \left.\qquad \frac{{P_{\frac{R}{r_n},\Theta|r_n}(dT,d\mathbf\Theta)-P_{\frac{R}{r_n}|r_n}(dT)P_{\Theta|r_n}(d\mathbf\Theta)}}{A(r_n)}\right] \, \\
&& \qquad \qquad\frac{{P_{\frac{R}{r_n},\Theta|r_n}(dT',d\mathbf\Theta')-P_{\frac{R}{r_n}|r_n}(dT')P_{\Theta|r_n}(d\mathbf\Theta')}}{A(r_n)}\ \to\  0.
\eeao
In the special case that $|A| \in RV_\rho$ for $\rho < 0$, {\eqref{eq:conv:cond} is met provided $r_n$ is} chosen such that
$$
O(n^{\frac{1}{\alpha+2{|\rho|}}+\epsilon}) \le r_n \le o(n^{\frac{1}{\alpha}}), \quad \text{for some $\epsilon>0$.}
$$

\end{remark}

{When the measures $\mu_1,\mu_2$ are infinite, \eqref{eq:cond} can be verified in specific cases.  This is illustrated in the following example.}

\bexam \label{ex:bilogistic}
%(Coles and Tawn, 1991, Modelling Extreme Multivariate Events)
Let $\mathbf{X}$ follow a bivariate logistic distribution, {i.e.,} $\mathbf{X}$ has cdf
\beqq \label{eq:logistic}
\P(X_1<x_1,X_2<x_2) = \exp(-(x_1^{-1/\gamma}+x_2^{-1/\gamma})^\gamma),\,\quad\, \gamma \in (0,1{)}.
\eeqq
{Then $\mathbf{X}$ has asymptotically independent components} if and only if $\gamma=1$.  It can be shown that $\mathbf{X}$ is regularly varying with index $\alpha=1$, i.e., $p_n = \P(R>r_n)\sim r_n^{-1}$ as $r_n\to\infty$. {Using the $L_1$-norm, $\|(x_1,x_2)\|=|x_1|+|x_2|$,} the pseudo-polar coordinate transform is $(R,\Theta) = (X_1+X_2,X_1/(X_1+X_2)) {\in (0,\infty) \times [0,1]}$ and the pdf of $(R,\Theta)$ is
\beam
f_{R,\Theta}(r,\theta) &=& r^{-2} \left(\theta(1-\theta)\right)^{-\frac{\gamma+1}{\gamma}}\left(\theta^{-\frac{1}{\gamma}}+(1-\theta)^{-\frac{1}{\gamma}}\right)^{\gamma-2}\, e^{-r^{-1}\left(\theta^{-\frac{1}{\gamma}}+(1-\theta)^{-\frac{1}{\gamma}}\right)^\gamma}\nonumber\\
&&\quad  \left(r^{-1}\left(\theta^{-\frac{1}{\gamma}}+(1-\theta)^{-\frac{1}{\gamma}}\right)^\gamma-\frac{\gamma-1}{\gamma}\right). \nonumber
%\\
%&\sim& r^{-2}\, \left(-\frac{\gamma-1}{\gamma}(\theta(1-\theta))^{-\frac{\gamma+1}{\gamma}}\left(\theta^{-\frac{1}{\gamma}}+(1-\theta)^{-\frac{1}{\gamma}}\right)^{\gamma-2} \right), \quad {\text{as $r\to\infty$,}}\nonumber\\
%&=:& f_R(r)f_\Theta(\theta). \label{eq:bilog:angmeas}
\eeam
We now consider the case of the infinite weight measure $\mu$ given in \eqref{eq:dcormeas} with $\kappa=1$ and derive the condition on the sequence $\{r_k\}$ for which the conditions of Theorem~\ref{thm:1} hold. First observe that
\beam
 f_{\frac{R}{r_n},\Theta|r_n}(t,\theta) &=& t^{-2} \left(\theta(1-\theta)\right)^{-\frac{\gamma+1}{\gamma}}\left(\theta^{-\frac{1}{\gamma}}+(1-\theta)^{-\frac{1}{\gamma}}\right)^{\gamma-2}\, e^{-r_n^{-1}t^{-1}\left(\theta^{-\frac{1}{\gamma}}+(1-\theta)^{-\frac{1}{\gamma}}\right)^\gamma}\nonumber\\
&&\quad  \left(r_n^{-1}t^{-1}\left(\theta^{-\frac{1}{\gamma}}+(1-\theta)^{-\frac{1}{\gamma}}\right)^\gamma-\frac{\gamma-1}{\gamma}\right) \nonumber\\
 &\to&  t^{-2}\frac{1-\gamma}{\gamma}\left(\theta(1-\theta)\right)^{-\frac{\gamma+1}{\gamma}}\left(\theta^{-\frac{1}{\gamma}}+(1-\theta)^{-\frac{1}{\gamma}}\right)^{\gamma-2}, \quad \text{as $n\to\infty$,}  \label{eq:bilog:angmeas}\\
 &=:& f_T(t)f_\Theta(\theta), \nonumber
\eeam
and
\beao
	&&r_n\left|f_{\frac{R}{r_n},\Theta|r_n}(t,\theta)-f_T(t)f_\Theta(\theta)\right| \\
 	&\le& f_T(t)f_\Theta(\theta)\left(r_n \left| e^{-r_n^{-1}t^{-1}\left(\theta^{-\frac{1}{\gamma}}+(1-\theta)^{-\frac{1}{\gamma}}\right)^\gamma}-1\right| \right. \\
	&&\quad \left.+  e^{-r_n^{-1}t^{-1}\left(\theta^{-\frac{1}{\gamma}}+(1-\theta)^{-\frac{1}{\gamma}}\right)^\gamma}t^{-1}\left(\theta^{-\frac{1}{\gamma}}+(1-\theta)^{-\frac{1}{\gamma}}\right)^\gamma \frac{\gamma}{1-\gamma}\right) \\
	&\le& f_T(t)f_\Theta(\theta)\left(t^{-1}\left(\theta^{-\frac{1}{\gamma}}+(1-\theta)^{-\frac{1}{\gamma}}\right)^\gamma+ t^{-1}\left(\theta^{-\frac{1}{\gamma}}+(1-\theta)^{-\frac{1}{\gamma}}\right)^\gamma \frac{\gamma}{1-\gamma}\right) \\
	&\le& t^{-3} \left(\theta^{-\frac{1}{\gamma}}+(1-\theta)^{-\frac{1}{\gamma}}\right)^{2\gamma-2} \frac{1}{1-\gamma} \\
	&\le& ct^{-3}, \quad \text{for $t\ge1$ and $\theta\in[0,1]$},
\eeao
where $c$ denotes a generic a constant whose value may change from line to line throughout the paper, and the last inequality comes from the facts that
$$
	\theta(1-\theta) \le \frac{1}{4} \quad\text{and}\quad \left(\theta^{-\frac{1}{\gamma}}+(1-\theta)^{-\frac{1}{\gamma}}\right)^{2\gamma-2}\le\left(\frac12\right)^{\frac{2-2\gamma}{\gamma}}<\infty.
$$
Letting
\beao
 h_n(t,\theta)&:=&\frac{f_{\frac{R}{r_n},\Theta|r_n}(t,\theta)-f_{\frac{R}{r_n}|r_n}(t)f_{\Theta|r_n}(\theta)}{r_n^{-1}},
\eeao
we have
\beao
	&&\hspace{-.5in}\max\left(\int_0^1\int_1^\infty |h_n(t,\theta)|dtd\theta, \int_0^1\int_1^\infty |\log(t)h_n(t,\theta)|dtd\theta \right) \\
	&\le&\int_0^1\int_1^\infty |th_n(t,\theta)|dtd\theta \\
	&\le& \int_0^1\int_1^\infty \left|\frac{f_{\frac{R}{r_n},\Theta|r_n}(t,\theta)-f_T(t)f_\Theta(\theta)}{t^{-1}r_n^{-1}}\right|dtd\theta \\
	&& \quad+ \int_0^1\int_1^\infty\left|f_T(t)\frac{f_{\Theta|r_n}(\theta)-f_\Theta(\theta)}{t^{-1}r_n^{-1}}\right|dtd\theta \\
	&&\qquad +\int_0^1\int_1^\infty \left|f_{\Theta|r_n}(\theta)\frac{f_{\frac{R}{r_n}|r_n}(t)-f_T(t)}{t^{-1}r_n^{-1}}\right| dtd\theta,
\eeao
where the first term can be bounded by
\beao
	 \int_0^1\int_1^\infty ct^{-2}dtd\theta <\infty,
\eeao
and the other terms can be bounded in the same way.
Since $R$ has infinite first moment, we apply the distance correlation to $\log R$ and $\mathbf\Theta$. The integral in \eqref{eq:conv:cond} is bounded by
\beao
	&&\frac{np_n}{r_n^2}\int_0^1\int_1^\infty\int_0^1\int_1^\infty |\log t-\log t'||\theta-\theta'| |h_n(t,\theta)||h_n(t',\theta')| dtd\theta dt'd\theta' \\
	&\le& c\frac{n}{r_n^3}\int_0^1\int_1^\infty\int_0^1\int_1^\infty (|\log t|+|\log t'|) |h_n(t,\theta)||h_n(t',\theta')| dtd\theta dt'd\theta' \\		
	&\le&c\frac{n}{r_n^3}\left(\int_0^1\int_1^\infty |\log(t) h_n(t,\theta)|dtd\theta\right) \left(\int_0^1\int_1^\infty|h_n(t,\theta)| dtd\theta\right) \, \le \, c\frac{n}{r_n^3},
\eeao
which converges to zero if $n=o(r_n^3)$.  Therefore if $\{r_n\}$ is chosen such that $r_n=o(n)$ and $n=o(r_n^3)$, then Theorem~\ref{thm:1} holds.

%Hence
%\beao
%f_{R,\Theta}(r,\theta) - f_R(r)f_\Theta(\theta) \sim  O(r^{-3}),
%\eeao
%and for any $B$ Borel in {[0,1]} and $r\ge r_n$,
%\beao
%P_{R,\Theta|r_n}((r,\infty]\times B) - \nu_1(r/r_n,\infty]  S(B)&=& \frac{\int_{rr_n}^\infty\int_B(f_{R,\Theta}(r,\theta) - f_R(r)f_\Theta(\theta)) drd\theta }{r_n^{-1}}\\
%&=&O(r_n^{-1}).
%\eeao
%Therefore \eqref{eq:sec:order} holds for $A(r_n) = r_n^{-1}$ and the conditions $np_n \to \infty$ and $\sqrt{np_n}A(r_n) \to 0$ can be satisfied with the choice of

\eexam

%--------------------------------------------  THM 2  --------------------------------------------%

The result in Theorem~\ref{thm:1} can be generalized from iid to a regularly varying time series setting, which we present in the next theorem. For a multivariate stationary time series $\{\bfX_t\}$ and $h\ge1$, set $\bfY_h = (\bfX_0,\ldots,\bfX_h)$.  Then $\{\bfX_t\}$ is regularly varying if
$$
	\frac{\P(x^{-1}\bfY_h \in \cdot)}{\P(x^{-1}\|\bfX_0\| >1)} \overset{v}\to \mu^*_h(\cdot), \quad x\to\infty,
$$
for some non-null measure $\mu^*_h$ on $\overline{\bbr}^{(h+1)d}_0=\overline{\bbr}^{(h+1)d}\backslash \{\bf0\}$, $\overline{\bbr} = \bbr \cup \{\pm\infty\}$, with the property that $\mu^*_h(tC)=t^{-\alpha}\mu^*_h(C)$ for any $t>0$ and Borel set $C\subset \overline{\bbr}^{(h+1)d}_0$.  See, for example, page 979 of \cite{davis:mikosch:2009}.  It follows easily that
\beqq \label{eq:mrv:conv}
	\frac{\P(x^{-1}(\bfX_0,\bfX_h)\in\cdot)}{\P(\|\bfX_0\| >x)} \overset{v}\to \mu_h(\cdot),
\eeqq
where
$$
	\mu_h(D) = C\cdot \mu_h^* (\{\mathbf{s} \in \overline{\bbr}^{(h+1)d}: (\mathbf{s}_1,\mathbf{s}_h) \in D\}).
$$
Assume that $\{\bfX_t\}$ is $\alpha$-mixing.  We assume the following conditions between $\{\bfX_t\}$ and the sequence of threshold $\{r_n\}$, which can be verified for various time series models \citep{davis:mikosch:2009}.
\begin{enumerate}
\item[]
\begin{enumerate}
\item[(\bf M)] \label{cond:m}
Assume $\revise{p_n^{-1}=\P^{-1}(\|\bfX_1\|>r_n)}=o(n^{1/3})$ and that there exists a sequence $\{l_n\}$ such that $l_n\to\infty$, $l_np_n\to0$, and\\
%i)
%	\beqq \label{eq:lm:cond:0}
%		\sum_{h=1}^\infty \mu_h(\left\{(\bfx,\bfx')|\|\bfx\|>1,\|\bfx'\|>1\right\}) < \infty;
%	\eeqq	
i)
	\beqq \label{eq:lm:cond:1}
		\revise{\left(\frac1{p_n}\right)^\delta \sum_{h=l_n}^\infty \alpha_h^\delta \to 0 \text{ for some $\delta\in(0,1)$;}}
	\eeqq 
ii)
	\beqq \label{eq:lm:cond:2}
		\lim_{h\to\infty} \limsup_{n\to\infty} \frac{1}{p_n} \sum_{j=h}^{l_n} \P(\|\bfX_0\|>r_n, \|\bfX_j\|>r_n) =0;
	\eeqq 
iii) 
	\beqq \label{eq:lm:cond:3}
		np_n \alpha_{l_n} \to 0.
	\eeqq
\end{enumerate}
\end{enumerate}
%This condition is parallel with the condition (M) in \cite{davis:mikosch:2009} and can be verified in various time series models.

\begin{theorem} \label{thm:2}

Let $\{\mathbf{X}_t\}$ be a multivariate regularly varying time series with tail index $\alpha>1$ and $\alpha$-mixing with coefficients $\{\alpha_h\}_{h\ge0}$. 
Assume the same conditions for the weight measure $\mu$ and the sequence of thresholds $\{r_n\}$ in Theorem~\ref{thm:1}, i.e., \eqref{eq:weight}, \eqref{eq:cond} hold,  and that condition~\hyperref[cond:m]{({\bf M})} holds.  Then
	\beqo \label{eq:depthm}
		n\hat p_nT_n \cid \int_{\bbr^{d+1}} |Q'(s,t)|^2\mu(ds,dt),
	\eeqo
where $Q'$ is a centered Gaussian process.  In particular,
	\beqo
		T_n \cip 0.
	\eeqo
	
\end{theorem}

The proof of Theorem~\ref{thm:2} is given in Appendix~\ref{app:2}.

Note that the limiting distributions $Q$ in Theorem~\ref{thm:1} and $Q'$ in Theorem~\ref{thm:2} are both intractable.  In practice, quantiles of the distributions are calculated using resampling methods. While in the iid case this can be done straightforwardly, in the weakly dependent case one needs to apply the block bootstrap or stationary bootstrap to obtain the desired result (see \cite{davis:mikosch:cribben:2012}).  In the following section, we present a threshold selection framework with a subsampling scheme that does not require independence between the observations.

%---------------------------------------------------------------------------------------------------------------------------------------------%
%--------------------------------------------  METHODOLOGY --------------------------------------------%
%---------------------------------------------------------------------------------------------------------------------------------------------%

\section{Threshold selection} \label{sec:method}

In this section, we propose a procedure to select the threshold for estimating the spectral measure $S$ from observations $\mathbf{X}_1,\cdots,\bfX_n$. 
%Although the theory we showed in the previous section assumed that data were iid, we shall see that for our method we may relax the assumption such that $(\bfX_i)$ needs only to be stationary and ergodic.
Let us first consider the case where a specific threshold $r_n$ is given. Then \eqref{eq:teststat} specifies the empirical distance covariance between $R/r_n$ and $\mathbf\Theta$ conditional on $R>r_n$. Under the assumption \eqref{eq:cond}, we have from Theorem \ref{thm:1},
$$
	n\hat p_n T_n \to \int_{\bbr^{d+1}} |Q{(s,t)}|^2\mu({ds,dt}),
$$
where $n\hat p_n$ is the number of observations such that $R_i>r_n$. In practice, the limit distribution $\int|Q|^2\mu(s,t)$ is intractable, but one can resort to bootstrapping.
%Note that the same limit distribution does the empirical distance covariance of independent pair $(\tilde{R},\tilde{\mathbf\Theta})$ converges, where $(\tilde{R},\tilde{\mathbf\Theta})$ is the limit for $(R/r_n,\mathbf\Theta)|_{R>r_n}$ as $r_n\to\infty$. 
Consider the hypothesis testing framework:
\beao
H_0 &:& \mbox{ $R/r_n$ and $\mathbf\Theta$ are independent given $R>r_n$ } \\
H_1 &:& \mbox{ $R/r_n$ and $\mathbf\Theta$ are not independent given $R>r_n$.}
\eeao
Define the $p$-value for testing $H_0$ versus $H_1$ to be
\beqq \label{eq:pv}
pv = \P\left.\left(\int_{\bbr^{d+1}} |Q{(s,t)}|^2\mu(ds,dt)>u\right)\right|_{u=n\hat p_n T_n}.
\eeqq
Then under $H_0$, $pv$ follows $U(0,1)$, while under $H_1$, $pv$ should be sufficiently small.

% Then a small $pv$ provides evidence of dependence between $R$ and $\mathbf\Theta$ given $R>r_n$. Indeed, under $H_0$, $pv$ follows $U(0,1)$, while under $H_1$, $pv$ should be sufficiently small.
% 
Now consider a decreasing sequence of candidate thresholds $\{r_k\}$. From \eqref{eq:pv}, a sequence of $p$-values $\{pv_k\}$, each corresponding to a threshold $r_k$, can be obtained. Our goal is to find the smallest threshold $r^*$ such that conditional on $R>r^*$, $\Theta$ can reasonably be considered independent of $R$. 
%We have the following general principles:
%\begin{itemize}
%\item
%	For $r_k$ small, the independence of $(R, \Theta)$ conditional on $R>r_k$ is not valid and hence the $pv_k$ would be small.
%\item
%	For $r_k$ large, the distribution of $(R,\Theta)$ conditional on $R>r_k$ is essentially indistinguishable from independence. In this case, the marginal distribution of $pv_k$ should follow a $U(0,1)$.
%\end{itemize}
Note that the $pv_k$'s are not independent for each $k$ since they are computed from the same set of data. Conventional multiple testing procedures, such as Bonferroni correction, are problematic to implement for dependent $p$-values. To counter these limitations, we propose an intuitive and direct method based on subsampling.

The idea is outlined as follows: For a fixed level $r_k$, we choose a subsample of size $n_k$ from the conditional empirical cdf $\hat F_{{\frac{R}{r_n},\Theta|r_k}}$ of $({R_i/r_k},\Theta_i)$ with $R_i>r_k$, $i=1,\ldots,n$. For this subsample, we compute the distance covariance $T_{n,k}$. To compute a $p$-value of $T_{n,k}$ under the assumption that the conditional empirical distribution is a product of the conditional marginals, we take a large number ($L$) of subsamples of {size} $n_k$ from 
\beqo
	\tilde F_{{{\frac{R}{r_n}},\Theta|r_k}}(d\theta,dr) = \hat F_{\Theta|{r_k}}(d\theta) \hat F_{{\frac{R}{r_n}}|{r_k}}(dr),
\eeqo
 and calculate the value $\tilde{T}_{n,k}^{(l)}, l=1,\ldots,L$ for each subsample. The $p$-value of $T_{n,k}$, $pv_k$, is then the empirical $p$-value of $T_{n,k}$ relative the $\{\tilde{T}_{n,k}^{(l)}\}_{l=1,\ldots,L}$. This process, starting with an initial subsample of $n_k$ from $\hat F_{{{\frac{R}{r_n}},\Theta|r_k}}$ is repeated $m$ times, which produces $m$ estimates $\{pv_{k}^{(j)}\}_{j=1,\ldots,m}$ of the $pv_k$, {which are independent conditional on the original sample}. These are then averaged
$$
	\overline{pv}_k = \frac{1}{m}\sum_{j=1}^m pv_{k}^{(j)}.
$$
So for the sequence of levels $\{r_k\}$, we produce a sequence of independent $p$-values $\{\overline{pv}_k\}$.

%At each level $r_k$, take $m$ random subsamples from the observations with $R_i>r_k$ to $m$ blocks, each of size $n_k$. Denote the $j$th subsample to be $\{(R_i^{(j)},\Theta_i^{(j)})\}_{i=1,\ldots,n_k}$, then with respect to a suitable measure $\mu$, we can compute the distance correlation between $R$ and $\Theta$ in this subsample by
%$$
%	T^{(j)}_k = T_{n_k}(R^{(j)},\Theta^{(j)};\mu).
%$$
%Note that the pairs $(R_i^{(j)},\Theta_i^{(j)})$'s are independent conditional on the full data from which they are sampled. Then from Theorem~\ref{thm:1}, under the null hypothesis that $R$ and $\Theta$ are independent given $R>r_k$, $n_kT^{(j)}_k$ converges to a random variable $\hat{Q}_k$ conditional on the full data. The conditional distribution of $\hat{Q}_k$ can be approximated by first randomly subsampling $R_i$'s and $\Theta$'s independently from the full data, and then compute their empirical distance correlation. Hence we can compute the $p$-value of $n_kT^{(j)}_k$ with respect to $\hat{Q}_k$. Denote this $p$-value by $pv_{kj}$.

Our choice of threshold $r$ at which $(\Theta,R)|R>r$ are independent (and dependent otherwise) will be based on an examination of the path of the mean $p$-values, $\{\overline{pv}_k\}$. Note the following two observations:
\begin{itemize}
\item
	If $R$ and $\mathbf\Theta$ are independent given $R>r_k$, then the ${pv_{k}^{(1)},\ldots,pv_{k}^{(m)}}$ will be iid {and approximately $U(0,1)$-distributed}, so that $\overline{pv}_k$ should center around 0.5. 
\item
	If $R$ and $\mathbf\Theta$ are dependent given $R>r_k$, then the ${pv_{k}^{(j)}}$'s will be well below 0.5 ({closer} to 0), and so will $\overline{pv}_k$. 
\end{itemize}
{
By studying the sequence $\{\overline{pv}_k\}$, which we call the mean $p$-value path, we choose the threshold to be the smallest $r_k$ such that $\overline{pv}_l$ is around 0.5 for $l<k$. A well-suited change-point method for our situation is the CUSUM algorithm, by \cite{page:1954}, which detects the changes in mean in a sequence by looking at mean-corrected partial sums. In our algorithm, we use a spline fitting method that is based on the CUSUM approach called wild binary segmentation (WBS), proposed by \cite{fryzlewicz:2014}. The WBS procedure uses the CUSUM statistics of subsamples and fits a piecewise constant spline to $\{\overline{pv}_k\}$. In our setting, we may choose $r_k$ to be the knot of the spline after which the fitted value is comfortably below 0.5.
}

There are several advantages to using the subsampling scheme. First, recall that the $p$-value path $\{pv_k\}$, which is obtained from the whole data set, has complicated serial structure and varies greatly from each realization. In contrast, the mean $p$-values $\overline{pv}_k$ from subsampling are {conditionally} independent and will center around 0.5 with small variance when the total sample size $n$ and the number of subsample $m$ is large.  This, in turns, helps to present a justifiable estimation for the threshold. Second, the calculation of distance covariance can be extremely slow for moderate sample size. Using smaller sample sizes for the subsamples, our computational burden is greatly reduced. In addition, this procedure is amenable to parallel computing, reducing the computation time even further. Third, the subsampling makes it possible to accommodate stationary but dependent data, waiving the stringent independent assumption.

The idea of looking at the mean $p$-value path is inspired by \cite{mallik:sen:banerjee:michailidis:2011}, which used the mean of $p$-values from multiple independent tests to detect change points in population means.

%---------------------------------------------------------------------------------------------------------------------------------------------%
%--------------------------------------------  DATA ILLUSTRATION --------------------------------------------%
%---------------------------------------------------------------------------------------------------------------------------------------------%

\section{Data Illustration} \label{sec:examples}

In this section, we demonstrate our threshold selection method through simulated and real data examples. 

{In practice, we set the sequence of thresholds $\{r_k\}$ to be the corresponding upper quantiles to $\{q_k\}$, a pre-specified sequence of quantile levels. The subsample size $n_k$ at each threshold $r_k$ is set as $n_k=n_0\cdot q_k$ for some $n_0<<n$.  This is designed such that for any $r_k$, each subsample is a $n_0/n$ fraction of all the eligible data points with $R>r_k$.  Then the choice of $\{n_k\}$ boils down to the choice of $n_0$, which should reflect the following considerations: i) $n_0$ should be large enough to ensure good resolution of $p$-values at all levels;  ii) $n_0/n$ should be sufficient small such that the subsamples do not contain too much overlap in observations;  iii) larger $n_0$ requires heavier computation for the distance correlation.  In our examples, where the total sample size $n$ ranges from 3000 to 20000, we find $n_0$ between 500 and 1000 to be a suitable choice. The number of subsamples $m$ can be set as large as computation capacity allows.  In our examples, we take $m=60$.
}

For all the examples, we choose the weight function $\mu$ for distance covariance to be \eqref{eq:dcormeas} with $\kappa=1${, and the number of replications used to calculate each $p$-value is $L=200$}. {To ensure that the moment conditions are met, the distance correlation is applied to the log of} the radial part $R$ {in} all {examples}.

%--------------------------------------------  EX 1  --------------------------------------------%

\subsection{Simulated data with known threshold} \label{ex:1}

To illustrate our methodology, we simulate observations from a distribution with a known threshold for which $R$ and $\Theta$ become independent.

Let $R$ be the absolute value of a $t$-distribution with 2 degrees of freedom and {$\Theta_1,\Theta_2$ be independent random variables such that $\Theta_1\sim U(0,1)$, $\Theta_2\sim Beta(3,3)$. Set}
$$
{\Theta |R =
	\begin{cases}
		\Theta_1, & \quad \text{if } R>r_{0.2}, \\
		\Theta_2, &  \quad \text{if } R\le r_{0.2}, \\
	\end{cases}}
$$
where $r_{0.2}$ is the upper $20\%$-quantile of $R$. Then $R$ and $\Theta$ are independent given $R>r$ if and only if $r\ge r_{0.2}$. 
Let $(X_{i1},X_{i2}) = (R_i\Theta_i,R_i(1-\Theta_i))$, $i=1,\ldots,n$, be the simulated observations. We generate $n=10000$ iid observations from this distribution. Figures~\ref{fig1:subfig1}, \ref{fig1:subfig2} and \ref{fig1:subfig3} show the data in Cartesian and polar coordinates. Our goal is to recover the tail angular distribution by choosing the appropriate threshold.

A sequence of candidate thresholds $\{r_k\}$ is selected to be the empirical upper quantiles of $R$ corresponding to {$\{q_k\}$, 150 equidistant points between 0.01 and 0.4}. We apply the procedure described in Section~\ref{sec:method} to the data. For each $r_k$, the mean $p$-value $\overline{pv}_{k}$ is calculated using $m=60$ random subsamples, each of size $n_{k}=500\cdot q_k$, from the observations with $R_i>r_k$. 
Figure~\ref{fig1:subfig4} shows the mean $p$-value path. For the WBS algorithm, we set the threshold to be the largest $r_k$ {such that for all thresholds $r$ (quantile level $q$) such that $r<r_k$ ($q>q_k$), the fitted spline of the $p$-value stays \revise{below} 0.45\footnote{\revise{Of course, other selection rules can be used.  For example, a more conservative approach would be choosing the threshold as the largest $r_k$ such that for $r>r_k$, the fitted spline of the $p$-value stays above 0.45.
}}}. The threshold levels chosen is $20.4\%$, which are in good agreement with the true independence level 0.2. The empirical cdfs of the truncated $\Theta_i$'s corresponding to the chosen thresholds is shown in Figure~\ref{fig1:subfig5}. We can see that the true tail angular cdf (i.e., $U(0,1)$) is accurately recovered.

%
%We further repeat the simulation for 100 times. A histogram of the threshold estimates are shown in the left panel of Figure \ref{fig:e1rep}.  Except for three (not very outlying) outliers, the threshold estimate is close to but tend to be slightly smaller than the truth. The right panel of Figure \ref{fig:e1rep} shows the 100 empirical cdf of the estimated angular distribution, overlaying with the true cdf for $U(0,1)$. We see that the estimation is quite accurate.
%
%
%\begin{figure} 
%\includegraphics[width=5.5in]{e1rep.pdf}
%\caption{Left: histogram of threshold estimates from 100 simulation samples; Right: 100 empirical angular cdf with the truth.}
%\label{fig:e1rep}
%\end{figure}

\begin{figure}[t]
	\begin{subfigure}[]{
		\includegraphics[width=.31\textwidth] {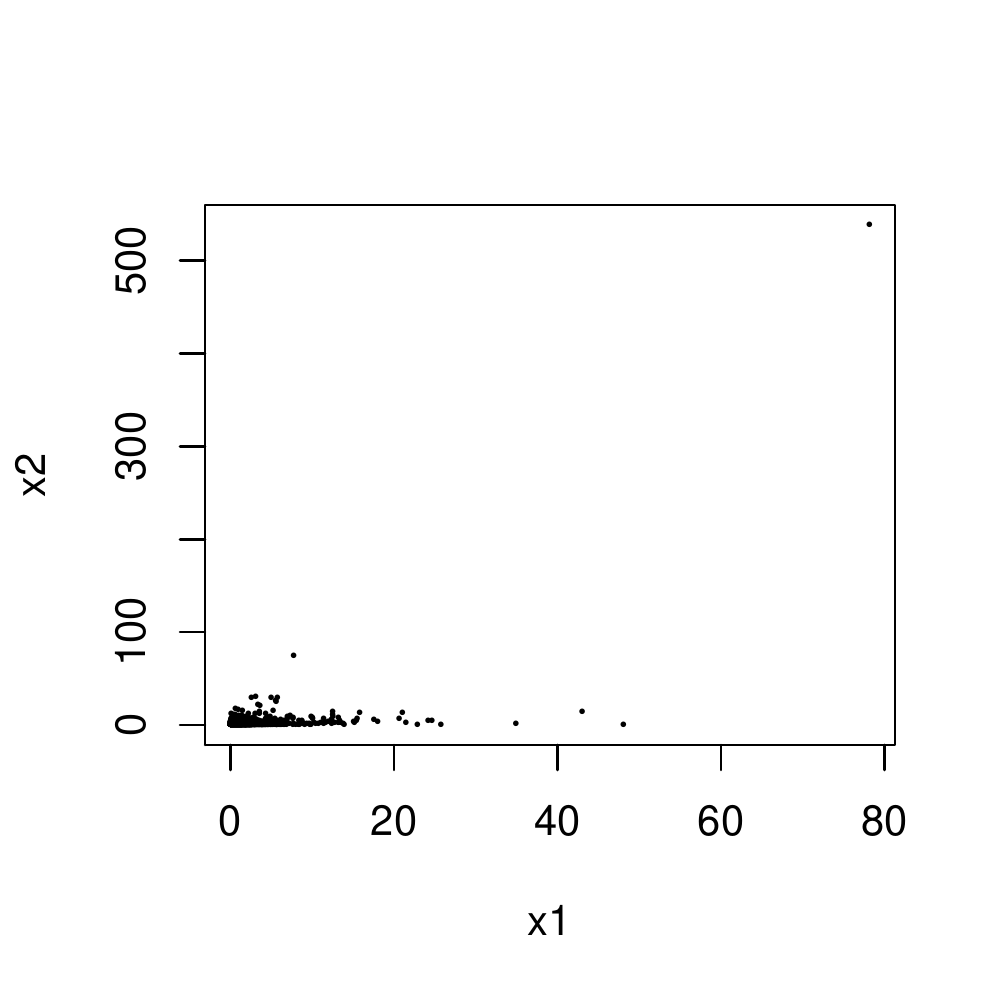}   
		\label{fig1:subfig1} }
	\end{subfigure}
	\begin{subfigure}[]{
		\includegraphics[width=.31\textwidth] {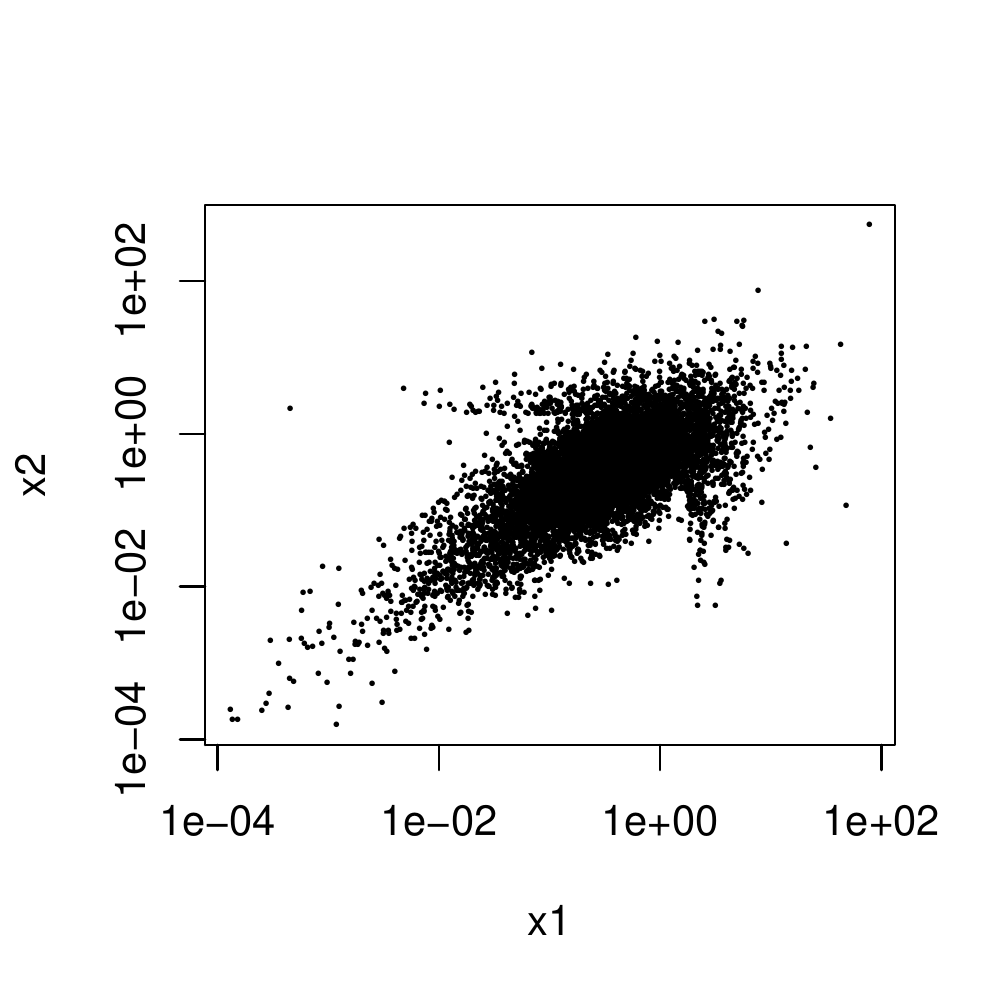}   
		\label{fig1:subfig2} }
	\end{subfigure}
	\begin{subfigure}[]{
		\includegraphics[width=.31\textwidth] {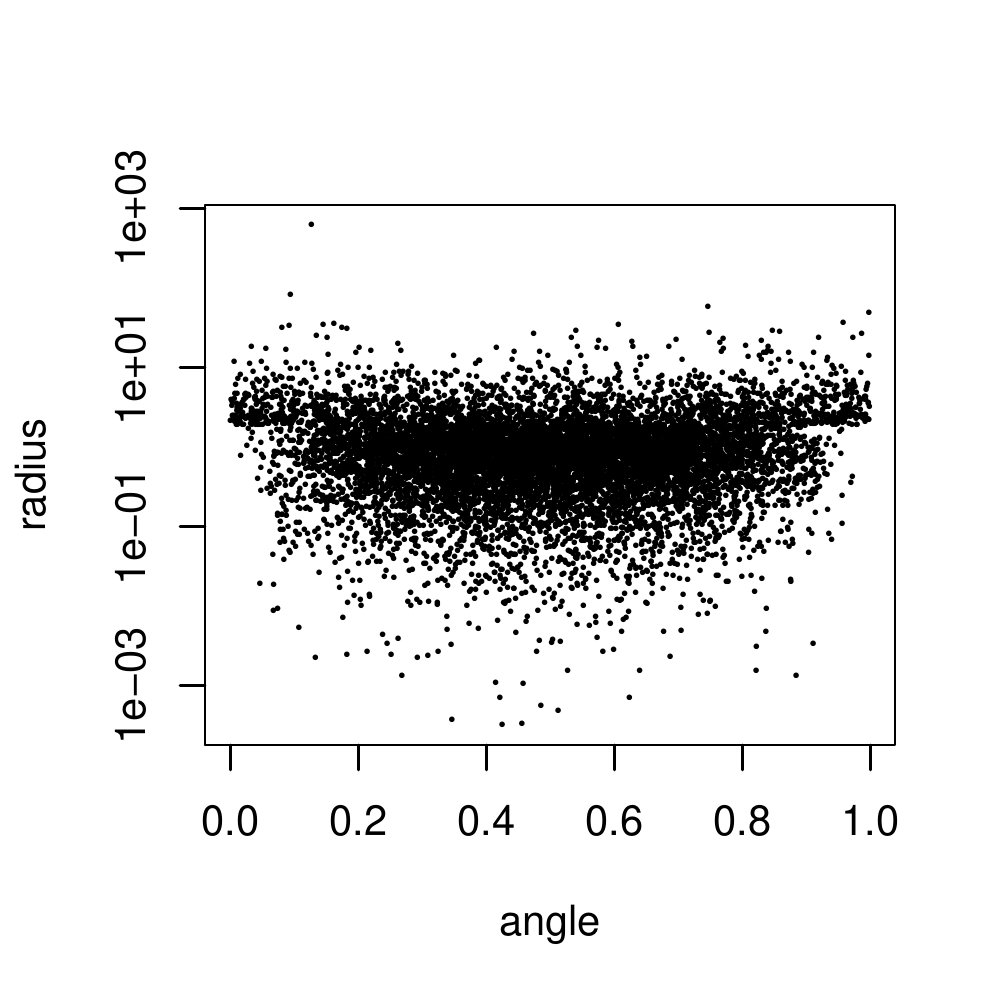}   
		\label{fig1:subfig3} }
	\end{subfigure}
	\begin{subfigure}[]{
		\includegraphics[width=.62\textwidth] {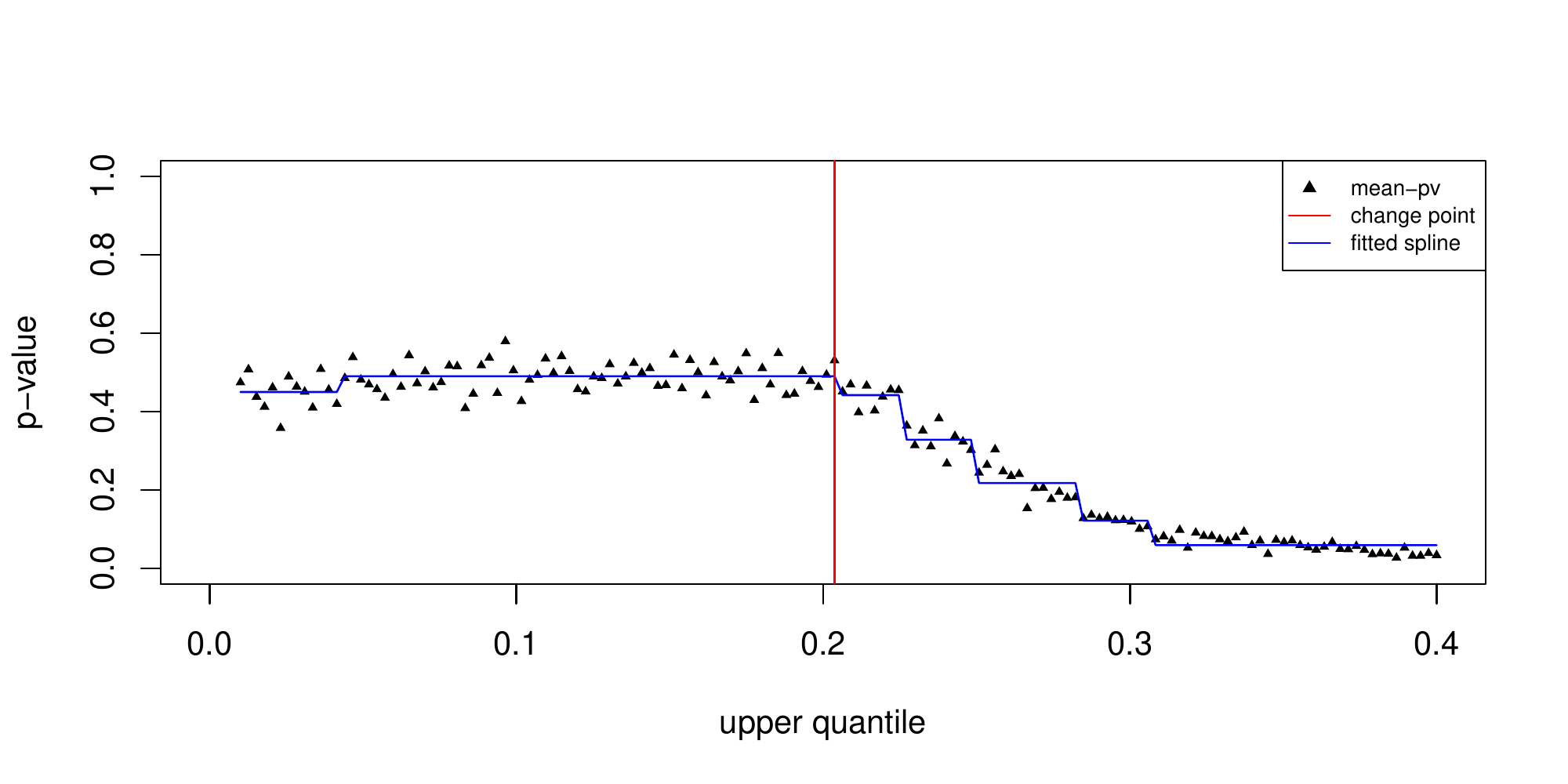}   
		\label{fig1:subfig4} }
	\end{subfigure}
	\begin{subfigure}[]{
		\includegraphics[width=.31\textwidth] {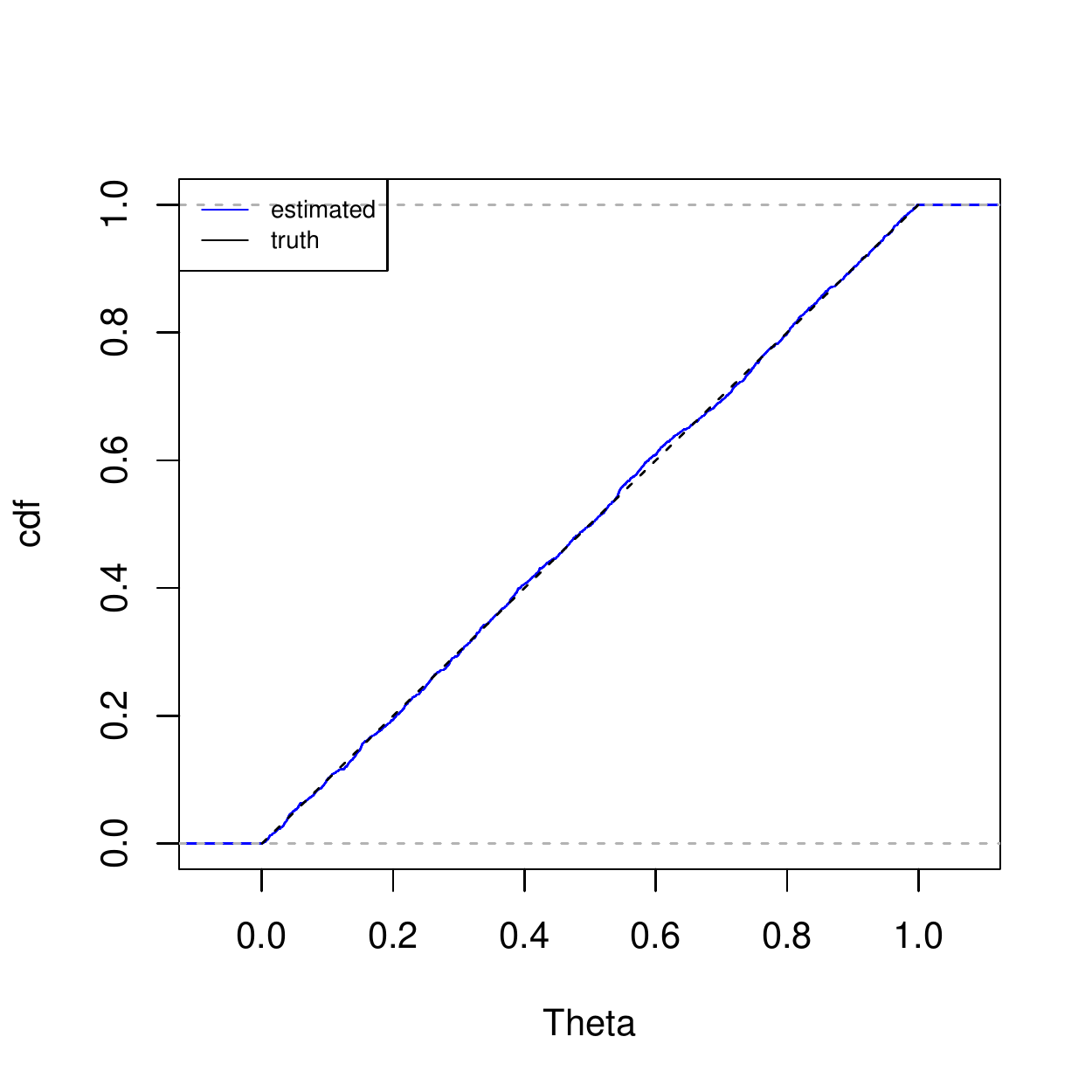}   
		\label{fig1:subfig5} }
	\end{subfigure}
	\label{fig:e1}
	\caption{Example \ref{ex:1}. (a) scatterplot of $(X_{i1},X_{i2})$; (b) scatterplot of $(X_{i1},X_{i2})$ in log-log scale; (c) scatterplot of $(R_i,\Theta_i)$; (d) mean $p$-value path (black triangles), fitted WBS spline (blue line), and the chosen threshold quantile (red vertical line); (e) estimated cdf of $\Theta$ using the threshold chosen, compared with the truth (black dotted).}
\end{figure}

%--------------------------------------------  EX 2  --------------------------------------------%

\subsection{Simulated logistic data} \label{ex:2}

We simulate data from a bivariate logistic distribution, which is bivariate regularly varying. Recall from Example \ref{ex:bilogistic} that $(X_1,X_2)$ follows a bivariate logistic distribution {if it has cdf \eqref{eq:logistic}}.
In this example, we set $\gamma=0.8$ and generate $n=10000$ iid observations from this distribution. Similar to the previous example, for each threshold $r_k$ corresponding to the upper $q_k$ quantile, {where $\{q_k\}$ is chosen to be the 150 equidistant points between 0.01 and 0.3}. The mean $p$-value $\overline{pv}_{k}$ is calculated using $m=60$ random subsamples of size $n_{k}=500\cdot q_k$ from the observations with $R_i>r_k$. 

Figures~\ref{fig2:subfig1} , \ref{fig2:subfig2} and \ref{fig2:subfig3} show the scatterplots of the data. Here the $L_1$-norm is used to transform the data into polar coordinates. Our algorithms suggests using $7.4\%$ of the data to estimate the angular distribution. The estimated cdf of the angular distribution is shown with the theoretical limiting cdf, derived from \eqref{eq:bilog:angmeas}, in Figure~\ref{fig2:subfig5}. So even though $R$ and $\Theta$ are not independent for any threshold $r_k$, our procedures produce good estimates of the limiting distribution of $\Theta$.

\begin{figure}[t]
	\begin{subfigure}[]{
		\includegraphics[width=.31\textwidth] {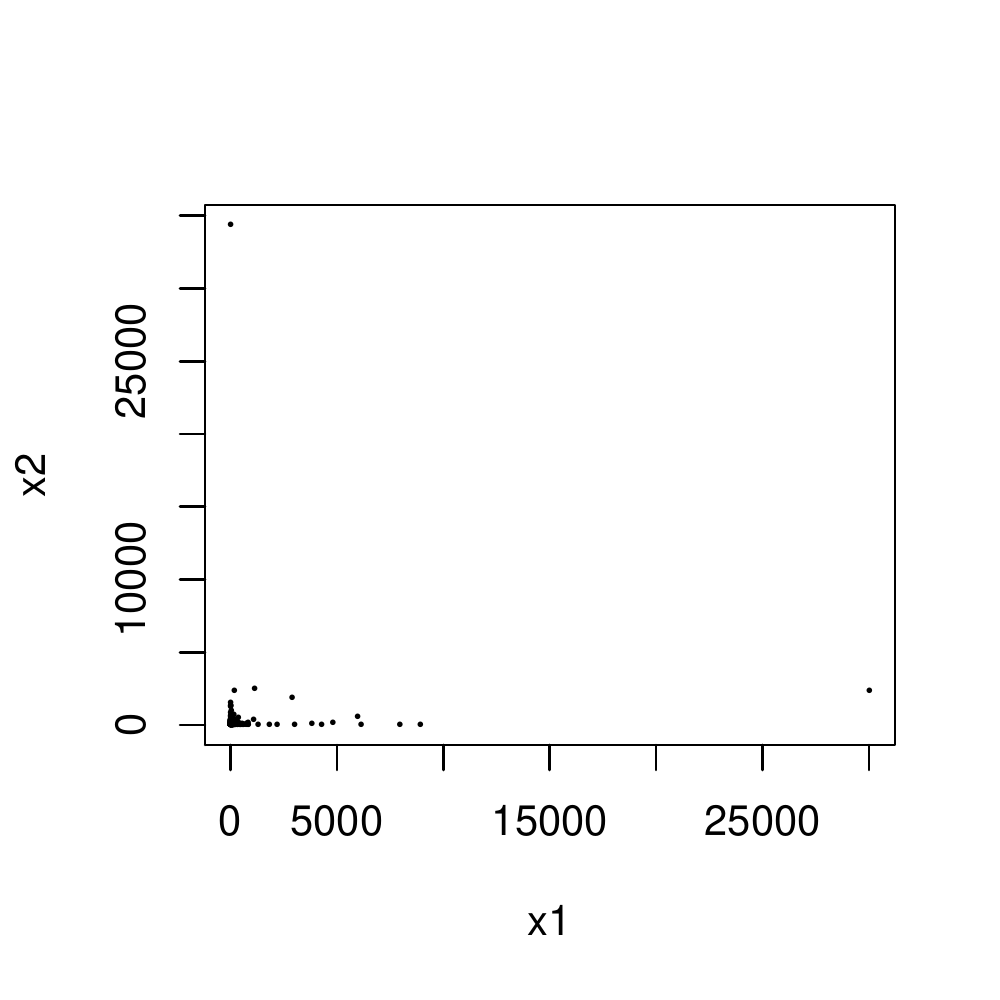}   
		\label{fig2:subfig1} }
	\end{subfigure}
	\begin{subfigure}[]{
		\includegraphics[width=.31\textwidth] {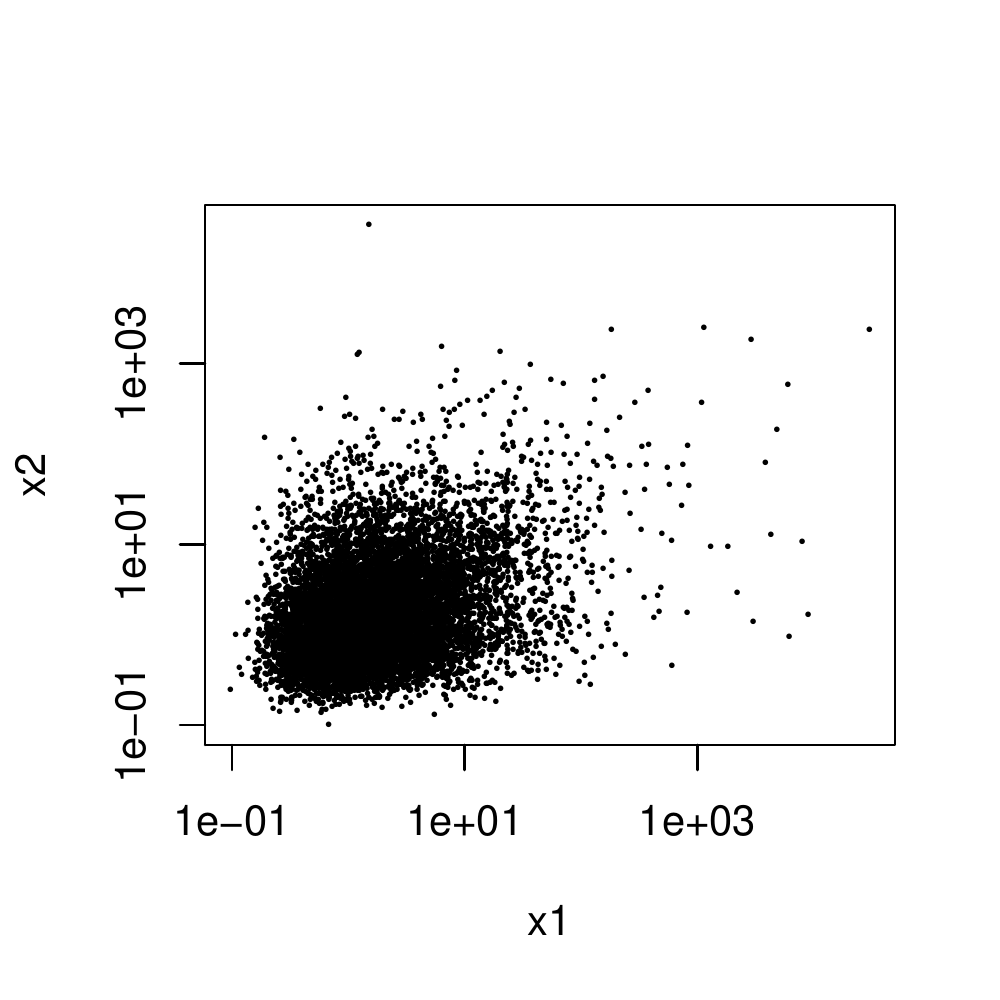}   
		\label{fig2:subfig2} }
	\end{subfigure}
	\begin{subfigure}[]{
		\includegraphics[width=.31\textwidth] {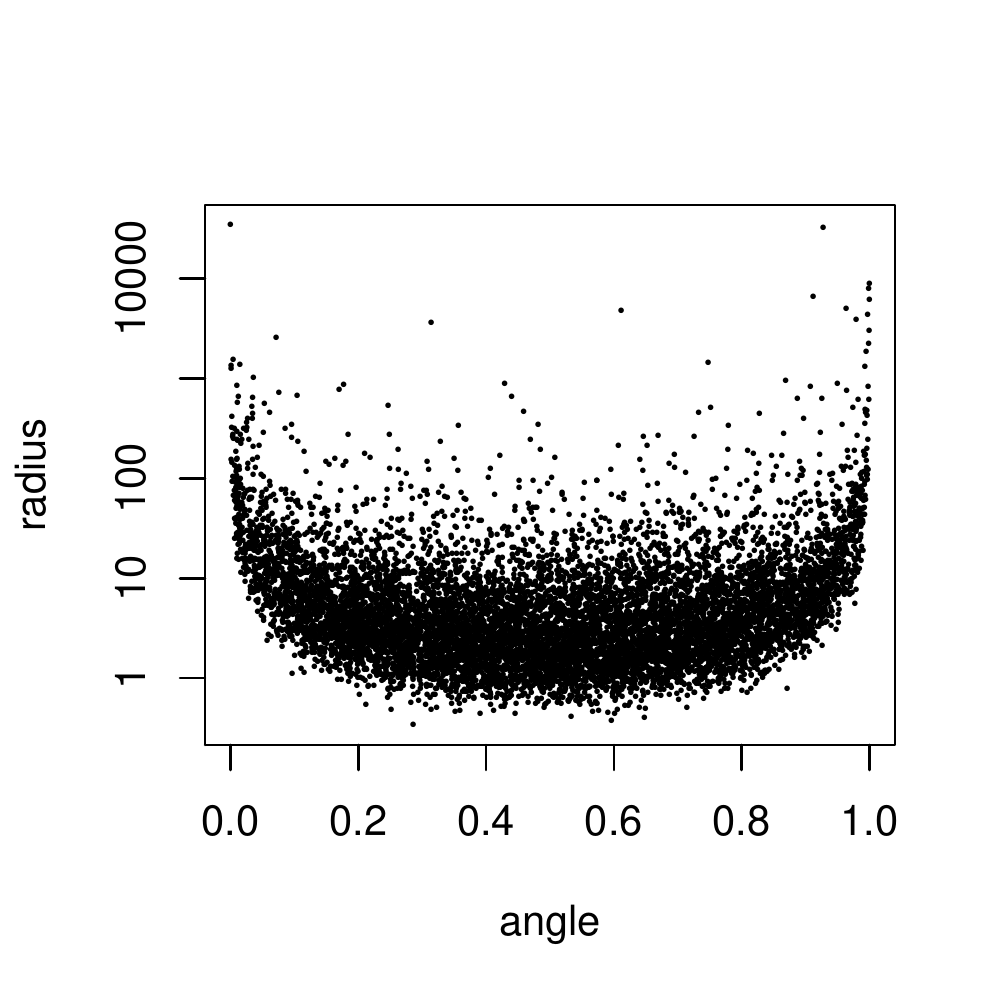}   
		\label{fig2:subfig3} }
	\end{subfigure}
	\begin{subfigure}[]{
		\includegraphics[width=.62\textwidth] {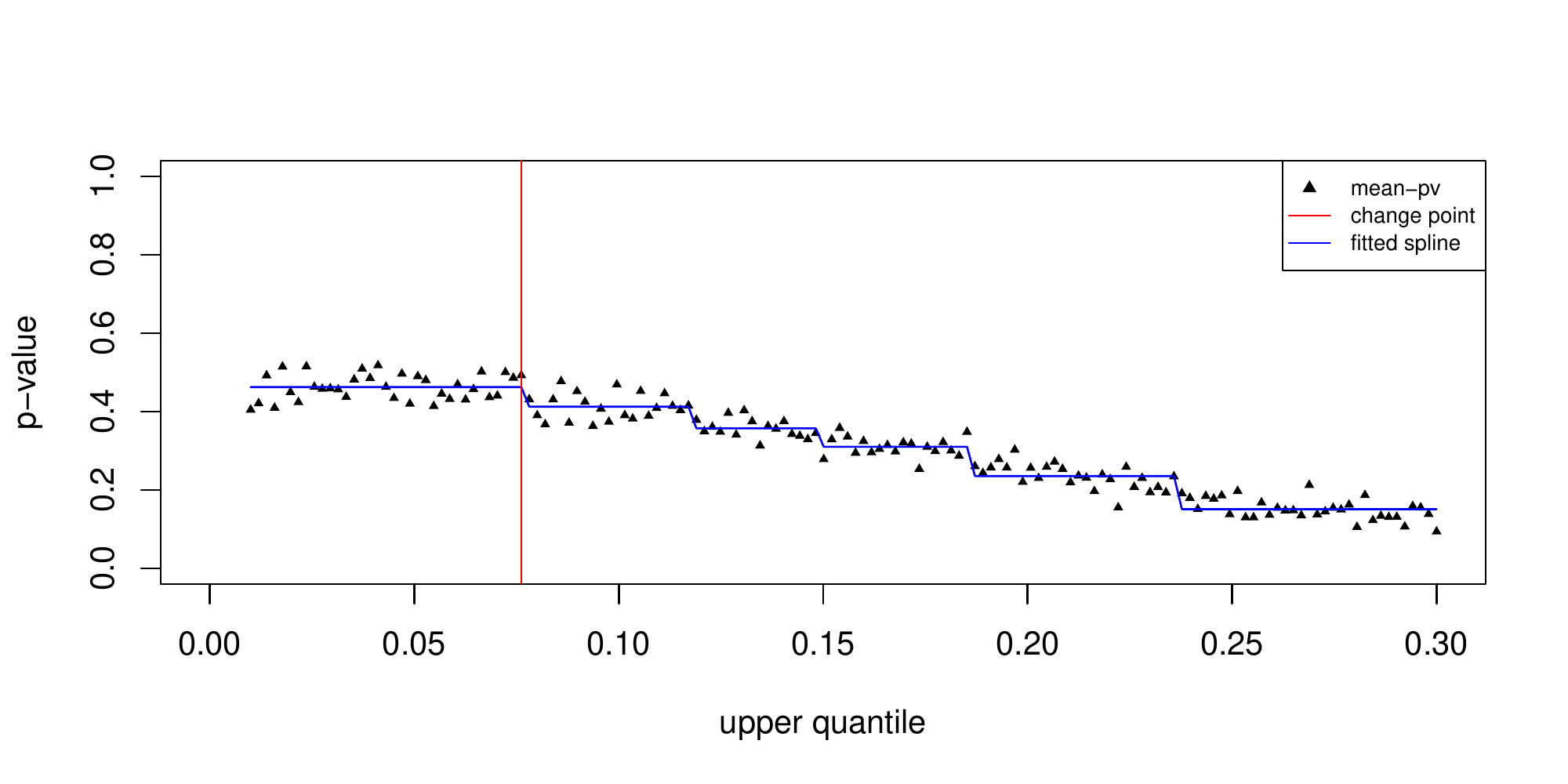}   
		\label{fig2:subfig4} }
	\end{subfigure}
	\begin{subfigure}[]{
		\includegraphics[width=.31\textwidth] {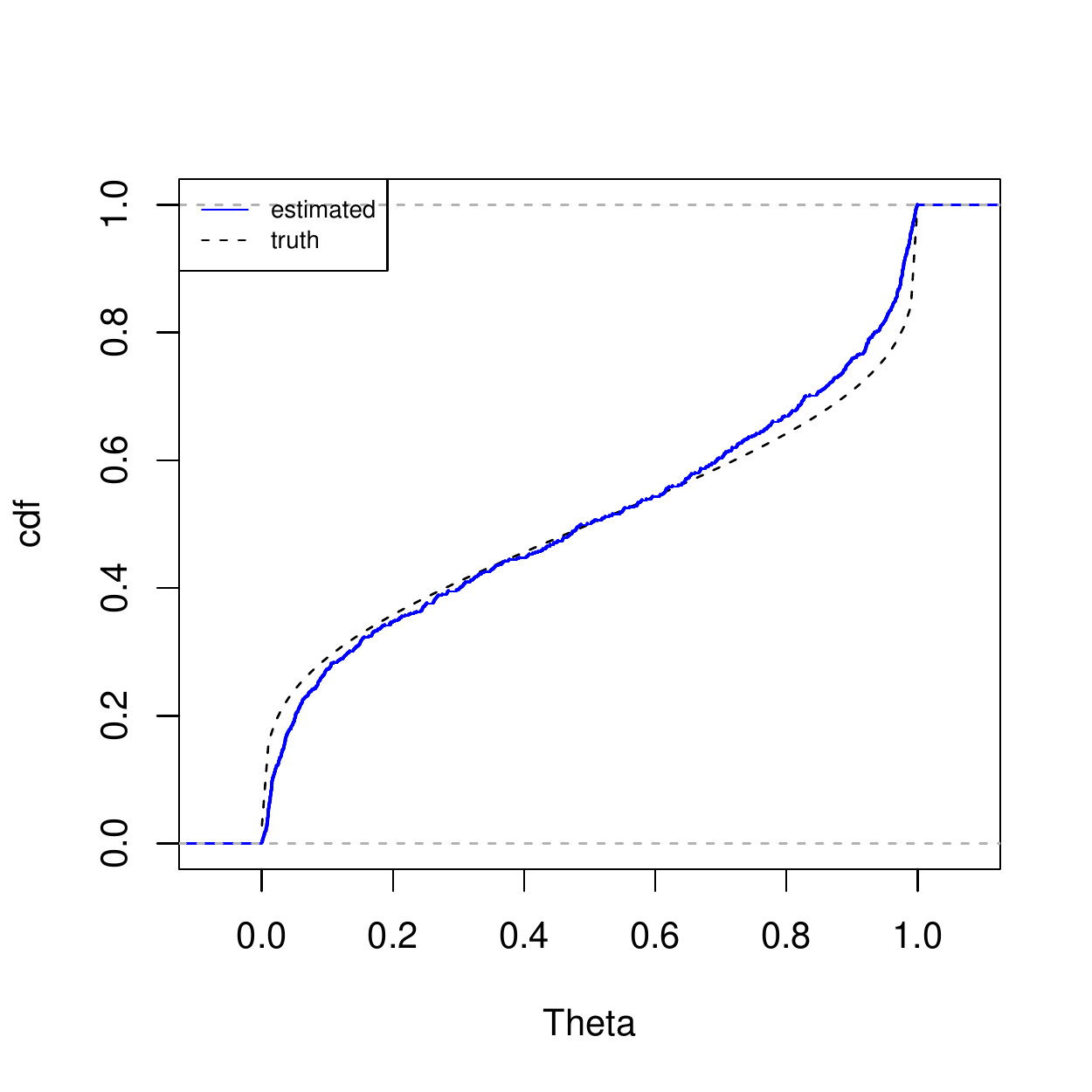}   
		\label{fig2:subfig5} }
	\end{subfigure}
	\label{fig:e2}
	\caption{Example \ref{ex:2}.  (a) scatterplot of $(X_{i1},X_{i2})$; (b) scatterplot of $(X_{i1},X_{i2})$ in log-log scale; (c) scatterplot of $(R_i,\Theta_i)$; (d) mean $p$-value path (black triangles), fitted WBS spline (blue line), and the chosen threshold quantile (red vertical line); (e) estimated cdf of $\Theta$ using the threshold chosen, compared with the theoretical limiting cdf (black dotted).}
\end{figure}

%--------------------------------------------  EX 3 --------------------------------------------%

\subsection{Real data} \label{ex:3}

In this example, we look at the following exchange rate returns relative to the US dollar: Deutsche mark (DEM), British pound (GBP), Canadian dollar (CAD), and Swiss franc (CHF). The time spans for the data are 1990-01-01 to 1998-12-31 with a total of 3287 days of observations. We examine the pairs GBP/CHF, CAD/CHF, DEM/CHF and estimate the angular density in the tail for each pair. 
Figures~\ref{fig3:subfig1}--\ref{fig3:subfig3} present the scatter plots of the data. The marginals of the observations are standardized using the rank transformation proposed in \cite{joe:smith:weissman:1992}:
$$
	Z_{i}=1/\log\{n/(Rank(X_{i})-.5)\},\quad i=1,\ldots,n.
$$
Again {$\{q_k\}$ is chosen to be the 150 equidistant points between 0.01 and 0.3}, and the mean $p$-value $\overline{pv}_{k}$ is calculated using $m=60$ random subsamples of size $n_{k}=500\cdot q_k$ from the observations with $R_i>r_k$.  Note that while it may not be reasonable to view the observations as iid, the subsampling scheme can still be applied to choose the threshold of independence between $R$ and $\Theta$. 

The mean $p$-value paths are shown in Figures~\ref{fig3:subfig7}--\ref{fig3:subfig9}. The threshold levels selected for the three pairs are $9.6\%$, $7.4\%$, $16\%$, respectively. Figures~\ref{fig3:subfig4}--\ref{fig3:subfig6} show the shape of the estimated angular densities for each pairs. As expected, the tails of the two central European exchange rates, DEM and CHF, are highly dependent. In contrast, that of {CAD and CHF} are almost independent.

\begin{figure}[t]
	\begin{subfigure}[]{
		\includegraphics[width=.31\textwidth] {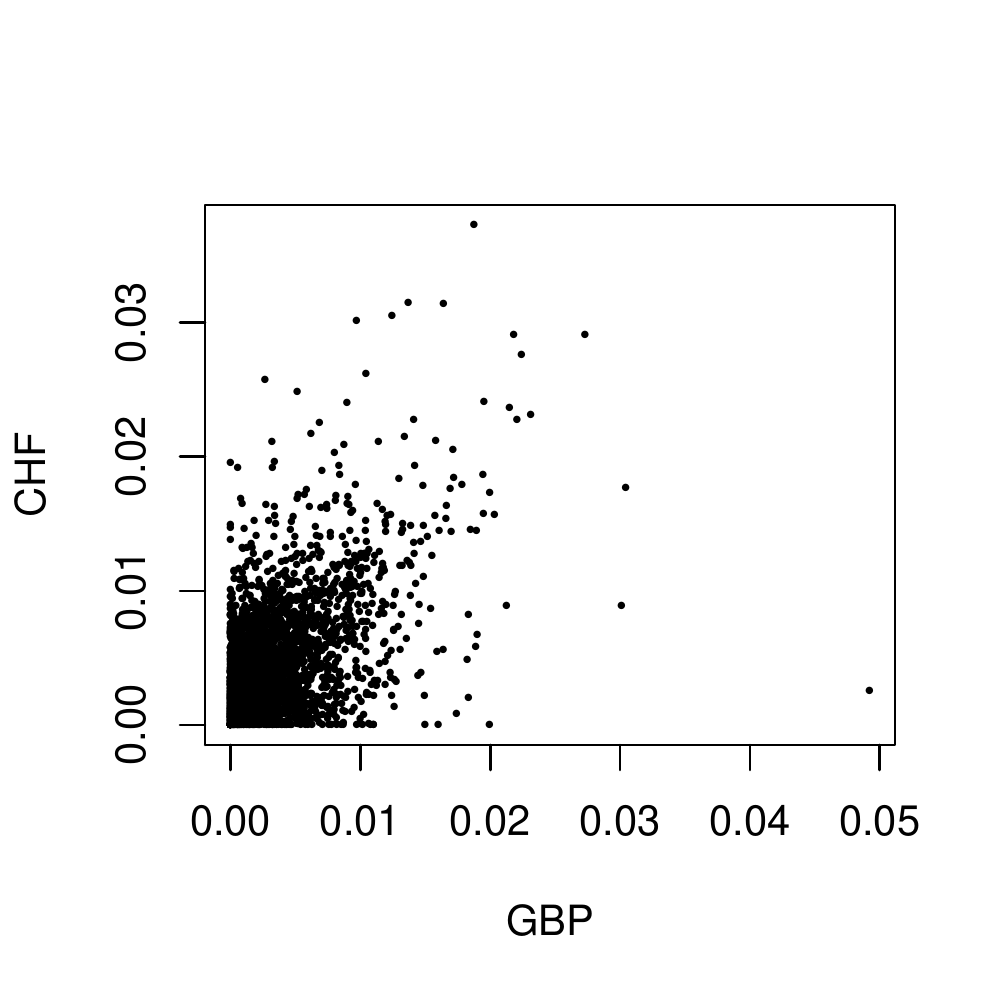}   
		\label{fig3:subfig1} }
	\end{subfigure}
	\begin{subfigure}[]{
		\includegraphics[width=.31\textwidth] {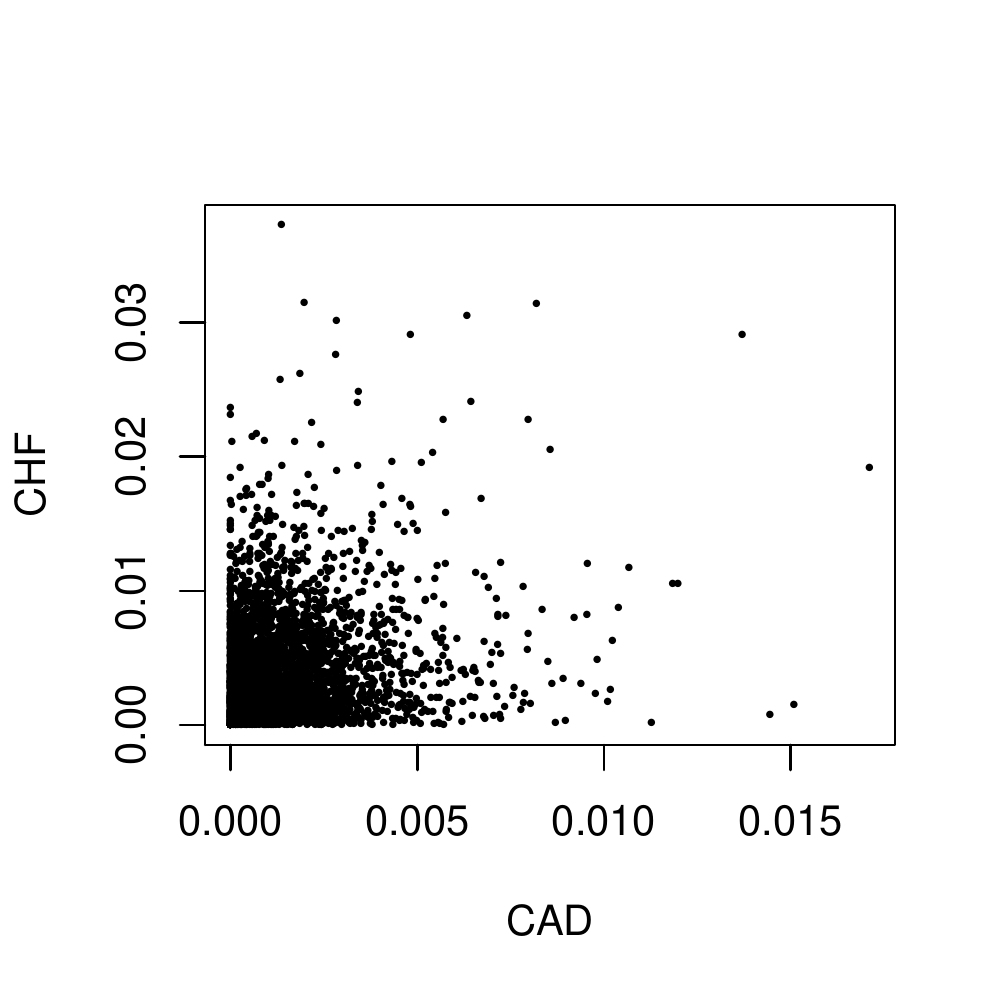}   
		\label{fig3:subfig2} }
	\end{subfigure}
	\begin{subfigure}[]{
		\includegraphics[width=.31\textwidth] {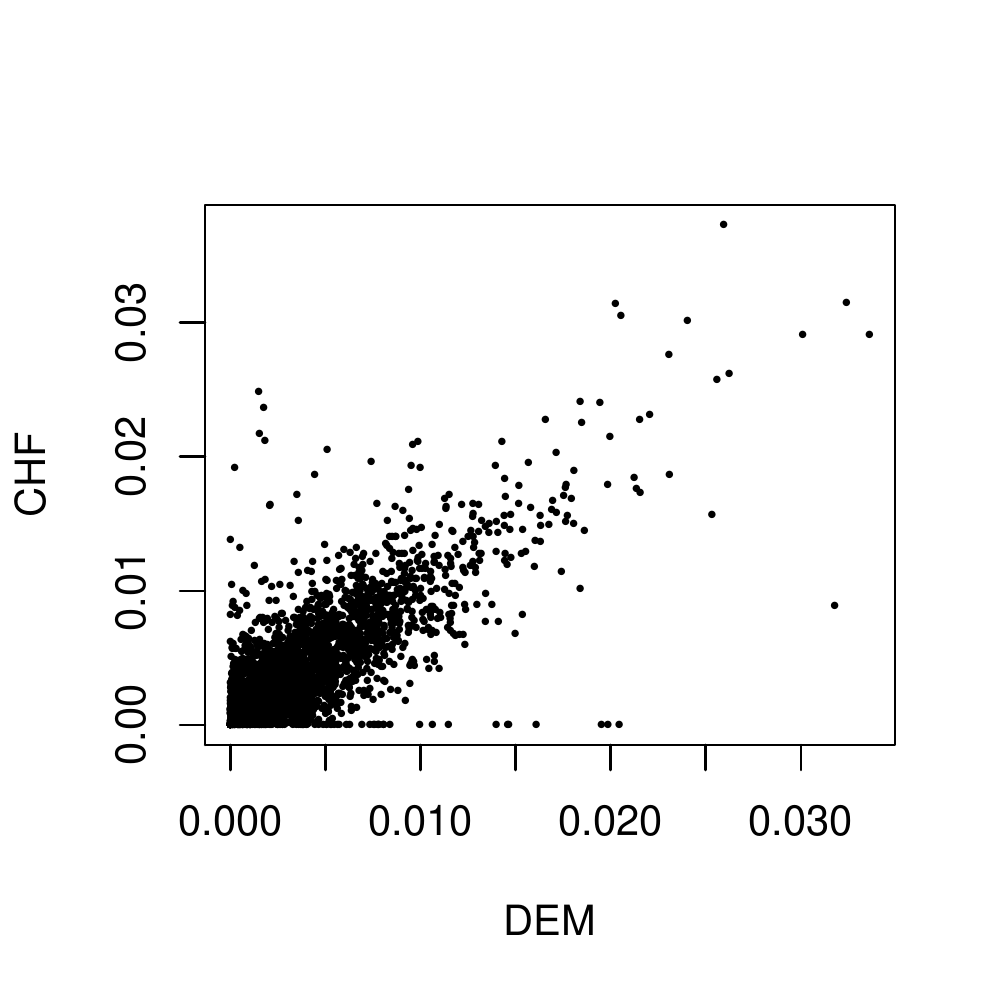}   
		\label{fig3:subfig3} }
	\end{subfigure}
	\begin{subfigure}[]{
		\includegraphics[width=.31\textwidth] {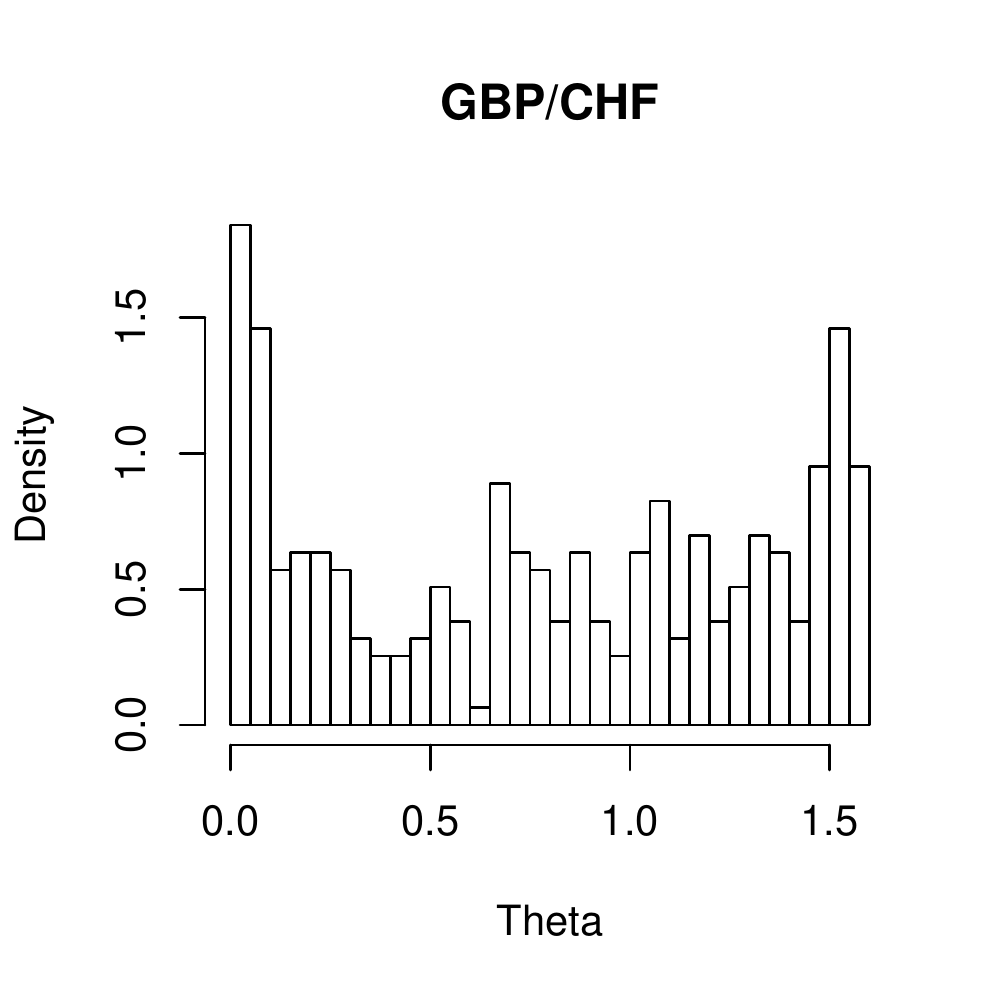}   
		\label{fig3:subfig4} }
	\end{subfigure}
	\begin{subfigure}[]{
		\includegraphics[width=.31\textwidth] {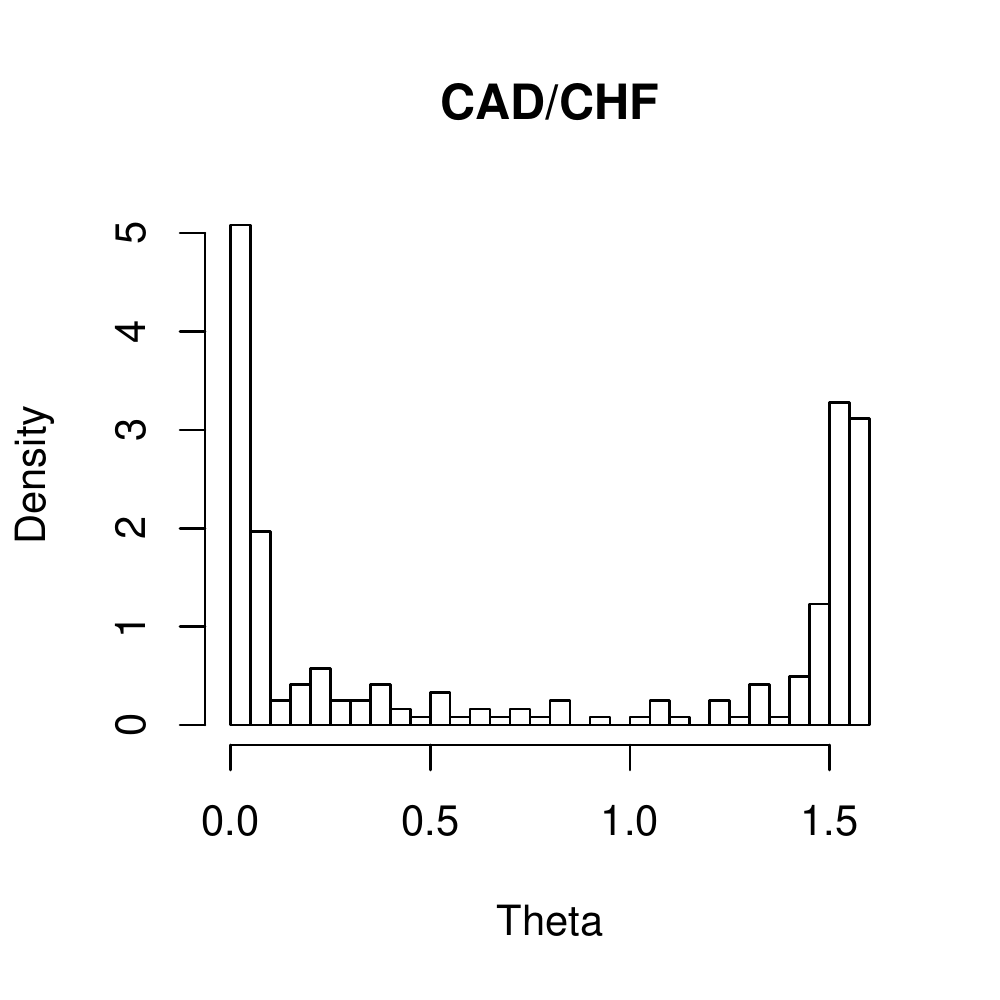}   
		\label{fig3:subfig5} }
	\end{subfigure}
	\begin{subfigure}[]{
		\includegraphics[width=.31\textwidth] {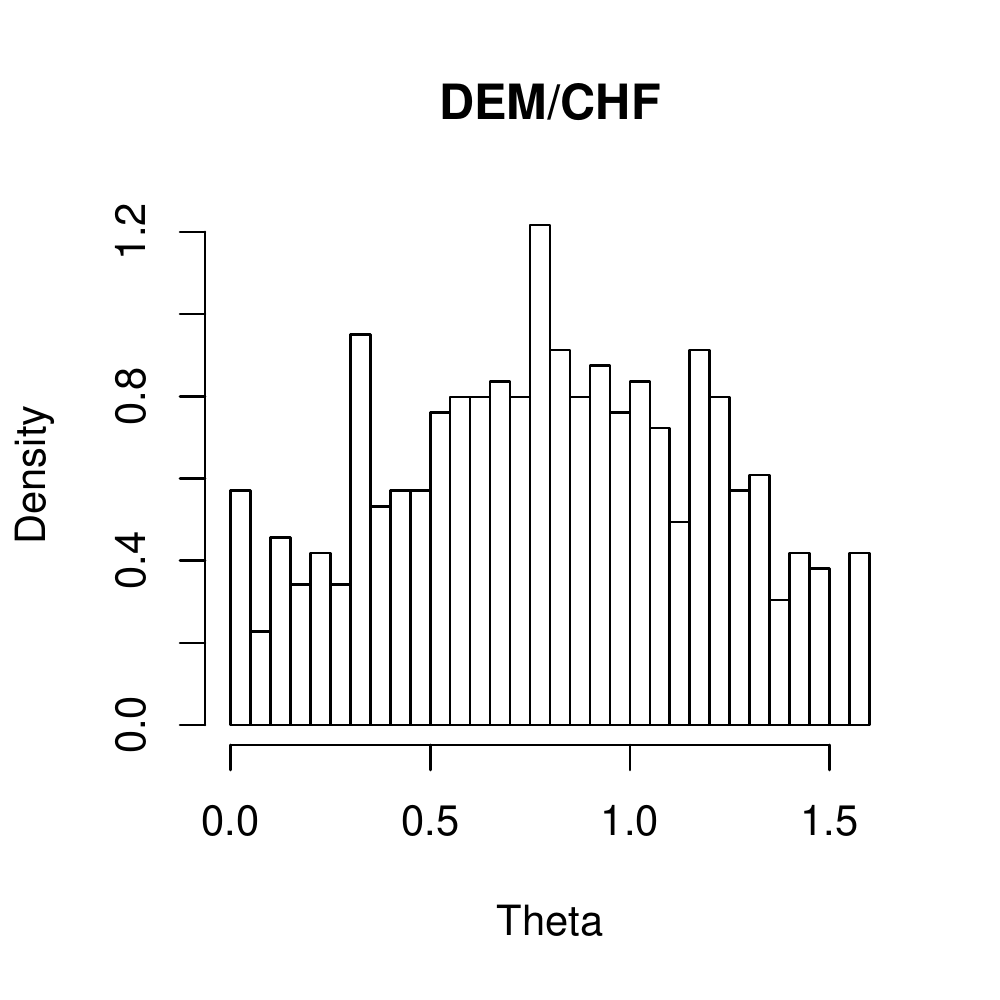}   
		\label{fig3:subfig6} }
	\end{subfigure}
	\label{fig:e3}
	\caption{Example \ref{ex:3}. Analysis of the paired exchange rate returns: CHF/DEM, CHF/GBP, CHF/CAD with respect to USD between 1990-01-01 to 1998-12-31. (a)--(c): Scatter plots of the standardized paired exchange rate returns; (d)--(f): Estimated angular densities using the estimated thresholds chosen.}
\end{figure}

\begin{figure}
	\begin{subfigure}[]{
		\includegraphics[width=.8\textwidth] {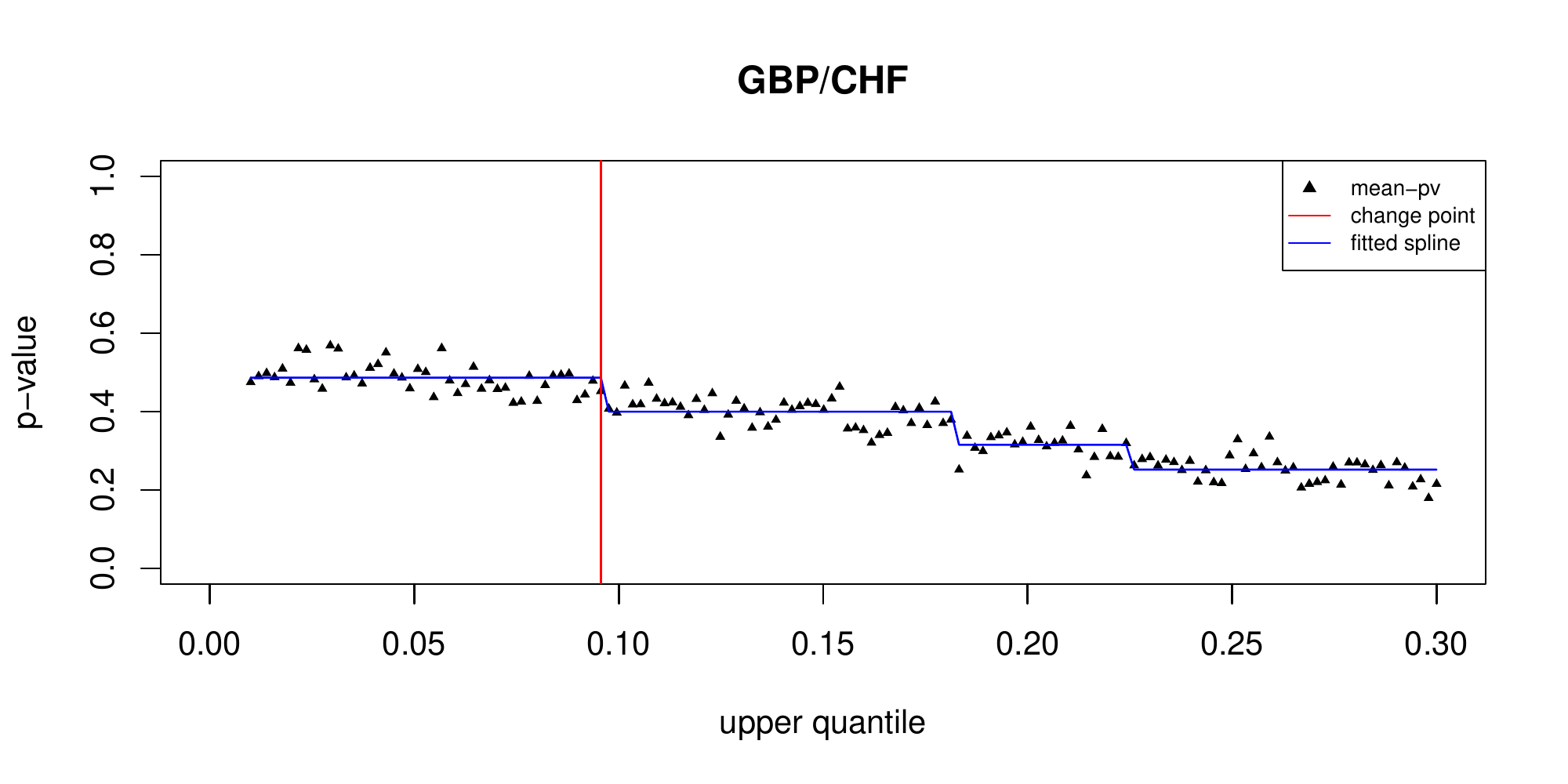}   
		\label{fig3:subfig7} }
	\end{subfigure}
	\begin{subfigure}[]{
		\includegraphics[width=.8\textwidth] {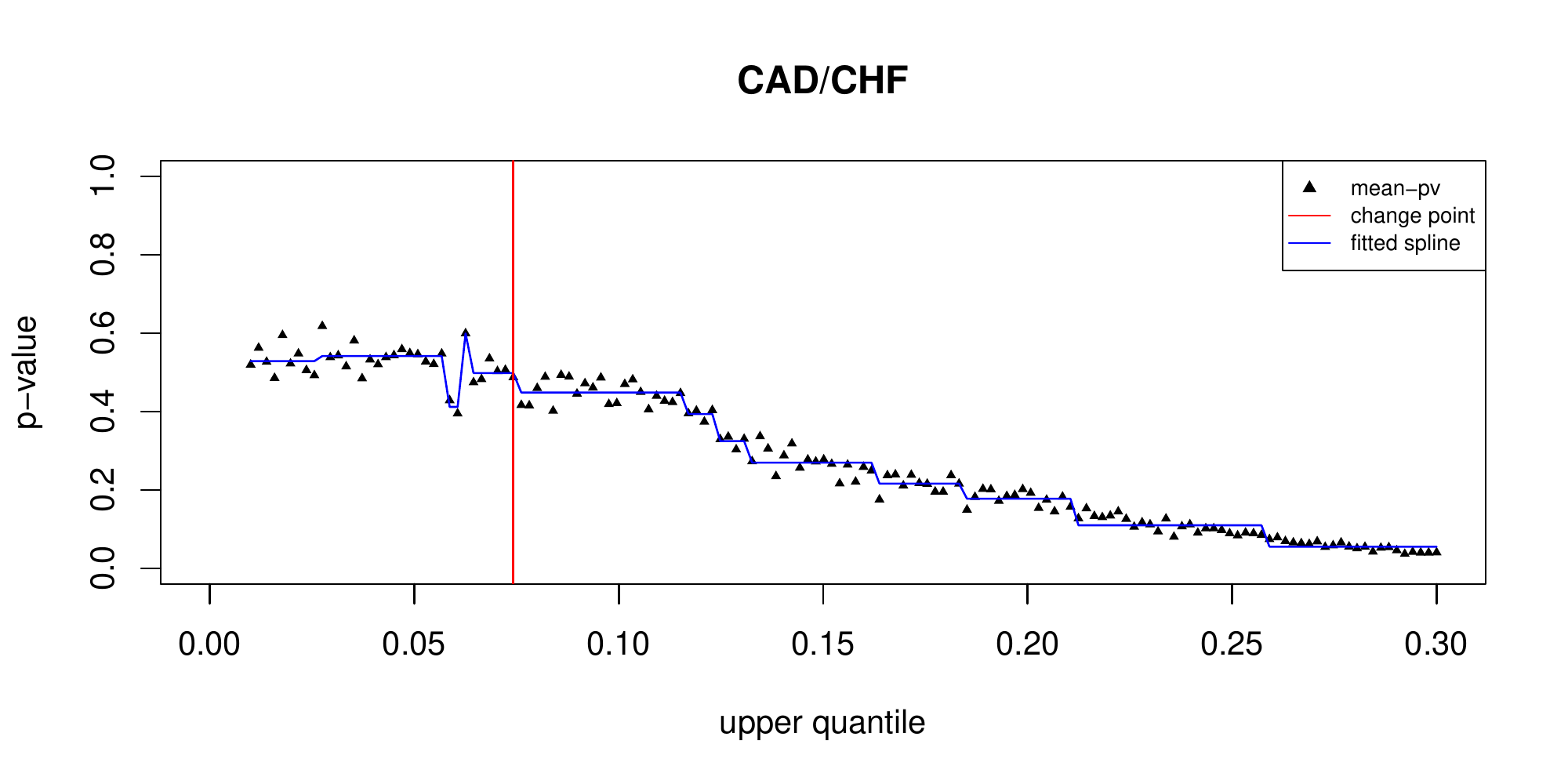}   
		\label{fig3:subfig8} }
	\end{subfigure}
	\begin{subfigure}[]{
		\includegraphics[width=.8\textwidth] {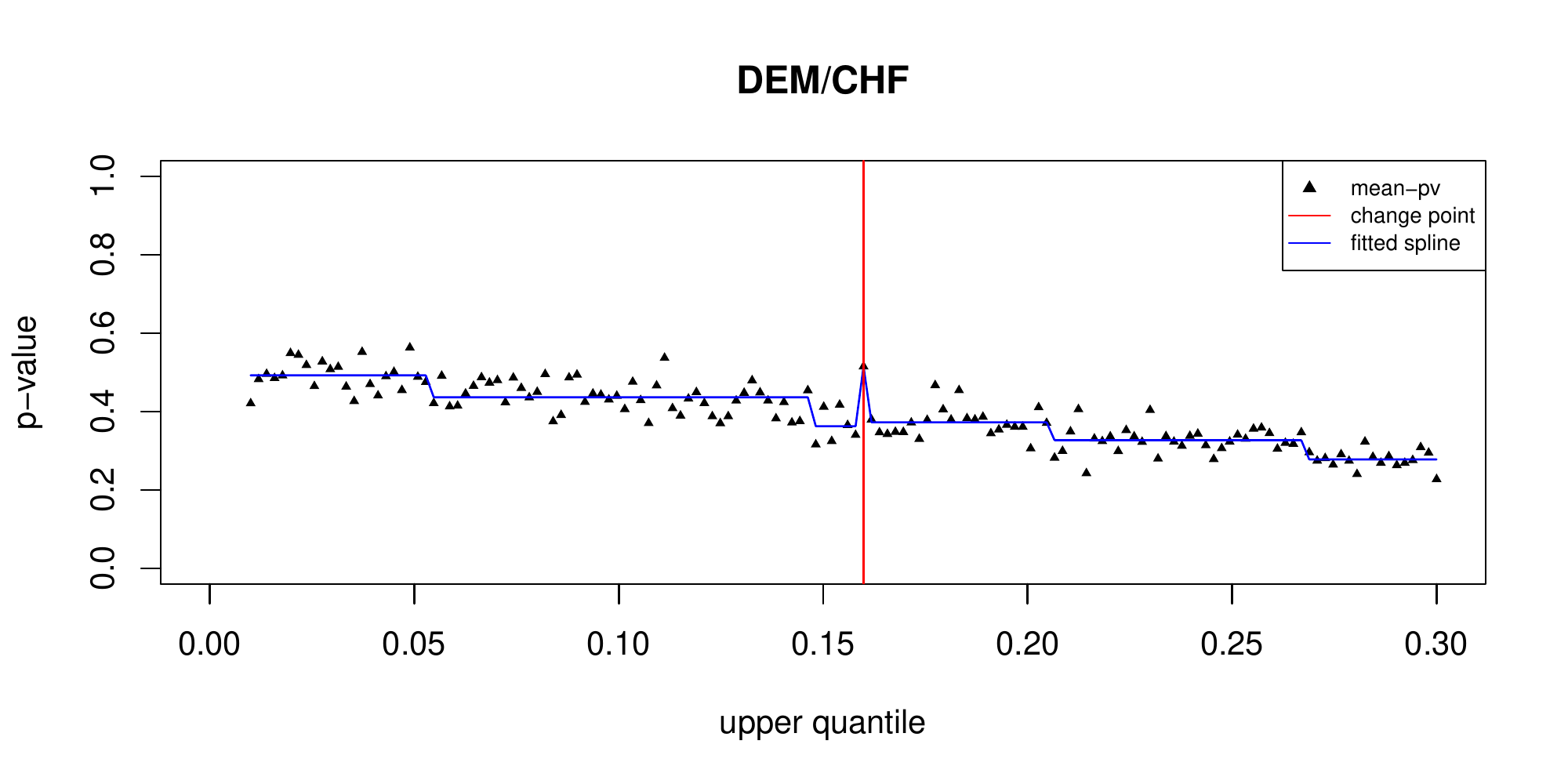}   
		\label{fig3:subfig9} }
	\end{subfigure}
	\label{fig:e32}
	\caption{Example \ref{ex:3} (cont.). Analysis of the paired exchange rate returns: CHF/DEM, CHF/GBP, CHF/CAD with respect to USD between 1990-01-01 to 1998-12-31. (a)--(c): mean $p$-value paths (black triangles), fitted WBS splines (blue lines) and the chosen threshold quantiles (red vertical line).}
\end{figure}

%--------------------------------------------  EX 4 --------------------------------------------%

\subsection{Simulated non-regularly varying data} \label{ex:4}

In this example, we generate data from a model which is not regularly varying.  Let $R$ be a random variable from the standard Pareto distribution:
$$
	\P(R>r) = r^{-1},\quad r\ge1.
$$
Let $\Theta_1,\Theta_2$ be independent random variables such that $\Theta_1\sim U(0,0.5)$, $\Theta_2\sim U(0.5,1)$. Set
$$
\Theta|R \sim
	\begin{cases}
		\Theta_1, & \quad \text{if }\log R \in(2k,2k+1] \text{ for some integer $k$}, \\
		\Theta_2, &  \quad \text{if } \log R \in(2k+1,2k+2] \text{ for some integer $k$}. \\
	\end{cases}
$$
For any positive integer $k$, it can be verify that
$$
	\P(\Theta \in (0,0.5)|R>e^{2k}) = \frac{1-e^{-1}}{1-e^{-2}},
$$
while
$$
	\P(\Theta \in (0,0.5)|R>e^{2k+1}) = \frac{e^{-1}-e^{-2}}{1-e^{-2}}.
$$
Hence $\P(\Theta \in \cdot |R>r)$ does not convergence as $r\to\infty$ and $\bfX=(R\Theta,R(1-\Theta))$ is not regularly varying. 

Let $(X_{i1},X_{i2}) = (R_i\Theta_i,R_i(1-\Theta_i))$, $i=1,\ldots,n$, be iid observations from this distribution, where $n=20000$. Figures~\ref{fig5:subfig1}, \ref{fig5:subfig2} and \ref{fig5:subfig3} show the data in Cartesian and polar coordinates. We apply our threshold selection algorithm to the data, with the threshold upper quantile levels $q_k$ chosen as the 150 equidistant points between 0.01 and 0.2. The mean $p$-value $\overline{pv}_{k}$ is calculated using $m=60$ random subsamples of size $n_{k}=500\cdot q_k$ from the observations with $R_i>r_k$. This is shown in Figure~\ref{fig5:subfig4}.

In this model, the radial part $R$ is regularly varying, but $\Theta$ and $R$ are dependent given $R>r$ for any $r$.  We expect the mean $p$-values to be well below 0.5, as are observed. No threshold is selected by the algorithm.  This suggests that our technique can potentially be used to detect misspecified models from the regular variation assumption, especially in the scenario where the heavy-tailedness of $R$ is observed but dependence between $R$ and $\Theta$ is suspected.

\begin{figure}[t]
	\begin{subfigure}[]{
		\includegraphics[width=.31\textwidth] {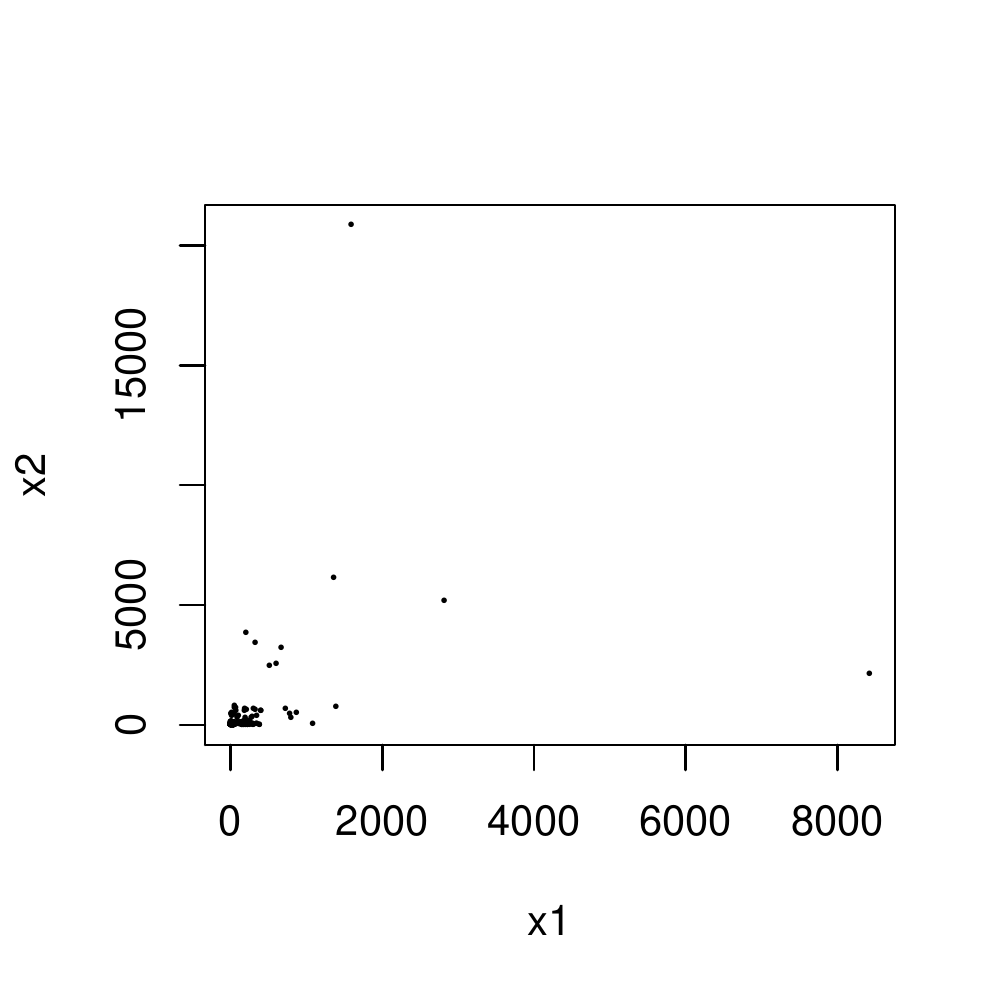}   
		\label{fig5:subfig1} }
	\end{subfigure}
	\begin{subfigure}[]{
		\includegraphics[width=.31\textwidth] {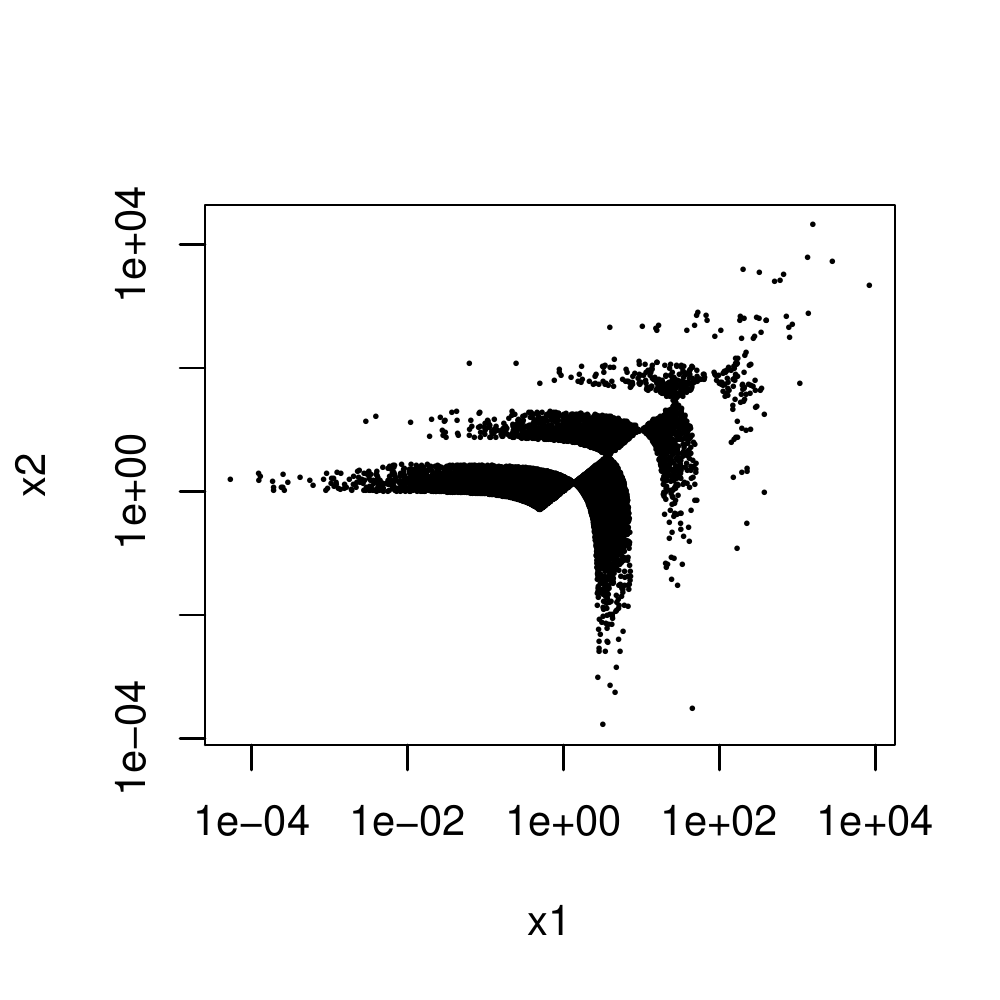}   
		\label{fig5:subfig2} }
	\end{subfigure}
	\begin{subfigure}[]{
		\includegraphics[width=.31\textwidth] {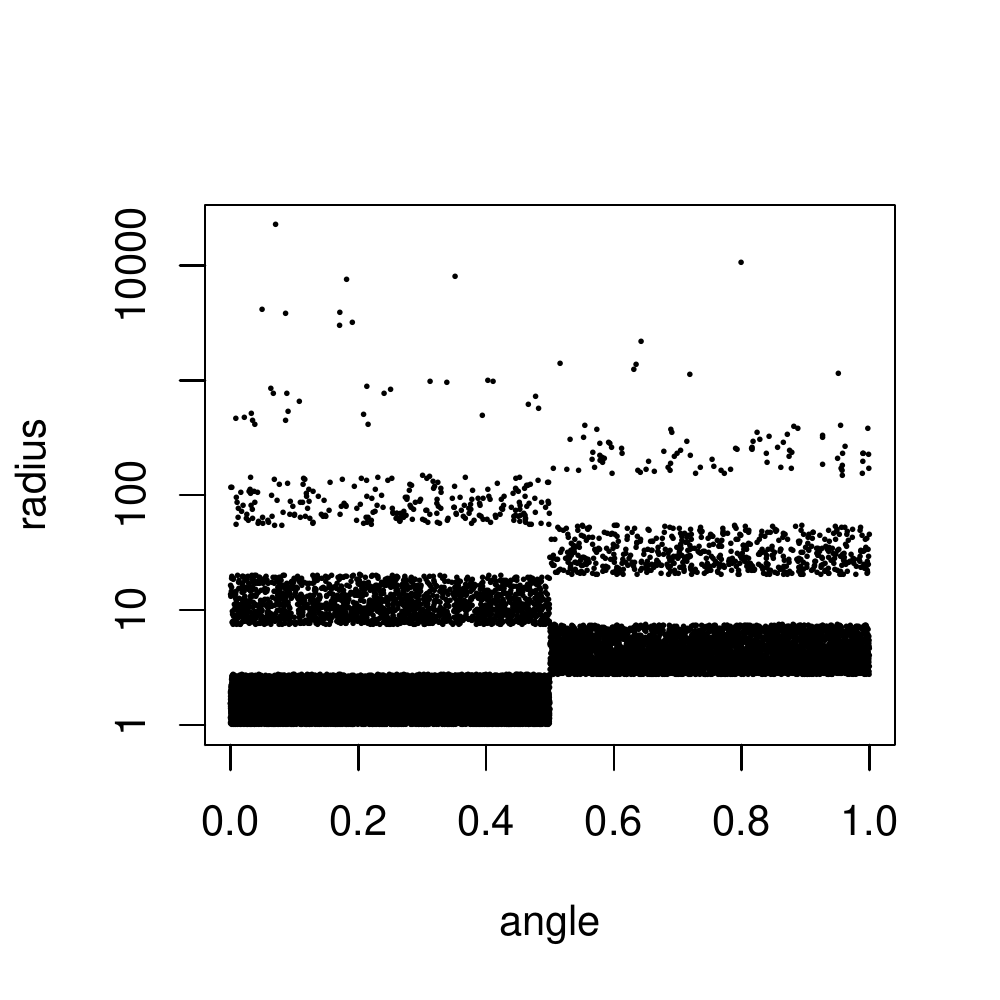}   
		\label{fig5:subfig3} }
	\end{subfigure}
	\begin{subfigure}[]{
		\includegraphics[width=.62\textwidth] {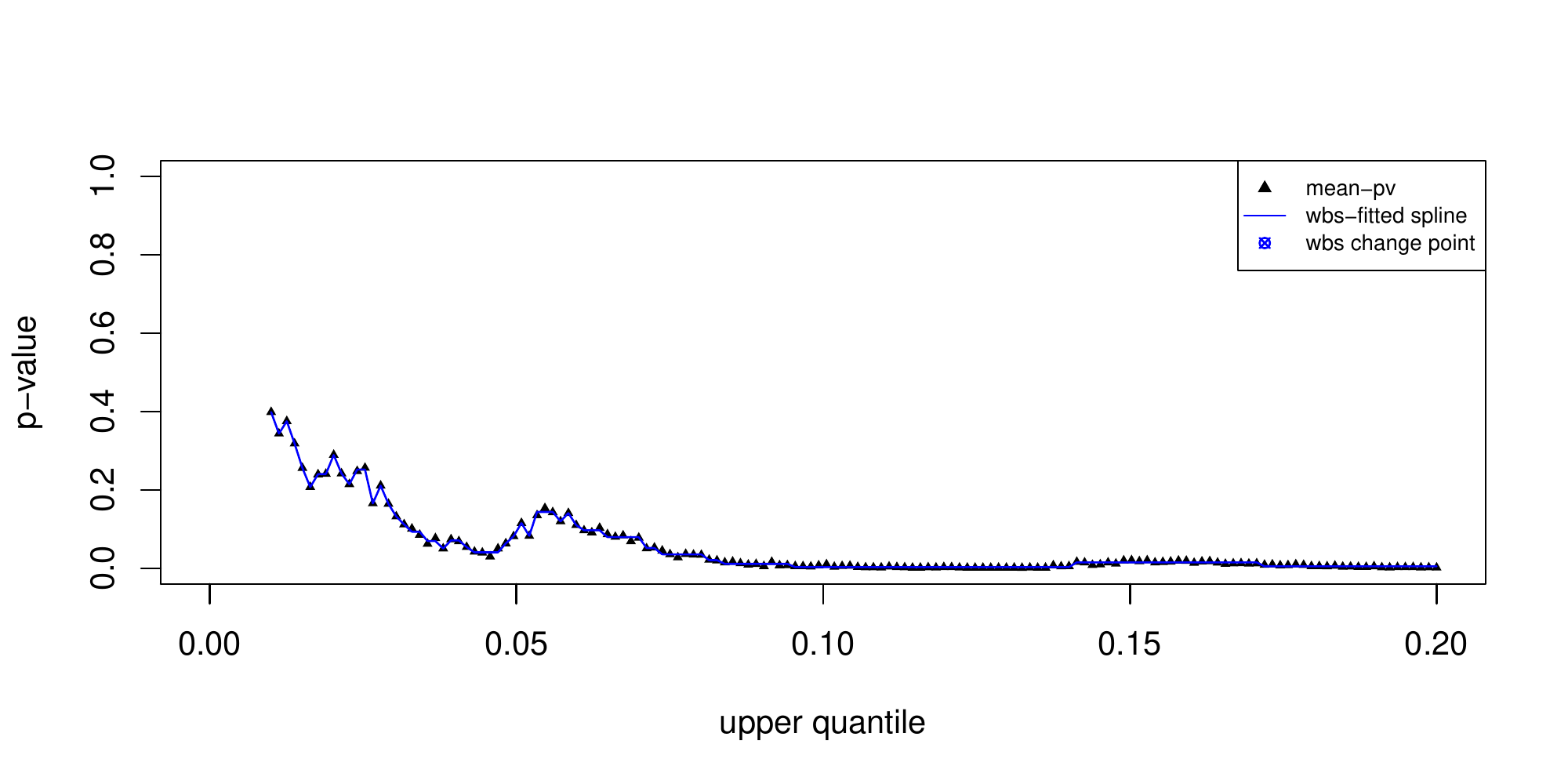}   
		\label{fig5:subfig4} }
	\end{subfigure}
	\label{fig:e5}
	\caption{Example \ref{ex:1}. (a) scatterplot of $(X_{i1},X_{i2})$; (b) scatterplot of $(X_{i1},X_{i2})$ in log-log scale; (c) scatterplot of $(R_i,\Theta_i)$; (d) mean $p$-value path (black triangles), fitted WBS spline (blue line), and the chosen threshold quantile (red vertical line).}
\end{figure}

%---------------------------------------------------------------------------------------------------------------------------------------------%
%--------------------------------------------  DISCUSSION --------------------------------------------%
%---------------------------------------------------------------------------------------------------------------------------------------------%

\section{Discussion}

In this paper, we propose a threshold selection procedure for multivariate regular variation, for which $R$ and $\Theta$ are approximately independent for $R$ beyond the threshold. While our problem is set in the multivariate heavy-tailed setting and we utilize distance covariance as our measure of dependence, our algorithm is essentially a change point detection method based on $p$-values generated through subsampling schemes. Hence this may be generalized to other problem settings and potentially incorporates other dependence measures. Though we have proposed an automatic selection for the threshold based on the fitted mean $p$-value path, we would like to emphasize that, like the Hill plot, this should be viewed as a visual tool rather than an optimal selection criterion. 
The final threshold should be based on the automatic procedure in conjunction with visual inspection of the $p$-value path.

We note that the choice of norm in the polar coordinate transformation \eqref{eq:polar} may result in significant differences in the choice of thresholds, which indicates the rate of convergence to the limit spectral density. This is especially evident in the near `asymptotic independence' case, where the mass of the angular distribution concentrates on the axes. 

As an illustration, we simulated iid observations $\{(X_{i1},X_{i2})\}_{i=1,\ldots,n}$ from the bivariate logistic distribution, where the cdf is given in \eqref{eq:logistic}, with $\gamma=0.95$ and $n=10000$. We apply the polar coordinate transformation with respect to the $L_p$-norm for $p=0.2,1,5$. {Note that in the case of $p=0.2$, $L_p$ is only a quasi-norm as it does not satisfy the triangular inequality.  However, it can be shown that \eqref{eq:ang:dist} holds and the limiting angular distribution exists for bivariate logistic distribution.} We compare the threshold selection results in Figure~\ref{fig:e41}. Note that in the cases of the $L_1$ and $L_5$-norms, the threshold levels are chosen to be upper $5\%$ and $12\%$, respectively, while in the case of the $L_{0.2}$-norm, it is not possible to select the threshold as the dependence between $R$ and $\Theta$ at all levels were shown to be significant. Indeed, this can be seen in Figure~\ref{fig:e42}, where we compare the histogram of $X_1^p/(X_1^p+X_2^p)$ given $\|X\|_p$ is large across three levels of truncations, $2\%$, $5\%$ and $12\%$, together with the theoretical limiting density curve. For the $L_{0.2}$-norm, the limiting angular density is poorly approximated by the truncated data for all levels. For the other two norms, the truncated observations according to the selected threshold provide decent approximations to the true limiting density of the angular component. {One possible explanation for this is that under the $L_{0.2}$-norm, the threshold is concave and hence observations on the diagonal are much easier to be classified as ``extremes'' than those near the axis.  As a result, the estimator of the angular density uses more observations near the diagonal, which may not be, in fact, close enough to the limit.} 
{This choice of norm} is an interesting topic and is the subject of ongoing research.

\begin{figure}[t]
	\includegraphics[width=.6\textwidth] {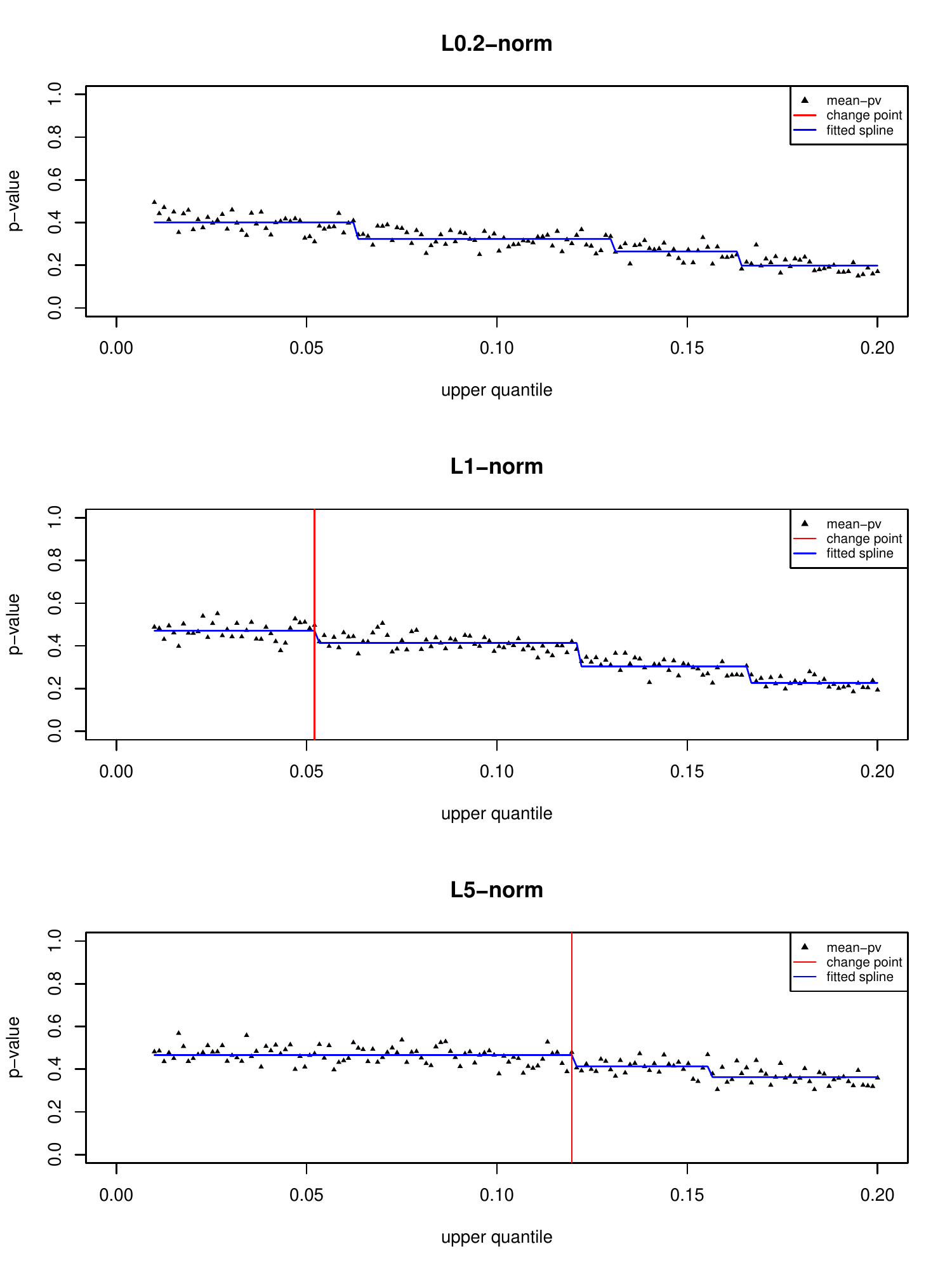}   
	\caption{Simulated logistic data of sample size $n=10000$ with $\gamma=0.95$. Threshold selection algorithm applied under the $L_{0.2}$-, $L_1$- and $L_5$-norms: mean $p$-value paths (black triangles), fitted WBS splines (blue lines) and the chosen threshold quantiles (red vertical line).}
	\label{fig:e41}
\end{figure}

\begin{figure}[t]
	\includegraphics[width=.7\textwidth] {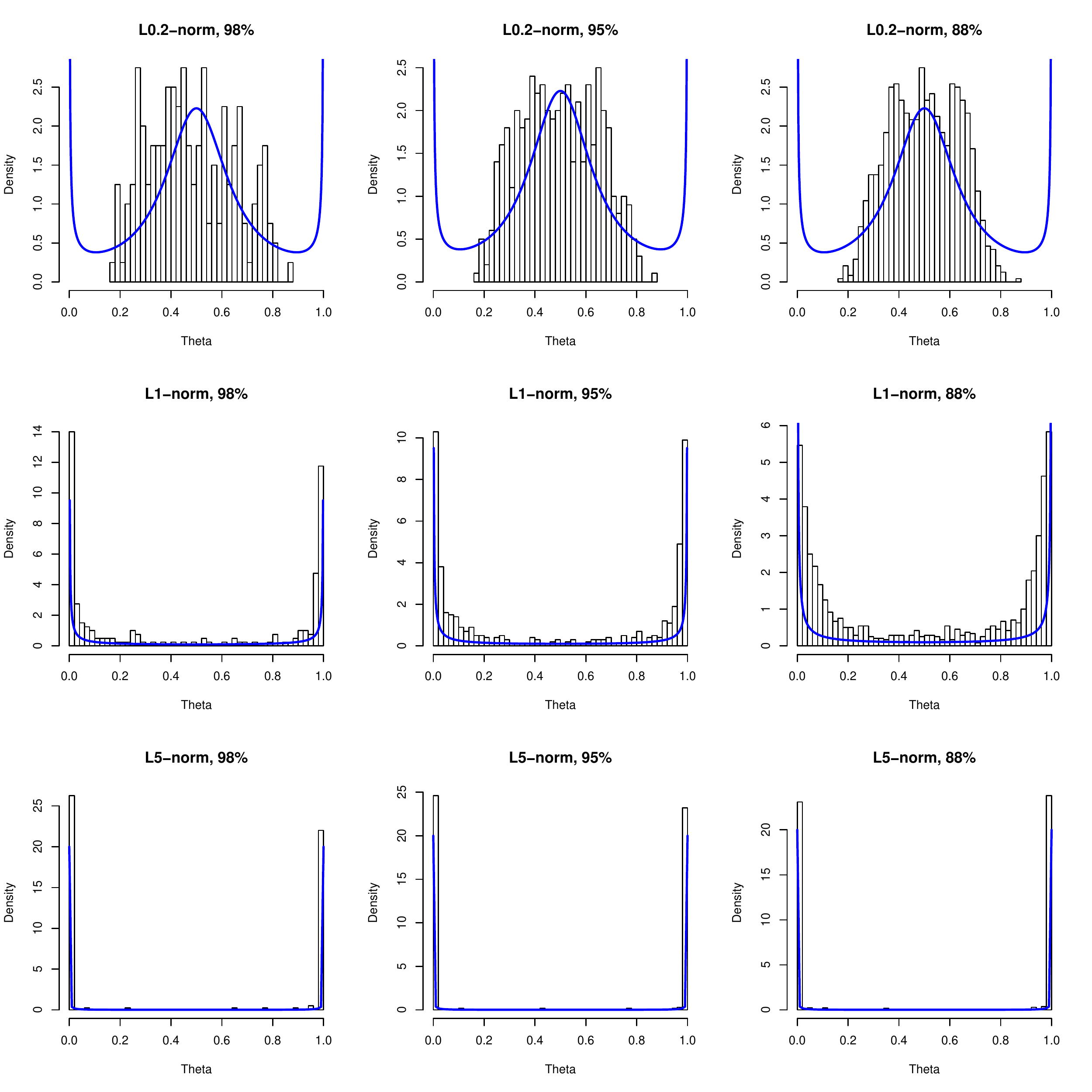}   
	\caption{Simulated logistic data of sample size $n=10000$ with $\gamma=0.95$. Histogram of $X_1^p/(X_1^p+X_2^p)$ for truncated levels $2\%$, $5\%$ and $12\%$ for $p=0.2,1,5$.}
	\label{fig:e42}
\end{figure}

%---------------------------------------------------------------------------------------------------------------------------------------------%
%--------------------------------------------  ACKNOWLEDGEMENT --------------------------------------------%
%---------------------------------------------------------------------------------------------------------------------------------------------%

\section*{Acknowledgement} \label{sec:acknow}

The foreign exchange rate data are obtained from OANDA from through R-package `qrmtools'.  We would like to thank Bodhisattva Sen for helpful discussions.  {We would also like to thank the editor and referees for their many constructive and insightful comments.}

%---------------------------------------------------------------------------------------------------------------------------------------------%
%--------------------------------------------  BIBLIOGRAPHY --------------------------------------------%
%---------------------------------------------------------------------------------------------------------------------------------------------%
%\bibliography{wanbib}

%\begin{thebibliography}{plainnat}
\bibliographystyle{plainnat}

\bibliography{wan}

\appendix

\section{Proof of Theorem~\ref{thm:1}} \label{app:1}

{Note from the definition of the empirical distance covariance in \eqref{eq:teststat}, the integrand can be expressed as} 
\beao
	C_n(s,t) 
		&=& \frac{1}{n\hat p_n}\sum_{j=1}^n e^{isR_j/r_n + it^T\mathbf\Theta_j}\mathbf{1}_{\{R_j>r_n\}} \\
		&& \quad - \frac{1}{n\hat p_n}\sum_{j=1}^n e^{isR_j/r_n}\mathbf{1}_{\{R_j>r_n\}} \frac{1}{n\hat p_n}\sum_{k=1}^n e^{ it^T\mathbf\Theta_k}\mathbf{1}_{\{R_k>r_n\}} \\
		&=& \frac{1}{n\hat p_n}\sum_{j=1}^n \left(e^{isR_j/r_n} 
			- \varphi_{{\frac{R}{r_n}}|r_n}(s)\right) \left(e^{it^T\mathbf\Theta_j} 
			- \varphi_{\Theta|r_n}(t)\right) \mathbf{1}_{\{R_j>r_n\}} \\
		&&-\,  \frac{1}{n\hat p_n}\sum_{j=1}^n \left(e^{isR_j/r_n} 
			- \varphi_{{\frac{R}{r_n}}|r_n}(s)\right)\mathbf{1}_{\{R_j>r_n\}} \  \frac{1}{n\hat p_n}\sum_{k=1}^n \left(e^{it^T\mathbf\Theta_k} 
			- \varphi_{\Theta|r_n}(t)\right) \mathbf{1}_{\{R_k>r_n\}}.
\eeao
Writing $ U_{jn} = \left(e^{isR_j/r_n} - \varphi_{{\frac{R}{r_n}}|r_n}(s)\right)\mathbf{1}_{\{R_j>r_n\}}$, $V_{jn} =  \left(e^{it^T\mathbf\Theta_j} - \varphi_{\Theta|r_n}(t)\right) \mathbf{1}_{\{R_j>r_n\}}$, we have
$$
	C_n(s,t) = \frac{p_n}{\hat{p}_n}\,\frac{1}{n} \sum_{j=1}^n \frac{U_{jn}V_{jn}}{p_n} -\left(\frac{p_n}{\hat{p}_n}\right)^2\,\frac{1}{n}\sum_{j=1}^n \frac{U_{jn}}{p_n} \,\frac{1}{n} \sum_{k=1}^n \frac{V_{kn}}{p_n}.
$$
Since $\E U_{jn} = \E V_{jn} = 0$ and $\E U_{jn}V_{jn}/p_n = \varphi_{{\frac{R}{r_n}},\Theta|r_n}(s,t)- \varphi_{{\frac{R}{r_n}}|r_n}(s)\varphi_{\Theta|r_n}(t)$, it is convenient to mean correct the summands and obtain
\beao
	C_n(s,t) 
		&=&  \frac{p_n}{\hat{p}_n}\,\frac{1}{n} \sum_{j=1}^n \left(\frac{U_{jn}V_{jn}}{p_n} 
			- \left(\varphi_{{\frac{R}{r_n}},\Theta|r_n}(s,t)
			- \varphi_{{\frac{R}{r_n}}|r_n}(s)\varphi_{\Theta|r_n}(t)\right)\right)
			-\left(\frac{p_n}{\hat{p}_n}\right)^2\,\frac{1}{n}\sum_{j=1}^n \frac{U_{jn}}{p_n} \,\frac{1}{n} \sum_{k=1}^n \frac{V_{kn}}{p_n} \nonumber \\
		&& +  \frac{p_n}{\hat{p}_n}\,(\varphi_{{\frac{R}{r_n}},\Theta|r_n}(s,t)
			- \varphi_{{\frac{R}{r_n}}|r_n}(s)\varphi_{\Theta|r_n}(t))  \nonumber \\
		&=:&  \left(\frac{p_n}{\hat{p}_n} \right) \tilde{E}_1 
			- \left(\frac{p_n}{\hat{p}_n} \right)^2 \tilde{E}_{21} \tilde{E}_{22} 
			+ \left(\frac{p_n}{\hat{p}_n} \right) \tilde{E}_3 \nonumber\\
		&=:&  \left(\frac{p_n}{\hat{p}_n} \right) \tilde{E}_1 
			- \left(\frac{p_n}{\hat{p}_n} \right)^2 \tilde{E}_{2} 
			+ \left(\frac{p_n}{\hat{p}_n} \right) \tilde{E}_3 \label{eq:4}
\eeao
Note that $ \tilde{E}_1, \tilde{E}_{21}, \tilde{E}_{22}$ are averages of iid zero-mean random variables and $ \tilde{E}_3$ is non-random. We first prove the second part of Theorem \ref{thm:1}. The first part of Theorem \ref{thm:1} follows easily in a similar fashion.

%--------------------------------------------  PROOF THM 1(2)  --------------------------------------------%

\begin{proof}[Proof of Theorem \ref{thm:1}(2)]

In order to show \eqref{eq:thm2}, it suffices to establish that
\beqq \label{eq:toshow1}
	n\hat p_n  \int_{\bbr^{d+1}} \left(\frac{p_n}{\hat{p}_n} \right)^2 |\tilde{E}_1|^2 \mu(ds,dt) \cid \int_{\bbr^{d+1}} |Q(s,t)|^2\mu(ds,dt),
\eeqq
and
\beqq\label{eq:toshow2}
	\left|n\hat p_nT_n - n\hat p_n  \int_{\bbr^{d+1}}  \left(\frac{p_n}{\hat{p}_n} \right)^2 |\tilde{E}_1|^2\mu(ds,dt)\right| 
%	\le n\hat p_n  \int_{\bbr^{d+1}} |E_2|^2\mu(ds,dt) + n\hat p_n  \int_{\bbr^{d+1}} |E_3|^2\mu(ds,dt) 
	\cip 0,
\eeqq
where \eqref{eq:toshow2} can be implied by 
\beqq \label{eq:toshow3}
	n\hat p_n  \int_{\bbr^{d+1}}  \left(\frac{p_n}{\hat{p}_n} \right)^2|\tilde E_2|^2\mu(ds,dt) + n\hat p_n  \int_{\bbr^{d+1}}  \left(\frac{p_n}{\hat{p}_n} \right)^2|\tilde E_3|^2\mu(ds,dt) \cip 0.
\eeqq
Notice that
\beao
	\E \left|\frac{\hat p_n}{{p}_n}-1\right|^2 = \E\left|\frac{1}{n}  \sum_{j=1}^n\left(\frac{ \mathbf{1}_{\{R_j>r_n\}}}{p_n} -1\right)\right|^2 = \frac{1}{n} \E \left|\frac{ \mathbf{1}_{\{R_1>r_n\}}}{p_n} -1\right|^2 \le \frac{1}{np_n} O(1) + \frac{1}{n} O(1) \to 0.
\eeao
Hence $\hat p_n/p_n \cip 1$ and for \eqref{eq:toshow1} and \eqref{eq:toshow3}, it is equivalent to prove that
\beqq\label{eq:toshow4}
	n p_n  \int_{\bbr^{d+1}} |\tilde E_1|^2\mu(ds,dt) \cid \int_{\bbr^{d+1}} |Q(s,t)|^2\mu(ds,dt)
\eeqq
and
\beqq \label{eq:toshow5}
	n p_n  \int_{\bbr^{d+1}}  |\tilde E_2|^2\mu(ds,dt) + n p_n  \int_{\bbr^{d+1}}  |\tilde E_3|^2\mu(ds,dt) \cip 0.
\eeqq
We will show the convergence \eqref{eq:toshow4} in Proposition~\ref{prop:E1}.
By \eqref{eq:cond},
$$
	n p_n  \int_{\bbr^{d+1}} |\tilde E_3|^2\mu(ds,dt) \to 0.
$$
So that \eqref{eq:toshow5} holds provided
\beqq \label{eq:toshow6}
	 n p_n  \int_{\bbr^{d+1}} |\tilde E_2|^2\mu(ds,dt)\cip 0,
\eeqq
which follows in a similar fashion as Proposition~\ref{prop:E1}.

\epf

%---------------------------------------------------------------------------------------------------------%
%
%Recall from \eqref{eq:3} that as $n\to\infty$, $R/r_n$ and $\mathbf\Theta$ become asymptotically independent and converge to $\nu_\alpha$ and $S$ respectively. Denote the characteristic functions of the corresponding limit distributions by 
%\beam
%	\varphi_R(s) &:=& {\int_1^\infty} \exp(is r) \alpha r^{-\alpha-1}dr = \lim_{n\to\infty} \varphi_{R|r_n}(s) \label{eq:phi:r} \\
%	\varphi_\Theta(t) &:=& \int_{\bbr^{d}}  \exp(it\theta) S(d\theta) = \lim_{n\to\infty} \varphi_{\Theta|r_n}(t).\label{eq:phi:theta}
%\eeam

%--------------------------------------------  PROP - E1 --------------------------------------------%

\begin{proposition} \label{prop:E1}

Assume {$\mu$ satisfies 
$$
	\int_{\bbr^{d+1}} (1\wedge |s|^{\beta})(1\wedge |t|^{2}) \mu(ds,dt) <\infty,
$$
and} that $np_n\to\infty$ as $n\to\infty$, then
\beqo \label{eq:6}
	n p_n  \int_{\bbr^{d+1}} |\tilde E_1|^2\mu(ds,dt) \cid \int_{\bbr^{d+1}} |Q(s,t)|^2\mu(ds,dt),
\eeqo
where $Q$ is a centered Gaussian process with covariance function {\eqref{eq:Qcov}}.

\end{proposition}

\begin{proof}[Proof of Proposition~\ref{prop:E1}]

We first show that
\beqq \label{eq:11}
	\sqrt{n p_n} \tilde E_1 \overset{d}\to Q(s,t), \quad \mbox{on } \mathcal{C}(\bbr^{d+1})
\eeqq
which can be implied by the finite distributional convergence of $\sqrt{n p_n} \tilde E_1(s,t)$ and its tightness on $\mathcal{C}(\bbr^{d+1})$. 

Write
$$
	\sqrt{n p_n} \tilde E_1 = \frac{1}{\sqrt{n}} \sum_{j=1}^n \left(\frac{U_{jn}V_{jn}}{\sqrt{p_n}} - \sqrt{p_n} (\varphi_{{\frac{R}{r_n}},\Theta|r_n}(s,t)- \varphi_{{\frac{R}{r_n}}|r_n}(s)\varphi_{\Theta|r_n}(t))\right) =: \frac{1}{\sqrt n} \sum_{j=1}^n Y_{jn},
$$
where $Y_{jn}$'s are iid random variables with mean 0. For fixed $(s,t)$, note that
$$
\var(Y_{1n}) = \E|Y_{1n}|^2 \,=\, \frac{\E|U_{1n}V_{1n}|^2}{p_n} (1+o(1)) \,=\, \frac{\E\mathbf{1}_{\{R_1>r_n\}}}{p_n}O(1) \, < \infty.
$$
On the other hand, any $\delta>0$, 
$$
\E|Y_{1n}|^{2+\delta}  \,=\, \frac{\E|U_{1n}V_{1n}|^{2+\delta}}{p_n^{1+\delta/2}} (1+o(1)) \,\le\, c\frac{\E\mathbf{1}_{\{R_1>r_n\}}}{p_n^{1+\delta/2}}(1+o(1)) = O(p_n^{-\delta/2})
$$
Then we can apply the central limit theorem for triangular arrays by checking the Lyapounov condition (see, e.g., \cite{billingsley:1995}) for the $Y_{jn}$'s:
$$
	\frac{ \sum_{j=1}^n \E|Y_{jn}|^{2+\delta}}{\left(\var\left( \sum_{j=1}^n Y_{jn}\right)\right)^\frac{2+\delta}{2}} = \frac{O(np_n^{-\frac{\delta}{2}})}{{n^{1+\frac{\delta}{2}}\var(Y_{1n})^{1+\frac{\delta}{2}}}} = O((np_n)^{-\frac{\delta}{2}} ) \to 0.
$$
It follows easily that for fixed $(s,t)$,
$$
	\sqrt{n p_n} \tilde E_1 \cid Q(s,t).
$$
The finite-dimensional distribution can be obtained using the Cram\'er-Wold device and the covariance function can be verified through calculations. 

We now show the tightness of $\sqrt{np_n} \tilde E_{1}$.  Note that
\beao
	\tilde E_{1}(s,t) &=& \frac{1}{n}\sum_{j=1}^n \frac{\left(e^{isR_j/r_n} 
			- \varphi_{{\frac{R}{r_n}}|r_n}(s)\right) \left(e^{it^T\mathbf\Theta_j} 
			- \varphi_{\Theta|r_n}(t)\right) \mathbf{1}_{\{R_j>r_n\}}}{p_n} \\
			&& \quad- \left(\varphi_{{\frac{R}{r_n}},\Theta|r_n}(s,t)- \varphi_{{\frac{R}{r_n}}|r_n}(s)\varphi_{\Theta|r_n}(t)\right) \\
			&=& \left(\frac{1}{n}\sum_{j=1}^n \frac{e^{isR_j/r_n+it^T\mathbf\Theta_j}  \mathbf{1}_{\{R_j>r_n\}}}{p_n} - \varphi_{{\frac{R}{r_n}},\Theta|r_n}(s,t)\right) \\
			&& \quad - \left(\frac{1}{n}\sum_{j=1}^n \frac{e^{isR_j/r_n}  \mathbf{1}_{\{R_j>r_n\}}}{p_n} - \varphi_{{\frac{R}{r_n}}|r_n}(s)\right)\varphi_{\Theta|r_n}(t) \\
			&& \qquad - \left(\frac{1}{n}\sum_{j=1}^n \frac{e^{it^T\mathbf\Theta_j}  \mathbf{1}_{\{R_j>r_n\}}}{p_n} - \varphi_{\Theta|r_n}(t)\frac{\hat p_n}{p_n}\right) \varphi_{{\frac{R}{r_n}}|r_n}(s) \\
			&=:& \tilde E_{11} + \tilde E_{12} + \tilde E_{13}.
\eeao
Without loss of generality, we show the tightness for $\sqrt{np_n} \tilde E_{11}$ and that of $\sqrt{np_n} \tilde E_{12}$ and $\sqrt{np_n} \tilde E_{13}$ follows from the same argument.

First we introduce some notation following that from \cite{bickel:wichura:1971}.  Fix $(s,t),(s',t') \in \bbr^{d+1}$ where $s<s'$ and $t<t'$.  Let $B$ be the subset of $\bbr^{d+1}$ of the form
$$
	B:= \left((s,t),(s',t')\right] = (s,s'] \times \prod_{k=1}^d (t_k,t_k'] \subset \bbr^{d+1}.
$$
For ease of notation, we suppress the dependence of $B$ on $(s,t), (s',t')$.
Define the increment of the stochastic process $\revise{\tilde E_{11}}$ on $B$ to be
\beao
	\tilde E_{11}(B) &:=&\frac{1}{n}\sum_{j=1}^n \sum_{z_0=0,1} \sum_{z_1=0,1} \cdots \sum_{z_d=0,1} (-1)^{d+1-\sum_j z_j} \\
	&& \qquad\tilde E_{11}\left(s+z_0(s'-s), t_1+z_1(t_1'-t_1) ,\ldots, t_d+z_d(t_d'-t_d)\right).
\eeao
From a sufficient condition of Theorem~3 of \cite{bickel:wichura:1971}, the tightness of $\sqrt{np_n} \tilde E_{1}$ is implied if the following statement holds for any $(s,t),(s',t')$ and corresponding $B$,
\beqo
	 \E|\sqrt{np_n}\tilde E_{11}(B)|^2 \le c|s-s'|^\beta \prod_{k=1}^d|t_k-t_k'|^\beta, \quad \text{for some }\beta>1.
\eeqo
It follows that
\beam
	&&\E \left| \sqrt{n p_n} \left(\tilde E_{11}(B)\right) \right|^2 \nonumber\\
	&=& np_n \E\left| \sum_{z_0=0,1}\cdots \sum_{z_d=0,1}(-1)^{d+1-\sum_j {z_j}}\frac{1}{n}\sum_{j=1}^n  e^{i(s+z_0(s'-s)){R_j}/r}\prod_{k=1}^d e^{i(t_k+z_k(t'_k-t_k))\Theta_{jk}}\frac{ \mathbf{1}_{\{R_j>r_n\}}}{p_n} \right.\nonumber\\
	&& \quad\quad\quad \left.- \sum_{z_0=0,1}\cdots \sum_{z_d=0,1}(-1)^{d+1-\sum_j {z_j}}\E \left[\left(e^{i(s+z_0(s'-s))R/r}\right)\prod_{k=1}^d e^{i(t_k+z_k(t'_k-t_k))\Theta_{k}}\frac{ \mathbf{1}_{\{R>r_n\}}}{p_n} \right]\right|^2\nonumber\\
	&=& np_n \E\left| \frac{1}{n}\sum_{j=1}^n (e^{isR_j/r_j} - e^{is'R_j/r_j})\prod_{k=1}^d (e^{it_k\Theta_{jk}} - e^{it'_k\Theta_{jk}})\frac{ \mathbf{1}_{\{R_j>r_n\}}}{p_n} \right.\nonumber\\
	&& \quad\quad\quad \left.- \E \left[(e^{isR/r} - e^{is'R/r})\prod_{k=1}^d (e^{it_k\Theta_k} - e^{it'_k\Theta_k})\frac{ \mathbf{1}_{\{R>r_n\}}}{p_n}\right]\right|^2 \nonumber\\
	&=& p_n \var\left( (e^{isR/r_j} - e^{is'R/r_j})\prod_{k=1}^d (e^{it_k\Theta_{k}} - e^{it'_k\Theta_{k}})\frac{ \mathbf{1}_{\{R>r_n\}}}{p_n} \right) \label{eq:iidsum}\\
	&\le& \E \left[\left|(e^{isR/r} - e^{is'R/r})\prod_{k=1}^d (e^{it_k\Theta_{k}} - e^{it'_k\Theta_{k}})\right|^2\middle| R>r_n\right].\nonumber
\eeam
From a Taylor series argument,
$$
	|e^{ix} - e^{ix'}|^2 \le c1\wedge|x-x'|^2 \le c1\wedge|x-x'|^\beta \le c|x-x'|^\beta, \quad \text{for any $\beta\in(0,2]$. }
$$
Hence for any $\beta \in (1,2\wedge\alpha)$,
\beao
	\E \left| \sqrt{n p_n} \tilde E_{11}(B) \right|^2 &\le& c|s-s'|^\beta \prod_{k=1}^d|t_i-t_i'|^\beta \E\left[(R/r_n)^\beta \prod_{k=1}^d|\Theta_k|^\beta|R>r_n\right] \\
	&<& c|s-s'|^\beta \prod_{k=1}^d|t_i-t_i'|^\beta,
\eeao
since $|\Theta_k|^\beta$'s are bounded and ${\sup_n}\E[(R/r_n)^\beta|R>r_n]<\infty$ by the regular variation assumption. This proves the tightness.

Define the bounded set
$$
	K_\delta = \{(s,t)|\ \delta < |s|<1/\delta,\delta <|t| <1/\delta\}, \quad \text{for $\delta<.5$. }
$$ 
Then, using \eqref{eq:11}, we have from the continuous mapping theorem,
\beqq \label{eq:bdconv}
	np_n\int_{K_\delta} |\tilde E_1|^2 \mu(ds,dt) \cid \int_{K_\delta} |Q(s,t)|^2 \mu(ds,dt).
\eeqq
On the other hand, for any $\beta<2\wedge\alpha$, we have
\beam
	&&\E|\sqrt{n p_n} \tilde E_1|^2 \nonumber\\
	&=& np_n \E \left|\frac{1}{n} \sum_{j=1}^n \left(\frac{U_{jn}V_{jn}}{{p_n}} - \E\left[\frac{U_{jn}V_{jn}}{{p_n}}\right]\right)\right|^2 \nonumber\\
	&\le& \frac{\E|U_{jn}V_{jn} - \E U_{jn}V_{jn}|^2}{p_n} \nonumber\\
	&\le& \frac{c\, \E|U_{jn}V_{jn}|^2}{p_n} \label{eq:unvn:bd}\\
	&= & \frac{c\,\E\left[\left|e^{isR_j/r_n} - \varphi_{{\frac{R}{r_n}}|r_n}(s)\right|^2\left|e^{it^T\mathbf\Theta_j} - \varphi_{\Theta|r_n}(t) \right|^2 \mathbf{1}_{\{R_j>r_n\}}\right]}{p_n} \nonumber\\
	&\le & 
	{
	\frac{c\,\E\left[\left(\left|e^{isR_j/r_n} -1\right|^2+\left| \varphi_{{\frac{R}{r_n}}|r_n}(s)-1 \right|^2\right) \left( \left|e^{it^T\mathbf\Theta_j}-1\right|^2 + \left|\varphi_{\Theta|r_n}(t)-1 \right|^2\right) \mathbf{1}_{\{R_j>r_n\}}\right]}{p_n} 
	}\nonumber\\
	&\le & 
	{
	 \frac{c\,\E\left[\left(1 \wedge \left|\frac{sR_j}{r_n}\right|^2 + \E\left[1 \wedge \left|\frac{sR_j}{r_n}\right|^2 |\frac{R}{r_n}>1\right] \right)\left(1\wedge|t\mathbf\Theta_j|^2 + \E\left[1\wedge|t\mathbf\Theta_j|^2|\frac{R}{r_n}>1\right]\right) \mathbf{1}_{\{R_j>r_n\}}\right]}{p_n} 
	}\nonumber\\
	&\le & 
	{
	 \frac{c\,\E\left[\left(1 \wedge \left|\frac{sR_j}{r_n}\right|^\beta + \E\left[1 \wedge \left|\frac{sR_j}{r_n}\right|^\beta |\frac{R}{r_n}>1\right] \right)\left(1\wedge|t\mathbf\Theta_j|^2 + \E\left[1\wedge|t\mathbf\Theta_j|^2|\frac{R}{r_n}>1\right]\right) \mathbf{1}_{\{R_j>r_n\}}\right]}{p_n} 
	}\nonumber\\
	&\le & \frac{c\,\E\left[\left(1 \wedge |s|^\beta\right) \left(\left|\frac{R_j}{r_n}\right|^\beta + \E\left[\left|\frac{R}{r_n}\right|^\beta |\frac{R}{r_n}>1\right] \right)\left(1\wedge|t|^2\right) \mathbf{1}_{\{R_j>r_n\}}\right]}{p_n} \nonumber\\
	&\le& c\,\E\left[\left(1 \wedge |s|^\beta (|R_j/r_n|^\beta + \E[|R/r_n|^\beta |R>r_n]) \right)\left(1\wedge|t|^2\right) |R>r_n\right]\nonumber \\
	&\le& c(1\wedge |s|^{\beta})(1\wedge |t|^{2}).\nonumber
\eeam
Therefore for any $\epsilon>0$,
\beao
	\lim_{\delta\to 0}\limsup_{n\to\infty} \P\left[np_n\int_{K^c_\delta} |\tilde E_1|^2 \mu(ds,dt) >\epsilon\right] 
		&\le& \frac{1}{\epsilon} \lim_{\delta\to 0}\limsup_{n\to\infty} \int_{K^c_\delta} \E|\sqrt{np_n}\tilde E_1|^2 \mu(ds,dt) \\
		&\le& \frac{1}{\epsilon} \lim_{\delta\to 0}\limsup_{n\to\infty} \int_{K^c_\delta} c (1\wedge |s|^{\beta})(1\wedge |t|^{2}) \mu(ds,dt) \to 0
\eeao
by the dominated convergence theorem. This combined with \eqref{eq:bdconv} shows the convergence of $ np_n\int |\tilde E_1|^2 \mu(ds,dt)$ to $\int |Q(s,t)|^2 \mu(ds,dt)$, and hence completes the proof of the proposition.
\qed
\end{proof}

%---------------------------------------------------------------------------------------------------------%

\begin{proof}[Proof of Theorem \ref{thm:1}(2) (cont.)]

Now it remains to show \eqref{eq:toshow6}. Similar to the proof of Proposition~\ref{prop:E1}, we can show that
$$
	\sqrt{np_n} \tilde E_{21} \cid Q'
$$
for a centered Gaussian process $Q'$, and
$$
	\tilde E_{22} \cip 0.
$$
Hence
$$
	\sqrt{np_n} \tilde E_{2} = \sqrt{np_n} \tilde E_{21}\tilde E_{22} \cip 0.
$$
The argument then follows similarly from the continuous mapping theorem and bounding the tail integrals.
\qed
\end{proof}

%--------------------------------------------  PROOF THM 1(2)  --------------------------------------------%

\begin{proof}[Proof of Theorem \ref{thm:1}(1)]

Similar to the proof of Theorem \ref{thm:1}(2), it suffices to show that 
\beqq \label{eq:13}
	 \int|\tilde E_i|^2\mu(ds,dt) \cip 0,\quad i=1,2,3.
\eeqq
The convergence \eqref{eq:13} for $i=1,2$ follows trivially from the more general results \eqref{eq:toshow4} and \eqref{eq:toshow6} in the proof of Theorem \ref{thm:1}(2). Hence it suffices to show
\beqq \label{eq:15}
	\int|\tilde E_3|^2\mu(ds,dt) \to 0,
\eeqq
where we recall that $\tilde E_3:=\varphi_{{\frac{R}{r_n}},\Theta|r_n}(s,t)- \varphi_{{\frac{R}{r_n}}|r_n}(s)\varphi_{\Theta|r_n}(t)$ is non-random.  

Let $P_{{\frac{R}{r_n}},\Theta|r_n}(\cdot) = P\left[\left(\frac{R}{r_n},\mathbf\Theta\right) \in \cdot | \frac{R}{r_n} > 1 \right] $ and $P_{{\frac{R}{r_n}}|r_n},P_{\Theta|r_n}$ be the corresponding marginal measures. Then from \eqref{eq:3},
$$
	P_{{\frac{R}{r_n}},\Theta|r_n} - P_{{\frac{R}{r_n}}|r_n}P_{\Theta|r_n} {\,\civ\, \nu_\alpha \times S-\nu_\alpha \times S = 0},
$$
and hence for fixed $(s,t)$,
$$
	\tilde E_3(s,t) = \int e^{is{T} + it^T\mathbf\Theta} \,(P_{{\frac{R}{r_n}},\Theta|r_n} - P_{{\frac{R}{r_n}}|r_n}P_{\Theta|r_n})(d{T},d\mathbf\Theta) \to 0.
$$
{For any $\beta < 2 \wedge\alpha$, using the same argument in \eqref{eq:unvn:bd}, }
$$
	|\tilde E_3|^2 = \left(\frac{\E|U_{jn}V_{jn}|}{p_n}\right)^2 \le c(1\wedge |s|^{\beta})(1\wedge |t|^{2}).
$$
Then \eqref{eq:15} follows from \eqref{eq:weight} and dominated convergence.
This concludes the proof. \qed

\end{proof}

%----------------------- MIXING ----------------------%

\section{Proof of Theorem~\ref{thm:2}}\label{app:2}

Following the same notation and steps as the proof of Theorem~\ref{thm:1} in Appendix~\ref{app:1}, it suffices to prove the following convergences for the mixing case:
\beqq\label{eq:ap2:1}
	\frac{\hat p_n}{p_n} \cip1, 
\eeqq
\beqq\label{eq:ap2:2}
	n p_n  \int_{\bbr^{d+1}} |\tilde E_1|^2\mu(ds,dt) \cid \int_{\bbr^{d+1}} |Q'(s,t)|^2\mu(ds,dt)
\eeqq
and
\beqq \label{eq:ap2:3}
	 n p_n  \int_{\bbr^{d+1}} |\tilde E_2|^2\mu(ds,dt)\cip 0.
\eeqq
We prove \eqref{eq:ap2:1} and \eqref{eq:ap2:2} in Propositions~\ref{prop:ap2:1} and \ref{prop:ap2:2}, respectively. The proof of \eqref{eq:ap2:3} follows in a similar fashion.   The proofs of both propositions rely on the following lemma. 

\revise{
Throughout this proof we make use of the results that  if $\{Z_t\}$ is stationary and $\alpha$-mixing with coefficient $\{\alpha_h\}$, then
\beqq \label{eq:alpha:ineq}
	|\cov(Z_0,Z_h)| \le c \alpha_h^{\delta} \left(\E|Z_0|^{2/(1-\delta)}\right)^{1-\delta}, \quad \text{for any $\delta \in (0,1)$},
\eeqq
%$$
%	|\cov(Z_0,Z_h)| \le 8 \alpha_h^{1/r} \|Z_0\|_p \|Z_h\|_q, \quad \text{for any $p,q,r\ge0$ and $\frac1p+\frac1q+\frac1r=1$}.
%$$
see Section 1.2.2, Theorem 3(a) of \cite{doukhan:1994}.
}

\begin{lemma} \label{lemma:fancy}
Let $\{\bfX_t\}$ be a multivariate stationary time series that is regularly varying and $\alpha$-mixing with mixing coefficient $\{\alpha_h\}$.  For a sequence $r_n\to\infty$, set $p_n=\P(\|\bfX_0\|>r_n)$.  %Assume $p_n=o(n^{1/3})$ and that there exists a sequence $\{l_n\}$ such that $l_n\to\infty$, $l_np_n\to0$, and\\
%i)
%	\beqq \label{eq:lm:cond:1}
%		\left(\frac1{p_n}\right)^\delta \sum_{h=l_n}^\infty \alpha_h^\delta \to 0
%	\eeqq
%	for some $\delta\in(0,1)$; \\
%ii)
%	\beqq \label{eq:lm:cond:2}
%		\lim_{h\to\infty} \limsup_{n\to\infty} \frac{1}{p_n} \sum_{j=h}^{l_n} \P(\|\bfX_1\|>r_n, \|\bfX_j\|>r_n) =0;
%	\eeqq \\
%iii) 
%	\beqq \label{eq:lm:cond:3}
%		np_n \alpha_{l_n} \to 0.
%	\eeqq
Let $f_1,f_2$ be bounded functions \revise{which vanish outside $\overline{\bbr}^d\backslash B_1(\bf0)$}, where $B_1(\bf0)$ is the unit open ball $\{\bfx|\|\bfx\|<1\}$, with sets of discontinuity of measure zero. Set,
$$
	S_n^{(i)} = \sum_{t=1}^n \left( f_i\left( \frac{\bfX_t}{r_n}\right) - \E f_i\left(\frac{\bfX_0}{r_n}\right) \right), \quad i=1,2.
$$
Assume that condition~\hyperref[cond:m]{({\bf M})} holds for $\{\alpha_h\}$ and $\{r_n\}$.
Then
\beqq \label{eq:lm:result}
	\frac{1}{\sqrt{np_n}} (S_n^{(1)},S_n^{(2)})^T \cid N(\bf0,\Sigma),
\eeqq
where the covariance matrix $[\Sigma_{ij}]_{i,j=1,2} = [\sigma^2(f_i,f_j)]_{i,j=1,2}$ with
\beqq \label{eq:covcal:1}
	\sigma^2(f_1,f_2):= \sigma^2_0(f_1,f_2) + 2 \sum_{h=1}^\infty \sigma_h^2(f_1,f_2)
\eeqq
and
\beqq \label{eq:covcal:2}
	\sigma_h^2(f_1,f_2) = \int f_1f_2 d\mu_h,\quad h\ge0.
\eeqq
In particular,
$$
	\frac{1}{{np_n}} (S_n^{(1)},S_n^{(2)})^T \cip \bf0.
$$
\end{lemma}

The proof of Lemma~\ref{lemma:fancy} is provided after the proofs of the propositions.

\bpr \label{prop:ap2:1}
	Assume that condition~\hyperref[cond:m]{({\bf M})} holds, then 
	$$
		\frac{\hat p_n}{p_n} \cip1, 
	$$
\epr

\bpf
We have
$$
	\frac{\hat p_n}{{p}_n}-1 = \frac{1}{n}  \sum_{j=1}^n\left(\frac{ \mathbf{1}_{\{R_j>r_n\}}}{p_n} -1\right) 
	= \frac{1}{np_n}  \sum_{j=1}^n ( \mathbf{1}_{\{R_j>r_n\}} - p_n).
$$
Apply Lemma~\ref{lemma:fancy} to $f(\bfx) =  \mathbf{1}_{\{\|\bfx\|>1\}}$ and the result follows.\qed
%where \eqref{eq:p1:1} follows from the $\alpha$-mixing condition {with the convention that $\infty^0=1$}, and \eqref{eq:p1:2} follows from the facts that $2/(1-\delta)>2$ and $| \mathbf{1}_{\{R_1>r_n\}}- p_n |<1$.\qed
\epf

\bpr \label{prop:ap2:2}
	Assume that condition~\hyperref[cond:m]{({\bf M})} holds, and that $\mu$ and $\{r_n\}$ satisfies \eqref{eq:weight} and \eqref{eq:cond}, respectively, then 
	$$
		n p_n  \int_{\bbr^{d+1}} |\tilde E_1|^2\mu(ds,dt) \cid \int_{\bbr^{d+1}} |Q'(s,t)|^2\mu(ds,dt),
	$$
where $Q'$ is a centered Gaussian process.
\epr

\begin{proof}

Let us first establish the convergence of $\sqrt{np_n}\tilde E_1(s,t)$ for fixed $(s,t)$. Take
\beao
	f_1(\bfx) &=& \text{Re}\left\{ \left(e^{is\|\bfx\|}-\E [e^{is\|\bfx\|}|\|\bfx\|>1]\right) \left(e^{it\bfx/\|\bfx\|} - \E [e^{it\bfx/\|\bfx\|}|\|\bfx\|>1]\right) {\bf1}_{\|\bfx\|>1}\right\} \\
	f_2(\bfx) &=& \text{Im}\left\{ \left(e^{is\|\bfx\|}-\E [e^{is\|\bfx\|}|\|\bfx\|>1]\right) \left(e^{it\bfx/\|\bfx\|} - \E [e^{it\bfx/\|\bfx\|}|\|\bfx\|>1]\right) {\bf1}_{\|\bfx\|>1}\right\}.
\eeao
Then from Lemma~\ref{lemma:fancy},
$$
	\frac{1}{\sqrt{np_n}} (S_n^{(1)},S_n^{(2)})^T = \sqrt{np_n} (\text{Re}\{\tilde E_1(s,t)\}, \text{Im}\{\tilde E_1(s,t)\}) \cid N(\bf0,\Sigma),
$$
where the covariance structure $\Sigma$ can be derived from \eqref{eq:covcal:1} and \eqref{eq:covcal:2}. \revise{This implies that
$$
	\sqrt{np_n} \tilde E_1(s,t) \cid Q'(s,t),
$$
where $Q'(s,t)$ is a zero-mean complex normal process with covariance matrix $\Sigma_{11}+\Sigma_{22}$ and relation matrix $\Sigma_{11}-\Sigma_{22}+i(\Sigma_{12}+\Sigma_{21})$.}

The finite-dimensional distributional convergence of $\sqrt{n\hat p_n} \tilde{E}_1$ to a $Q'(s,t)$ can be generalized using the Cram\'er-Wold device and we omit the calculation of the covariance structure. {The tightness condition for the functional convergence follows the same arguments in the proof of Proposition~\ref{prop:E1} from \cite{bickel:wichura:1971}, with equality \eqref{eq:iidsum} replaced by a variance calculation of the sum of $\alpha$-mixing components \revise{using the inequality \eqref{eq:alpha:ineq}} and condition \eqref{eq:bdconv} is verified through the same argument. This completes the proof of Proposition~\ref{prop:ap2:2}.} \qed

\end{proof}

\begin{proof}[of Lemma~\ref{lemma:fancy}]
The proof follows from that of Theorem~3.2 in \cite{davis:mikosch:2009}.  Here we outline the sketch of the proof and detail only the parts that differ from their proof.

By the vague convergence in \eqref{eq:mrv:conv}, we have \\
i)
$$
	\frac{1}{p_n} \E \left[f_i \left(\frac{\bfX_0}{r_n}\right)\right]  \,\to\, \int f_i d\mu_0 \quad\text{and} \quad \frac{1}{p_n} \E \left[f_i^2 \left(\frac{\bfX_0}{r_n}\right)\right]  \,\to\, \int f^2_i d\mu_0;
$$
ii)
$$
	\frac{1}{p_n} \var \left[f_i \left(\frac{\bfX_0}{r_n}\right)\right] \,=\, \frac{1}{p_n} \E \left[f_i^2 \left(\frac{\bfX_0}{r_n}\right)\right] -  p_n \left(\frac{1}{p_n}\E \left[f_i \left(\frac{\bfX_0}{r_n}\right)\right]\right)^2 \,\to\, \int f^2_i d\mu_0 = \sigma^2_0(f_i,f_i);
$$	
iii)
\beao
	\frac{1}{p_n} \cov \left[f_i \left(\frac{\bfX_0}{r_n}\right), f_j\left(\frac{\bfX_h}{r_n}\right)\right] 
%	&=& \frac{1}{p_n} \E\left[f_i \left(\frac{\bfX_0}{r_n}\right) f_j\left(\frac{\bfX_h}{r_n}\right)\right] -  p_n \cdot \frac{1}{p_n}\E \left[f_i \left(\frac{\bfX_1}{r_n}\right)\right]\cdot \frac{1}{p_n}\E \left[f_j \left(\frac{\bfX_1}{r_n}\right)\right]\\
	& \to& \int f_i f_j d\mu_h \,=\, \sigma^2_h(f_i,f_j).
\eeao
Let us first consider the marginal convergence of $\frac{1}{\sqrt{np_n}}S_n^{(i)}$ for $i=1,2$.  Without loss of generality, we suppress the dependency on $i$ and set
$$
	Y_{tn} := f\left( \frac{\bfX_t}{r_n}\right) - \E f\left(\frac{\bfX_0}{r_n}\right).
$$
Then 
$$
	\frac{1}{p_n} \var \left[Y_{1n}\right]  \to  \sigma^2_0(f,f) \quad \text{and} \quad
	\frac{1}{p_n} \cov \left(Y_{1n}, Y_{(h+1)n}\right) \to \sigma^2_h(f,f).
$$
We also have the following two results for $|\cov(Y_{1n}, Y_{(h+1)n})|$:
\beam
	&& \lim_{h\to\infty}\limsup_{n\to\infty}\sum_{j=h}^{l_n} \frac{1}{p_n}|\cov(Y_{1n}, Y_{(j+1)n})| \nonumber\\
	&\le& \lim_{h\to\infty}\limsup_{n\to\infty}\sum_{j=h}^{l_n} \frac{1}{p_n}\E\left|f \left(\frac{\bfX_0}{r_n}\right) f\left(\frac{\bfX_j}{r_n}\right) \right| + \sum_{j=h}^{l_n}\frac{1}{p_n}\E\left(\left|f \left(\frac{\bfX_0}{r_n}\right) \right|\right)^2 \nonumber\\
	&\le& \lim_{h\to\infty}\limsup_{n\to\infty}\sum_{j=h}^{l_n}  \frac{c}{p_n}\E\left({\bf1}_{\{\|\bfX_0\|>r_n\}}\right)\left({\bf1}_{\{\|\bfX_j\|>r_n\}}\right) + \sum_{j=h}^{l_n}\frac{c}{p_n}\left(\E{\bf1}_{\{\|\bfX_0\|>r_n\}}\right)^2 \nonumber\\
	&\le& \lim_{h\to\infty}\limsup_{n\to\infty}\sum_{j=h}^{l_n} \frac{c}{p_n}\P\left(\|\bfX_0\|>r_n, \|\bfX_j\|>r_n\right) + cl_np_n\nonumber \\
	&=&0 \label{eq:covsum:2}
\eeam
from condition~\eqref{eq:lm:cond:2},
and
\revise{
%\beam
%	\lim_{n\to\infty}\sum_{j=l_n}^\infty \frac{1}{p_n}|\cov(Y_{1n}, Y_{(j+1)n})|  
%	&\le& \lim_{n\to\infty}\sum_{j=l_n}^\infty\frac{1}{p_n}\left| \cov \left[f \left(\frac{\bfX_0}{r_n}\right), f\left(\frac{\bfX_j}{r_n}\right)\right] \right| \nonumber \\
%	&\le& \lim_{n\to\infty}\sum_{j=l_n}^\infty\frac{c}{p_n}\alpha_j=  0  \label{eq:covsum:1}
%\eeam
\beam
	\lim_{n\to\infty}\sum_{j=l_n}^\infty \frac{1}{p_n}|\cov(Y_{1n}, Y_{(j+1)n})|  
	&\le& \lim_{n\to\infty}\sum_{j=l_n}^\infty\frac{1}{p_n}\left| \cov \left[f \left(\frac{\bfX_0}{r_n}\right), f\left(\frac{\bfX_j}{r_n}\right)\right] \right|\nonumber \\
	&\le& \lim_{n\to\infty}\sum_{j=l_n}^\infty\frac{1}{p_n}\alpha_j^{\delta} \left(\E\left| f \left(\frac{\bfX_0}{r_n}\right)\right|^{2/(1-\delta)}\right)^{1-\delta} \nonumber \\
	&\le& \lim_{n\to\infty}\sum_{j=l_n}^\infty\frac{c}{p_n}\alpha_j^{\delta} \left(\E\left({\bf1}_{\{\|\bfX_0\|>r_n\}}\right)^{2/(1-\delta)}\right)^{1-\delta}\nonumber \\
	&\le& \lim_{n\to\infty}\sum_{j=l_n}^\infty c\alpha_j^\delta p_n^{-\delta} \nonumber \\
	&=& 0 \label{eq:covsum:1}
\eeam
}from condition~\eqref{eq:lm:cond:1}.

We apply the same technique of small/large blocks as used in \cite{davis:mikosch:2009}.  Let $m_n$ and $l_n$ be the sizes of big and small blocks, respectively, where $l_n\ll m_n\ll n$.  Let $I_{kn} = \{(k-1)m_n+1,\ldots,km_n\}$ and $J_{kn}=\{(k-1)m_n+1,\ldots,(k-1)m_n+l_n\}$, $k=1,\ldots, n/m_n$, be the index sets of big and small blocks respectively. Set $\tilde{I}_{kn} = I_{kn}\backslash J_{kn}$, i.e., $\tilde{I}_{kn}$ are the big blocks with the first $l_n$ observations removed. For simplicity,  we set $m_n:=1/p_n$ and assume that the number of big blocks $n/m_n=np_n$ is integer-valued.  The non-integer case can be generalized without additional difficulties.

Denote
$$
	S_n(B) := \sum_{t\in B} Y_{tn},
$$
then
$$
	\sum_{t=1}^n Y_{tn} = S_n(1:n) = \sum_{k=1}^{np_n} S_n(I_{kn}) = \sum_{k=1}^{np_n} S_n(\tilde{I}_{kn}) + \sum_{k=1}^{np_n} S_n(J_{kn}).
$$
Let $\{\tilde{S}_n(\tilde{I}_{kn})\}_{k=1,\ldots,np_n}$ be iid copies of $\tilde{S}_n(\tilde{I}_{1n})$.  To prove the convergence of $\frac{1}{\sqrt{np_n}}S_n(1:n)$, it suffices to show the following:
\beqq \label{eq:alpha:clt:1}
	\text{$\frac{1}{\sqrt{np_n}}\sum_{k=1}^{np_n} \tilde{S}_n(\tilde{I}_{kn})$ and $\frac{1}{\sqrt{np_n}}\sum_{k=1}^{np_n} S_n(\tilde{I}_{kn})$ has the same limiting distribution},
\eeqq
\beqq \label{eq:alpha:clt:3}
	\frac{1}{\sqrt{np_n}}\sum_{k=1}^{np_n} S_n(J_{kn}) \overset{p}\to 0,
\eeqq
and
\beqq \label{eq:alpha:clt:2}
	\frac{1}{\sqrt{np_n}}\sum_{k=1}^{np_n} \tilde{S}_n(\tilde{I}_{kn}) \overset{d}\to N(0,\sigma^2(f,f)).
\eeqq
%We now explore the relationship requirements between $l_n$, $p_n$ and $n$ such that \eqref{eq:alpha:clt:1}, \eqref{eq:alpha:clt:3} and \eqref{eq:alpha:clt:2} hold.

The statement \eqref{eq:alpha:clt:1} \revise{holds} if
\beqq \label{eq:blocksize:1}
	np_n \alpha_{l_n} \to 0, \quad \text{as $n\to\infty$.}
\eeqq
This follows from the same argument in equation~(6.2) in \cite{davis:mikosch:2009}.

For condition \eqref{eq:alpha:clt:3}, it suffices to show that
$$
	\frac{1}{{np_n}}\var\left(\sum_{k=1}^{np_n} S_n(J_{kn})\right) \to 0.
$$ 
Note that
$$
	\frac{1}{{np_n}}\var\left(\sum_{k=1}^{np_n} S_n(J_{kn})\right) \le  \var(S_n(J_{1n}))+ 2 \sum_{h=1}^{np_n-1} (1-\frac{h}{np_n}) |\cov(S_n(J_{1n}),S_n(J_{(h+1)n}))| =:P_1+P_2.
$$
We have
\beao
	\limsup_{n\to\infty} P_1 &=& \limsup_{n\to\infty} \var\left(\sum_{j=1}^{l_n}Y_{jn}\right) \\
	&\le& \limsup_{n\to\infty} l_np_n \left(\frac{\var(Y_{1n})}{p_n} + 2 \sum_{j=1}^{l_n-1}(1-\frac{j}{l_n}) \frac{|\cov(Y_{1n},Y_{(j+1)n})|}{p_n}\right) \\
	&\le& \limsup_{n\to\infty} l_np_n\frac{\var(Y_{1n})}{p_n} + \lim_{h\to\infty}\limsup_{n\to\infty}  2l_np_n \sum_{j=1}^{h-1} \frac{|\cov(Y_{1n},Y_{(j+1)n})|}{p_n} \\
	&& \quad+    \lim_{h\to\infty}\limsup_{n\to\infty}  2l_np_n\sum_{j=h}^{l_n-1} \frac{|\cov(Y_{1n},Y_{(j+1)n})|}{p_n} \\
%	&\le& o(1) + O(l_np_n) + 2 l_np_n \sum_{j=h}^{l_n-1} \frac{|\E f\left( \frac{\bfX_1}{r_n}\right) f\left( \frac{\bfX_{j+1}}{r_n}\right)|}{p_n} \\
%	&\le& 0+0+ \lim_{h\to\infty}\limsup_{n\to\infty}\sum_{j=h}^{l_n-1}\left(\frac{c}{p_n}\P\left(\|\bfX_0\|>r_n, \|\bfX_h\|>r_n\right) + cp_n\right)\\
%	&\le&\frac{c}{p_n} \lim_{h\to\infty}\limsup_{n\to\infty}\sum_{j=h}^{l_n-1}\P\left(\|\bfX_0\|>r_n, \|\bfX_h\|>r_n\right) + \limsup_{n\to\infty} cl_np_n \\
	&=& 0
%	&\le& o(1) +  O(\frac{l_n}{m_np_n^{\delta}})\sum_{h=1}^\infty\alpha_h^\delta, \quad \text{for any $\delta\in(0,1]$},
\eeao
where the last step follows from dominated convergence and \eqref{eq:covsum:2}.
And for the other term,
\revise{
\beao
	P_2 &\le& 2\sum_{h=1}^{np_n-1} \sum_{s\in J_{1n}}\sum_{t \in J_{(h+1)n}}\left| \cov(Y_{sn},Y_{tn})\right| \\
	&\le& 2 \sum_{h=1}^{np_n-1}  l_n \sum_{k=h/p_n-l_n+1}^{h/p_n} \left| \cov(Y_{1n},Y_{(k+1)n})\right| \\
	&\le& 2l_np_n\sum_{k=1/p_n-l_n+1}^{\infty}\frac{\left| \cov(Y_{1n},Y_{(k+1)n})\right| }{p_n}\\
	&\le& 2l_np_n\sum_{k=l_n+1}^{\infty}\frac{\left| \cov(Y_{1n},Y_{(k+1)n})\right| }{p_n} \to 0.
\eeao
}Note that $1/p_n=m_n$ is the size of big blocks $I_{kn}$'s and $1/p_n-l_n+1 = m_n-l_n+1$ is the distance between consecutive small blocks $(J_{kn},J_{(k+1)n})$'s.  The last limit follows from \eqref{eq:covsum:1}.

To finish the proof, we need to establish the central limit theorem in \eqref{eq:alpha:clt:2}.  Note the $\tilde{S}_n(\tilde{I}_{ln})$'s are iid with $\E\tilde{S}_n(\tilde{I}_{ln})=0$.  We now calculate its variance.  Recall that $1/p_n-l_n$ is the size of $\tilde{I}_{1n}$, the big block with small block removed.  Then
\beao
	\var \left(\tilde{S}_n(\tilde{I}_{1n})\right) &=& \var\left(\sum_{j=1}^{1/p_n-l_n}Y_{jn}\right) \\
	&=& (\frac{1}{p_n}-l_n) \var(Y_{jn}) + 2 \sum_{k=1}^{1/p_n-l_n-1}(1/p_n-l_n-k) \cov(Y_{1n},Y_{(k+1)n}) \\
	&=&(\frac{1}{p_n}-l_n) \var(Y_{jn}) + 2 \left(\sum_{k=1}^{h}+ \sum_{k=h+1}^{l_n} + \sum_{k=l_n+1}^{1/p_n-l_n-1}\right)\left(1-\frac{l_n+k}{1/p_n}\right) \frac{1}{p_n}\cov(Y_{1n},Y_{(k+1)n}) \\
	&:=& I_0 + I_1 + I_2 + I_3.
%	&\le& np_n(m_n-l_n) O(\frac1n) + 2 n\sum_{h=1}^{m_n}\frac{1}{np_n} \E|U_{1n}V_{1n}U_{(h+1)n}V_{(h+1)n}|  \\
%	&=& O(np_n) + c \sum_{h=1}^{m_n}\frac{1}{p_n} \E[\mathbf1_{\{R_1>r_n\}}\mathbf1_{\{R_{h+1}>r_n\}}]  \\
%	&=& O(1) + c \left(\sum_{h=1}^{m} + \sum_{h=m+1}^{m_n}\right)\P[R_{h+1}>r_n|R_1>r_n] =: O(1) + c(I_1+I_2).
\eeao
Here
$$
	\lim_{n\to\infty} I_0 =  \lim_{n\to\infty}(1-l_np_n)\frac{1}{p_n} \var(Y_{jn}) = \sigma_0^2(f,f),
$$
and
$$
	\lim_{n\to\infty} I_1 = \lim_{n\to\infty} 2\sum_{k=1}^{h}\left(1-p_n(l_n+k)\right) \frac{\cov(Y_{1n},Y_{(k+1)n})}{p_n} = 2 \sum_{k=1}^h \sigma^2_k(f,f).
$$
We also have
$$
	\lim_{h\to\infty}\limsup_{n\to\infty}|I_2| \le \lim_{h\to\infty}\limsup_{n\to\infty}\sum_{k=h+1}^{l_n}\frac{|\cov(Y_{1n},Y_{(k+1)n})|}{p_n} =0
$$
from \eqref{eq:covsum:2},
and
$$
	\lim_{n\to\infty}|I_3| \le \lim_{n\to\infty}\sum_{k=l_n}^{\infty}\frac{|\cov(Y_{1n},Y_{(k+1)n})|}{p_n} =0
$$
from \eqref{eq:covsum:1}. 
Therefore
$$
	\lim_{n\to\infty}\var \left(\tilde{S}_n(\tilde{I}_{1n})\right) = \lim_{n\to\infty} I_0 + \lim_{h\to\infty}\lim_{n\to\infty} I_1  = \sigma_0^2(f,f) + 2 \sum_{k=1}^\infty \sigma^2_k(f,f) =: \sigma^2(f,f)
$$
as defined.  \revise{To show that this infinite sum converges, it suffices to show that 
$$
	\sum_{h=1}^\infty \mu_h(\left\{(\bfx,\bfx')|\|\bfx\|>1,\|\bfx'\|>1\right\}) < \infty.
$$
This follows from \eqref{eq:lm:cond:2} in condition~\hyperref[cond:m]{({\bf M})}, for if
$$
	\sum_{h=1}^\infty \mu_h(\left\{(\bfx,\bfx')|\|\bfx\|>1,\|\bfx'\|>1\right\}) = \infty,
$$
then
\beao
	 \limsup_{n\to\infty} \frac{1}{p_n} \sum_{j=h}^{l_n} \P(\|\bfX_0\|>r_n, \|\bfX_j\|>r_n) 
	 &\ge&  \liminf_{n\to\infty} \sum_{j=h}^{l_n} \P(\|\bfX_0\|>r_n, \|\bfX_j\|>r_n | \|\bfX_0\|>r_n) \\
	 &\ge&  \sum_{j=h}^\infty \mu_j(\left\{(\bfx,\bfx')|\|\bfx\|>1,\|\bfx'\|>1\right\}) = \infty,
\eeao
which leads to a contradiction.}

To apply the central limit theorem, we verify the \revise{Lindeberg}'s condition,
\beao
	 \E \left[(\tilde{S}_n(\tilde{I}_{1n}))^2 \mathbf1_{\{|\tilde{S}_n(\tilde{I}_{1n})|>\epsilon\sqrt{np_n}\}} \right] 
	 &\le&  \E \left[\left(\sum_{j=1}^{1/p_n-l_n}Y_{jn}\right)^2 \mathbf1_{\{|\tilde{S}_n(\tilde{I}_{1n})|>\epsilon\sqrt{np_n}\}} \right] \\
	&\le& \E \left[c(1/p_n-l_n)^2 \mathbf1_{\{|\tilde{S}_n(\tilde{I}_{1n})|>\epsilon\sqrt{np_n}\}} \right] \\
	&\le& c \frac{1}{p_n^2} \P\left[{|\tilde{S}_n(\tilde{I}_{1n})|>\epsilon\sqrt{np_n}}\right] \\
	&\le& c\frac{1}{p_n^2} \frac{\var\left[\tilde{S}_n(\tilde{I}_{1n})\right]}{\epsilon^2np_n}  \,=\,O(\frac{1}{np_n^3}) \to 0.
\eeao
This completes the proof for the convergence of $\frac{1}{\sqrt{np_n}}S_n(1:n)$.  

The joint convergence of $\frac{1}{\sqrt{np_n}} (S_n^{(1)},S_n^{(2)})^T$ follows from the same line of argument together with the Cr\'amer-Wold device.  In particular,
$$
	\frac{1}{np_n} \cov\left(S_n^{(i)},S_n^{(j)}\right) = \sigma^2(f_i,f_j), \quad i,j=1,2.
$$
This completes the proof of the lemma. \qed
\end{proof}

\begin{remark}
Lemma~\ref{lemma:fancy} itself is a more general result of independent interest.  The result can be generalized for functions $f_i$ defined on $\overline{\bbr}^{d}\backslash \{\bf0\}$ with compact support.  In this case, condition~\eqref{eq:lm:cond:2} should be modified to
$$
	\lim_{h\to\infty} \limsup_{n\to\infty} \frac{1}{p_n} \sum_{j=h}^{l_n} \P(\|\bfX_0\|>\epsilon r_n, \|\bfX_j\|>\epsilon r_n) =0
$$
for some $\epsilon>0$, where $\text{support}(f)\subseteq\overline{\bbr}^{d}\backslash B_\epsilon(\bf0)$.  Also, as seen during the proof of the lemma, the conditions on $p_n$, $l_n$, and $\alpha_t$ can be further relaxed.
\end{remark}

%---------------------------------------------------------------------------------------------------------------------------------------------%
%---------------------------------------------------------------------------------------------------------------------------------------------%
\end{document}